\documentclass[11pt,letterpaper,dvipsnames]{amsart}

\usepackage[letterpaper,margin=1.4in,footskip=0.8in]{geometry}
\usepackage{microtype}
\usepackage{amsmath,amssymb,amsthm}
\usepackage{mathtools}
\usepackage[color=cyan!30]{todonotes}
\usepackage{xcolor}
\usepackage{enumitem}
\setlist{noitemsep}
\usepackage{mathrsfs}
\usepackage{tikz-cd}
\usepackage{array}
\usepackage[pdfusetitle,colorlinks]{hyperref}
\hypersetup{citecolor=black,linkcolor=black,urlcolor=black}
\usepackage[capitalise,noabbrev]{cleveref}
\crefformat{equation}{\ensuremath{(#2#1#3)}}
\crefmultiformat{equation}{\ensuremath{(#2#1#3)}}{ and~\ensuremath{(#2#1#3)}}{, \ensuremath{(#2#1#3)}}{, and~\ensuremath{(#2#1#3)}}
\numberwithin{figure}{section}
\numberwithin{equation}{section}
\newtheorem{theorem}[equation]{Theorem}
\newcommand{\uu}{\mathrm{U}}
\newtheorem{lemma}[equation]{Lemma}
\newtheorem{corollary}[equation]{Corollary}
\newtheorem{proposition}[equation]{Proposition}
\theoremstyle{definition}
\newtheorem{definition}[equation]{Definition}
\newtheorem{notation}[equation]{Notation}
\theoremstyle{definition}
\newtheorem{remark}[equation]{Remark}
\theoremstyle{definition}
\newtheorem{example}[equation]{Example}
\theoremstyle{definition}
\newtheorem{construction}[equation]{Construction}

\newtheorem{maintheorem}{Theorem}

\theoremstyle{definition}
\newtheorem{question}[equation]{Question}

\AddToHook{env/definition/begin}{\crefalias{equation}{definition}}
\AddToHook{env/question/begin}{\crefalias{equation}{question}}
\AddToHook{env/theorem/begin}{\crefalias{equation}{theorem}}
\AddToHook{env/proposition/begin}{\crefalias{equation}{proposition}}
\AddToHook{env/corollary/begin}{\crefalias{equation}{corollary}}
\AddToHook{env/observation/begin}{\crefalias{equation}{observation}}
\AddToHook{env/construction/begin}{\crefalias{equation}{construction}}
\AddToHook{env/example/begin}{\crefalias{equation}{example}}
\AddToHook{env/notation/begin}{\crefalias{equation}{notation}}
\AddToHook{env/lemma/begin}{\crefalias{equation}{lemma}}
\AddToHook{env/remark/begin}{\crefalias{equation}{remark}}
\AddToHook{env/question/begin}{\crefalias{equation}{question}}
\AddToHook{env/conjecture/begin}{\crefalias{equation}{conjecture}}

\usepackage[all]{xy}
\usepackage{relsize}
\usepackage[bbgreekl]{mathbbol}
\usepackage{amsfonts}
\DeclareSymbolFontAlphabet{\mathbb}{AMSb} 
\DeclareSymbolFontAlphabet{\mathbbl}{bbold}
\newcommand{\Prism}{{\mathlarger{\mathbbl{\Delta}}}}

\usepackage{lipsum}
\usepackage{fancyhdr}
\mathchardef\mhyphen="2D

\DeclareMathOperator{\Fun}{Fun}

\DeclareMathOperator{\Hom}{Hom}

\DeclareMathOperator{\Map}{Map}

\newcommand{\op}{\mathrm{op}}
\newcommand{\U}{\mathrm{U}}

\DeclareMathOperator{\Spec}{Spec}

\DeclareSymbolFontAlphabet{\mathbb}{AMSb} 
\DeclareSymbolFontAlphabet{\mathbbl}{bbold}

\newcommand{\cC}{\mathcal{C}}

\newcommand{\cO}{\mathscr{O}}

\newcommand{\cH}{\mathcal{H}}

\newcommand{\id}{\mathrm{id}}

\newcommand{\Sp}{\mathrm{Sp}}

\def\on{\operatorname}
\newcommand{\perf}{\mathrm{perf}}
\newcommand{\bZ}{\mathbb{Z}}
\newcommand{\bF}{\mathbb{F}}
\newcommand{\bG}{\mathbb{G}}
\newcommand{\bQ}{\mathbb{Q}}
\newcommand{\bU}{\mathbb{U}}
\newcommand{\colim}{\on{colim}}
\newcommand{\Mod}{\on{Mod}}

\newcommand{\uni}{\mathrm{uni}}
\newcommand{\St}{\mathrm{St}}
\newcommand{\Alg}{\mathrm{Alg}}
\newcommand{\Aff}{\mathrm{Aff}}
\newcommand{\ft}{\mathrm{qft}}

\newcommand\restr[2]{{\left.\kern-\nulldelimiterspace#1\vphantom{\big|}\right|_{#2}}}
\newcommand{\suchthat}{\;\ifnum\currentgrouptype=16 \middle\fi\vert\;}

\newcolumntype{C}[1]{>{\centering\arraybackslash}p{#1}}

\numberwithin{figure}{subsection}
\numberwithin{equation}{subsection}

\title{Artin--Mazur formal groups and Milne duality via unipotent spectra}

\author{Shubhodip Mondal}
\author{Tasos Moulinos}
\author{Lucy Yang}
\address[Shubhodip Mondal]{Purdue University}
\email{mondalsh@purdue.edu}
\address[Tasos Moulinos]{CNRS, Université Paris-Saclay}
\email{tasos.moulinos@universite-paris-saclay.fr}
\address[Lucy Yang]{Columbia University}\email{ly2620@columbia.edu}

\begin{document}

\begin{abstract}
We introduce and develop the notion of ``unipotent spectra." This is defined to be the stabilization of To\"en's category of affine stacks, and is related to recent work of Mondal--Reinecke. 
Unipotent spectra give rise to unipotent stable homotopy groups and unipotent homology, which are new invariants for schemes valued in unipotent group schemes. 
As applications, we recover the Artin--Mazur formal groups associated to schemes without any vanishing assumptions. Further, we show that syntomic cohomology admits a natural refinement to a perfect unipotent spectrum. 
Finally, we extend Milne's work on arithmetic duality theorems to the category of perfect unipotent spectra and apply it to refine Poincar\'e duality in syntomic cohomology.
\end{abstract}

\maketitle
\markleft{\tiny{S. MONDAL, T. MOULINOS AND L. YANG}}
\markright{\tiny{ARTIN–MAZUR FORMAL GROUPS \& MILNE DUALITY VIA UNIPOTENT SPECTRA}}
\vspace{-13mm}
\tableofcontents
\newpage
\section{Introduction}
\subsection{Motivation \texorpdfstring{\&}{&} context}
In \cite{MR457458}, Artin and Mazur attached certain formal groups to algebraic varieties. More precisely, for a smooth proper scheme $X$ over a perfect field $k$ of characteristic $p>0$, they attached certain formal groups $\Phi^r(X)$ for $r \in \mathbb{N}$. In \cite[Question~(a)]{MR457458}, they raised the question of constructing an object in some \emph{derived category}, which would be finer than the collection of $\Phi^r(X)$ for $r \in \mathbb{N}$. As they pointed out, the construction of such an object would be quite subtle as one would have to extend (or bypass) the work of Cartier on the theory of formal groups, on which their work was based.

In a different vein, in \cite{Milne111}, Milne extended Poincaré duality from étale to syntomic cohomology of smooth proper schemes over a perfect field $k$ of characteristic $p$. A key insight was that both finite groups and vector spaces over the ground field appear in cohomology, so any such duality should simultaneously incorporate Pontryagin duality for finite groups and linear duality over the base field $k$. By upgrading syntomic cohomology to a functor landing in perfect unipotent group schemes, Milne was able to establish such a setup.

In \cite{soon}, a notion of unipotent homotopy group schemes was used to reconstruct the Artin--Mazur formal groups under certain strong vanishing assumptions. However, the general situation was not addressed in \cite{soon}. The notion of unipotent homotopy type in \cite{soon} of a scheme is based on To\"en's work on affine stacks. In view of the representability results for affine stacks in \cite{MR2244263}, it is also natural to wonder whether syntomic cohomology of smooth proper $k$-schemes can be studied using this formalism. In this paper, both of these questions will be addressed by developing a framework of unipotent \emph{stable} homotopy theory.

With the above motivations in mind, we introduce the stable $\infty$-category of \emph{unipotent spectra}, which we propose as the categorical home for both Artin–Mazur formal groups and Milne’s duality results. Our key definition is the following:

\begin{definition}[Unipotent spectra] Let $\mathrm{AffSt}_{A*}$ denote the $\infty$-category of pointed affine stacks over a commutative ring $A$. Note that $\mathrm{AffSt}_{A*}$ is naturally equipped with an endofunctor $\Omega$ determined by sending $X \mapsto * \times_X *. $ We define $\mathrm{Sp}^{\mathrm{U}}_A$ to be the inverse limit of the tower of $\mathbb{Z}$-indexed $\infty$-categories 
\begin{equation*}
 \ldots \to \mathrm{AffSt}_{A*}\xrightarrow{\Omega} \mathrm{AffSt}_{A*}\to\ldots .
\end{equation*}
We call the stable $\infty$-category $\mathrm{Sp}^{\mathrm{U}}_A$ the category of unipotent spectra over $A$.
\end{definition}

\begin{remark}
Our terminology is motivated by \cite{soon}, where Mondal--Reinecke developed the notion of unipotent homotopy theory: to every scheme $X$ over $k$, the authors attach an affine stack $\mathbb{U}(X)$ which is called the unipotent homotopy type of $X$. When $X$ is pointed and cohomologically connected, this allows one to consider $\pi_n(\mathbb{U}(X))$, which is called the unipotent homotopy group scheme and denoted by $\pi_n^{\mathrm{U}}(X)$. Under our assumptions, $\pi_n^{\mathrm{U}}(X)$ is a unipotent affine group scheme (possibly of infinite type). 
\end{remark}

\begin{remark}[Unipotent stable homotopy type]
    Any stack $ Y $ over $ k $ has a unipotent stable homotopy type $ \Sigma^\infty_+ Y$ (see \cref{defn:unipotent_stable_htpy_type}), which is a unipotent spectrum. For $n \in \mathbb{Z}$, the object $\pi_n(\Sigma^\infty_+ Y)$ is representable by a unipotent group scheme, which we will call the $n$th unipotent stable homotopy group scheme.

\end{remark}
We give some examples of unipotent spectra below.

\begin{example}
 Given any spectrum $E \in \mathrm{Sp}$ in the usual sense, there is a canonical way to attach a unipotent spectrum $E^{\uu} \in \mathrm{Sp}^{\on{U}}_k,$ which can be regarded as a ``unipotent completion" of $E$. See \cref{unicomofspectra}.  
\end{example}

 \begin{example}[{\cref{introex1}}]\label{examintro}
Let $G$ be a commutative unipotent affine group scheme over a field $k$. 
Then there is an Eilenberg–MacLane unipotent spectrum $G \in \mathrm{Sp}^{\U}_{k}$ over $ k $ associated to $ G $; its $ n $th space is the affine stack $ B^nG := K(G,n) $. 
\end{example}

\begin{example}
As a particular example of the above, let $\mathbb{H}$ be the fixed points of the Frobenius endomorphism on the $p$-typical Witt vector ring scheme $W$. This is a unipotent group scheme, which arises as the Cartier dual to the multiplicative formal group  $\widehat{\bG}_m$. The classifying stack $B \mathbb{H}$ arises as the generic fiber of the filtered circle studied in \cite{moulinos2019universal} and is an affine stack. The sequence of affine stacks obtained by taking further deloopings  $\left \{B^n \mathbb{H} \right \}_{n \ge 0}$ defines a  unipotent spectrum over $k$, denoted by $\mathbb{H}$, which plays the role of the unipotent completion of the Eilenberg-MacLane spectrum $ H\mathbb{Z} $. 
\end{example}

From a homotopy theoretic point of view, the role played by unipotent spectra can be summarized by the following table, explaining the analogy with usual spectra in a broader context.
\vspace{2mm}
\begin{center}

\bgroup
\def\arraystretch{1.5}
\begin{tabular}{ | C{15em} | C{15em}| } 
\hline
\textbf{Usual homotopy theory} & \textbf{Unipotent homotopy theory} \\ 

\hline
\text{Spaces} & \text{Affine stacks} \\ 
\hline
\text{(Homotopy) groups} & Unipotent (homotopy) group schemes \\
\hline
\text{Spectra} & \text{Unipotent spectra}\\
\hline
\text{Chain complexes} & $\mathbb{Z}$\text{-modules in unipotent spectra}\\
\hline

\hline
\end{tabular}
\egroup
\end{center}
\vspace{2mm}
The notion of unipotent spectra also has other applications in the context of $p$-adic cohomology theories of varieties over fields of characteristic $p$, as we indicate below.
\begin{remark}[Constructible sheaves and unipotent spectra]
 Let $X = \mathrm{Spec}\, A$ for a regular $\mathbb{F}_p$-algebra $A$. Using a form of the Riemann--Hilbert correspondence, one can show that the derived category of constructible $\mathbb{F}_p$-vector spaces $D^b_{\mathrm{cons}}(X_{\mathrm{\acute{e}t}}, \mathbb{F}_p)$ embeds inside $\mathbb{F}_p$-modules in $\mathrm{Sp}^\U_A$. However, the latter is much larger than $D^b_{\mathrm{cons}}(X_{\mathrm{\acute{e}t}}, \mathbb{F}_p)$. For example, $\mathbb{G}_a$ (or its perfection $\mathbb{G}_a^{\mathrm{perf}}$) viewed as $\mathbb{F}_p$-module in $\mathrm{Sp}^\U_A$ does not lie in the essential image of the embedding of $D^b_{\mathrm{cons}}(X_{\mathrm{\acute{e}t}}, \mathbb{F}_p)$.
\end{remark}{}

Let us return for a moment to our goal of  addressing \cite[Question~(a)]{MR457458} due to Artin--Mazur, and mention some recent developments related to this story for additional context. 
In \cite{BOL}, Bragg--Olsson proved that a suitable sheafification of the Artin--Mazur formal groups are always pro-representable, a result previously obtained by Raynaud using different methods \cite{MR563468}. 
We denote the representing pro-object by $\Phi^n(X)^{\mathrm{fl}}$ and refer to it as the $n$th \emph{flat} Artin--Mazur formal group. In the context of Artin--Mazur formal groups, Mondal--Reinecke prove the following result:

\begin{theorem}[{\cite[Theorem~1.0.9]{soon}}]\label{mr1}Let $n \ge 1$ be an integer.
Let $X$ be a pointed proper scheme over an algebraically closed field $k$ of characteristic $p>0$ satisfying \begin{equation}\label{conditionintro}
H^0 (X, \mathcal O) \simeq k, \quad H^i(X,\mathcal O) = 0 \text{ for all } 0 < i < n, \quad \text{and} \quad H^{n+1}(X, \mathcal O)=0.
\end{equation} Let $\Phi^n(X)$ denote the $n$-th Artin--Mazur formal group defined in this context. 
Then if $n > 1$, $\Phi^n(X)$ is naturally isomorphic to the Cartier dual $\pi^{\mathrm{U}}_n(X)^\vee$ of the $n$-th unipotent homotopy group scheme of $X$. 
If $n=1$, $\Phi^n(X)$ is naturally isomorphic to $(\pi_1^{\mathrm{U}}(X)^{\mathrm{ab}})^{\vee}.$
\end{theorem}
 
The work in our paper is partly inspired by the above homotopy theoretic reconstruction of Artin--Mazur formal groups. In view of \cref{mr1}, the authors in \cite{soon} proposed the heuristic that the theory of Artin--Mazur formal groups could be viewed as a notion of \emph{homology} theory for unipotent homotopy theory. In our paper, we will make this precise and work within the framework of unipotent spectra to reconstruct the flat Artin--Mazur formal groups in general without any cohomology vanishing assumptions such as \cref{conditionintro} above. As we will see, this requires a significant amount of additional work and the idea of using the ``coniveau filtration."

\vspace{1mm}
\subsection{Main theorems}
Our first aim is to develop the foundations of unipotent spectra and establish several general results that closely reflect the category of usual spectra. 
Namely, we prove the following results.

\begin{theorem} Let $ k $ be a field.    
\begin{enumerate}
    \item The category of bounded below unipotent spectra over $k$ admits a natural \emph{$t$-structure} whose heart is equivalent to the abelian category of commutative affine unipotent group schemes over $k$ (\cref{prc}). 
    \item The $ \infty $-category of \emph{ind}-unipotent spectra is equipped with a natural symmetric monoidal structure that preserves small colimits separately in each variable (see \cref{ind-unispect}).

     \item The $\infty$-category of bounded below unipotent spectra over $ k $ embeds fully faithfully in the category of modules over a certain $\mathbb{E}_1$-algebra in spectra given by the endomorphism spectrum of $\mathbb{G}_a$ (see \cref{recog}). 
     \item Let $X$ be a pointed stack over $k$. The homotopy group schemes of the unipotent spectrum $\Sigma^\infty X$ recover the unipotent stable homotopy groups that can be defined using the Freudenthal suspension theorem for affine stacks \cite[Proposition~3.4.10]{soon} (see \cref{specfreudenthalcomp}). Namely, we have 
    $$\pi_i (\Sigma^\infty X) \simeq \varinjlim_{k}\pi_{i+k} ((\mathbb{U} \Sigma)^k (X)).$$
\end{enumerate}{}
\end{theorem}

With a view towards application to Artin--Mazur formal groups, we study the $\mathbb{Z}$-linearization of unipotent spectra, leading to the notion of unipotent homology.
\begin{definition} [Unipotent homology]
    Let $ k $ be a field and $ X $ a stack over $ k $. 
    We define the \emph{unipotent homology} $H^{\uu}_* (Y):= \Sigma^\infty_+ Y \otimes \mathbb{Z}$ of $ X $ to be the $ \mathbb{Z} $-linearization of $ \Sigma^\infty_+ X $. Let $ Y $ be a finite-dimensional scheme over $ k $ and $y \in Y$ be a point. In \cref{defn:unipotent_local_homology}, we introduce a \emph{local variant of unipotent homology} which we denote by $ H^\uu_{*, y}(Y_y)$. We denote $\pi_i(H^{\uu}_* (Y))$ (resp.~$\pi_i( H^\uu_{*, y}(Y_y)$)) by $H^\uu_*(Y)$ (resp.~$H^\uu_{i,y}(Y_y)$).
\end{definition}
We prove the following profiniteness result for unipotent homology group schemes, which can be viewed as an analogue of \cite[Theorem~1.0.5]{soon}.

\begin{theorem}
    Let $X$ be a stack over a field $k$ of characteristic $p$ such that $H^i (X, \mathcal{O})$ is a torsion $k_\sigma [F]$-module for each $i \ge 0.$ 
    Then $H^\uu_{i}(X)$ is a profinite unipotent commutative group scheme for each $i \ge 0 $ (see \cref{profinitenessthm}).  

\end{theorem}

Next, we equip the unipotent homology $H^{\uu}_* (X)$ of a scheme $X$ with a coniveau filtration--which we denote by $F^*H^\uu_*(X)$--following work of To\"{e}n (see \cref{filtrationdeff}). 
We show in \cref{descriptionofgradedpiece} that the graded pieces of this filtration can be described as 
 \begin{equation}\label{congraded}    \mathrm{gr}^i H^\uu_*(X) \simeq \prod_{x \in X^{(i)}} H^\uu_{*, x}(X_x),  
 \end{equation}
where $X^{(i)}$ denotes the set of points of $X$ of codimension $i$. 
Now the coniveau filtration gives rise to the following ``coniveau spectral sequence" (see \cref{congraded})
$$ E_1^{i,j} = \prod_{x \in X^{(i)}} H_{i+j,x}^\uu (X_x) \implies H_{i+j}^\uu(X) $$ 
whose $ E_1 $-page consists of unipotent group schemes. We prove the following result generalizing the work of To\"en for smooth schemes which relies on certain purity results.  
\begin{theorem} [{\cref{filtrationuse}}]
    Let $ X $ be a finite-dimensional Cohen–Macaulay scheme over $ k $. 
    Then its unipotent homology $H^\uu_*(X)$, equipped with the coniveau filtration lies in the connective part of the Beilinson $t$-structure in the stable $\infty$-category of $\mathbb{Z}$-module objects in unipotent spectra over $ k $. 
\end{theorem}

\begin{definition}
    One can define $J_*^\uu(X):= \tau_{\ge 0}^\mathrm{B} (F^* H^\uu_* (X))$, which has the natural structure of a chain complex of commutative unipotent group schemes, since it lies in the heart of the Beilinson $t$-structure. See \cref{ntn:filtered_generalized_loc_jacobian}.
\end{definition}
Using $J_*^\uu(X)$, we prove the following result regarding cohomology with coefficients in a commutative unipotent group scheme, which generalizes a result of To\"{e}n in the smooth case \cite[Proposition 3.7]{Toee23}.
\begin{theorem}[{cf.~\cref{animpresult}}]
 Let $X$ be a Cohen--Macaulay scheme over a field $k$. For any commutative unipotent group scheme $G$ over $k$, we have an isomorphism 
$$
R\mathrm{Hom}_{D(\mathrm{Uni})}(J_*^\uu(X), G ) \xrightarrow{\sim} R\Gamma (X, G) .$$Here, $D(\mathrm{Uni})$ denotes the derived category of the abelian category of unipotent commutative group schemes over $k.$
\end{theorem}{}
Now, let $X$ be a smooth proper scheme over $k.$ 
Let $(\Phi^n_X)^\mathrm{fl}$ denote the sheafification of the functor $\Phi^n_X$ defined by Artin--Mazur (see \cref{artinmazur}) for the fppf topology on $\mathrm{Art}_k^\mathrm{op}.$ Then Bragg--Olsson proved that $(\Phi^n_X)^\mathrm{fl}$ is pro-representable for every $n$. 
The following result generalizes \cref{mr1} without any vanishing assumptions and recover the Artin--Mazur formal groups in general; this addresses \cite[Question~(a)]{MR457458} due to Artin--Mazur. 
\begin{maintheorem}[{\cref{isitmainthm}}]\label{introartinmaz}
  Let $X$ be a smooth proper scheme over a perfect field $k$ of characteristic $p>0$. Then for all $i \ge 0$, the Cartier dual of the flat Artin--Mazur formal group $(\Phi_X^i)^\mathrm{fl}$ is canonically isomorphic to the unipotent group scheme $E_2^{i,0}$, arising in the second page of the coniveau spectral sequence.
\end{maintheorem}
One may compare \cref{introartinmaz} to Bloch and Ogus's description of the $E_2$-page of the coniveau spectral sequence in certain cohomology theories \cite{bloch-ogus}. 
This also raises the following question which is not pursued in our paper.

\begin{question}
 Is there a classical description of the unipotent group schemes $E^{i,j}_2$ for $j>0 $ arising from the coniveau spectral sequence on unipotent homology of a smooth proper scheme?   
\end{question}

Let us now return to our other primary motivation: providing a natural framework for Milne's duality theorems. In order to do this, in \cref{section5}, we introduce the notion of perfect unipotent spectrum over a perfect field $k$ of characteristic $p>0$; this is defined to be a spectrum object in the category of perfect affine stacks (see \cref{perfectunispec}). In view of the equivalence between affine stacks and coconnective derived rings, the category of perfect affine stacks corresponds to the subcategory of coconnective derived rings on which the Frobenius map is an isomorphism. This implies that for a unipotent spectrum to be perfect is a property, as opposed to any additional structure. 

Now, similarly to \cref{examintro}, any perfect, unipotent, commutative affine group scheme can be viewed as an perfect unipotent spectrum. In \cref{perfunispec}, we isolate a class of perfect unipotent spectra whose homotopy group schemes are perfect, unipotent group schemes of quasi-finite type (see \cref{subsection:perfect_qft_group_schemes}); such objects are called quasi-finite type perfect unipotent spectra. We show that there is a good theory of duality for such objects, which extends Milne's duality \cite{Milne111}.

\begin{maintheorem}\label{introduality}
Let $k$ be a perfect field of characteristic $p$. 

\begin{enumerate}
 \item   ({\cref{duality}}) Let $(\bF_p- \Mod^{\U, \mathrm{perf}, \mathrm{ft}}_{k})^{\on{bd}} $ denote the category of quasi-finite type perfect unipotent $\bF_p$-modules over $k$ which are bounded with respect to the $t$-structure on unipotent spectra. Then the functor 
\[
R\underline{\Hom}( - , \bZ / p): (\bF_p-\Mod(\on{St}_k))^{\op} \to \bF_p-\Mod(\on{St}_k)
\]
restricts to an autoduality of $(\bF_p- \Mod^{\U, \mathrm{perf}, \mathrm{ft}}_{k})^{\on{bd}}$.   
\item  (\cref{duality Q/Z}) Let $(\bZ- \Mod^{\U, \mathrm{perf}, \mathrm{ft}}_{k})^{\on{bd}} $  denote the category of quasi-finite type perfect unipotent $\bZ$-modules over $k$ which are bounded with respect to the $t$-structure on unipotent spectra. Then the functor 
\[
R\underline{\Hom}( - , \bQ_p / \bZ_p): (\bZ-\Mod(\on{St}_k))^{\op} \to \bZ-\Mod(\on{St}_k)
\]
restricts to an autoduality of $(\bZ- \Mod^{\U, \mathrm{perf}, \mathrm{ft}}_{k})^{\on{bd}}$. 
\end{enumerate}
\end{maintheorem}

Finally, we apply these ideas to syntomic cohomology, cf. \cite{BMS2, IR}, of proper varieties over $k$. 
Namely, we show the following.

\begin{maintheorem}[see~\cref{subsection:syntomic_duality}] \label{maintheoremsyntomic}
 Let $X$ be a proper lci scheme of dimension $d$ over a perfect field $k$ of characteristic $p>0$ and $ i \in \mathbb{Z}$. Then the functor determined by
\[
\on{Sch}^{\on{perf}}_{k} \ni S \mapsto R \Gamma_{\on{Syn}}(X \times S, \bZ/p^n(i))  \in D(\mathbb{Z})
\]
 is represented by a perfect unipotent spectrum over $k$, which we denote by $ \bZ/p^n(i)^{\uni}_{X}$. Further, if $X$ is additionally assumed to be smooth, $\bZ/p^n(i)^{\uni}_{X}$ is of quasi-finite type and there is a natural isomorphism 
$$ \mathbb{Z}/p^n(i)^{\mathrm{uni}}_X \simeq ({\mathbb{Z}/p^n(d-i)^{\mathrm{uni}}_X})^\vee [-2d]$$ of perfect unipotent spectra, where the right hand side uses the notion of duality from \cref{introduality}.
\end{maintheorem}{}

\begin{remark}
In \cref{finalsection} we extend the equivalence of  \cref{maintheoremsyntomic} to the $p$-complete setting. Namely, we define full subcategory $\cC^{\on{pro-qft}}$ of $p$-complete unipotent $\bZ$-modules consisiting of \emph{pro}-quasi-finite objects, together with an involutive equivalence $\mathbb{D}: \cC^{\on{pro-}\on{qft}} \to  (\cC^{\on{pro-}\on{qft}})^\op $.  By definition, the functor determined by 
\[
\on{Sch}^{\on{perf}}_{k} \ni S \mapsto R \Gamma_{\on{Syn}}(X \times S, \bZ_p(i)) \in D(\mathbb{Z})
\]
is representable in this category, allowing us to extend the equivalence of \cref{maintheoremsyntomic} beyond the $p^n$ torsion case. 
\end{remark}

Our starting point for developing the $\infty$-category of perfect unipotent spectra was Breen's results in \cite{breenperf} on the vanishing of higher Ext groups of $\bG_a$ in the category of $\bF_p$-module sheaves over the perfect site. Due to this, the study of perfect unipotent $\infty$-category of perfect unipotent $\bZ$ and $\bF_p$ modules becomes much more tractable. Indeed, the Artin-Schreier sequence 
\[
 0 \to \bZ/p \to \bG_a \xrightarrow{F-1} \bG_a \to 0,
\]
allows us to control the behavior of the functor $R\underline{\Hom}_{\bF_p}(-, \bZ/p)$. Given Milne's work, one may hope that this functor restricts to a duality on some subcategory of perfect unipotent modules. This led us to the notion of a quasi-finite spectrum, which is exactly the conditions needed to get the duality. The latter notion is formulated using perfect quasi-finite type groups schemes, which has antecedents in the literature. Indeed, over an algebraically closed field, an equivalent notion was introduced by Serre in \cite{MR118722} under the name \emph{quasi-algebraic group}.  Artin, in \cite{Artin-supersingular}, following ideas of Grothendieck, conjectured a duality in which $\bQ_p/\bZ_p$ played the role as a dualizing (ind)-object in some derived category of quasi-algebraic quasi-unipotent groups, in which the flat cohomology of a surface is representable. This was of course realized by Milne's work in \cite{Milne111}; as we show, these phenomena all naturally live in our world of unipotent modules.

\subsection{Outline}
In \cref{unispectra}, we develop the foundations of unipotent spectra. In \cref{unispec2.1}, we introduce the definition of unipotent spectra and prove the existence of a certain $t$-structure. In \cref{monoidal}, we construct and prove the existence of a natural symmetric monoidal structure on the category of ind-unipotent spectra. In \cref{applyschemes}, we begin applying our constructions to schemes. Namely, we discuss unipotent stable homotopy types of schemes (\cref{defn:unipotent_stable_htpy_type}) and prove a result (\cref{specfreudenthalcomp}) relating unipotent stable homotopy groups with the unipotent homotopy group schemes studied in \cite{soon}. Then we discuss unipotent homology in \cref{def2.49} and prove the profiniteness theorem (\cref{profinitenessthm}). We also introduce a local variant of unipotent homology (\cref{defn:unipotent_local_homology}), which plays an important role in \cref{section3}. In \cref{subsection:recognition_thm}, we prove the recognition theorem for unipotent spectra.

In \cref{section3}, we discuss the coniveau filtration on unipotent homology. In \cref{descriptionofgradedpiece}, we describe the graded pieces of this filtration in terms of unipotent local homology. We use this to deduce a certain purity property for Cohen--Macaulay schemes in \cref{connect}. In \cref{sec3.2}, we reformulate the latter result using the language of Beilinson t-structures. In \cref{sec3.3} we apply this to flat cohomology of Cohen--Macaulay schemes with coefficients in unipotent group schemes and prove \cref{animpresult}.

In \cref{section4}, these tools are then applied to the study of Artin--Mazur formal groups, where we prove \cref{isitmainthm} (\cref{introartinmaz}).

In \cref{section5} we introduce perfect unipotent spectra and prove \cref{introduality} and \cref{maintheoremsyntomic}. In \cref{subsection:perfect_qft_group_schemes} we introduce some preliminaries on (perfect) quasi-finite type group schemes.  In \cref{subsection:perfect_unipotent_spectra} we define perfect affine stacks and perfect unipotent spectra. In \cref{subsection:perfect_uni_spectra_recognition} we prove a recognition theorem for perfect unipotent $\bF_p$-modules and $\bZ$-modules, which plays an important role  in the duality theory that we will establish. 
In \cref{subsection:duality_perfect_unipotent_spectra} we prove that linear duality on sheaves of $\bF_p$-module spectra restricts to a duality on the full subcategory of perfect quasi-finite type unipotent $\bF_p$-modules. In \cref{subsection:syntomic_duality} we show that $\mod p$-syntomic cohomology (for any given weight) admits a refinement to a perfect unipotent spectrum and describe how it behaves relative to the aforementioned duality. In \cref{subsection duality in Zmod} we extend the duality to the full-subcategory  perfect quasi-finite type unipotent $\bZ$-modules, and study mod $p^n$ syntomic cohomology. Finally, in    \cref{finalsection} we study this duality in the $p$-complete setting.

\subsection{Notation \texorpdfstring{\&}{&} conventions}
\begin{enumerate}
\item As in \cite{MR2244263} and \cite{luriehigher}, we work with a certain Grothendieck universe (containing the set of natural numbers); to deal with the size-related aspects of certain
constructions, one sometimes needs to choose an enlargement of the Grothendieck universe, which will be kept implicit in our paper, similarly to \cite{luriehigher}. 

    \item We freely use the theory of $\infty$-categories developed in \cite{luriehigher}. 
    We will implicitly regard 1-categories as $ \infty $-categories via the nerve of \S1.1.2 loc. cit. 
    We let $\mathcal{S}$ denote the $\infty$-category of spaces and $\mathrm{Sp}$ denote the $\infty$-category of spectra. For any presentable $\infty$-category $\cC$, we use the notation $\Sp(\cC)$ to denote the stable $\infty$-category of spectrum objects in $\cC$.  For $E \in \on{CAlg}(\Sp)$, we let $E-\Mod(\cC)$ denote the stable $\infty$-category of $E$-module objects in $\Sp(\cC)$. We let $\mathrm{Map}$ denote the mapping space and $\mathrm{RHom}$ denote the mapping spectrum. In the relevant set up, we let $\underline{\mathrm{Map}}$ denote the internal mapping space and $R\underline{\mathrm{Hom}}$ denote the internal mapping spectrum. For $t$-structures, we use the homological convention.

  \item  For a discrete commutative ring $A$, we let $\mathrm{Alg}_A$ denote the category of $A$-algebras (in a certain Grothendieck universe). We let $\mathrm{St}_A$ denote the full subcategory of $\mathrm{Fun}\left(\mathrm{Alg}_A, \mathcal{S}\right)$ on those functors which satisfy descent for the fpqc topology, and call it the category of stacks over $A$. We let $\mathrm{AffSt}_A$ denote the category of affine stacks over $A$ in the sense of \cite{MR2244263}. We use $\mathrm{St}_k^\mathrm{afin}$ to denote almost finitary stacks over a field $k$ (\cref{almostfin}). 

  \item We let $\mathrm{Sp}_A^\mathrm{U}$ denote the category of unipotent spectra in \cref{unispectradef}, and $\mathrm{Sp}_A^{\mathrm{U}-}$ for the category of bounded below unipotent spectra (\cref{notabelow}). We let $\mathrm{Sp}_A^{\mathrm{U}, \mathrm{perf}}$ denote the category of perfect unipotent spectra (\cref{perfectunispec}). For any $\mathbb{E}_\infty$ ring $E$, we let $E-\mathrm{Mod}_A^{\mathrm{U}} \coloneqq E-\mathrm{Mod}(\mathrm{Sp}_A^{\mathrm{U}})$, which is called the category of unipotent $E$-modules over $A$ (\cref{intro-mod-uni}). We let $\mathrm{Sp}_k^{\mathrm{U}, \mathrm{perf}, \mathrm{ft}}$ denote the category of quasi-finite type unipotent spectra over a perfect field $k$ (\cref{perfunispec}). We denote by $E-\mathrm{Mod}_A^{\mathrm{U}, \mathrm{perf}}$ and $E-\mathrm{Mod}_A^{\mathrm{U}, \mathrm{perf}, \mathrm{ft}}$ the category of $E$-module objects in $\mathrm{Sp}_k^{\mathrm{U}, \mathrm{perf}, \mathrm{}}$ and $\mathrm{Sp}_k^{\mathrm{U}, \mathrm{perf}, \mathrm{ft}}$, respectively.
\end{enumerate}{}

\subsection{Acknowledgements}
We would like to thank Ben Antieau, Bhargav Bhatt, Akhil Mathew, Emanuel Reinecke, and Bertrand Toën for helpful conversations about the ideas presented in this work. 
We would also like to thank Ben Antieau for helpful comments on a draft. 

This project was initiated while the first and second authors were in residence at the Institute for Advanced Study. During the preparation of this paper, S.M. was supported by the University of British Columbia and a start up grant from Purdue University. The authors wish to thank Columbia University for its hospitality. 
\newpage
\section{Unipotent spectra}\label{unispectra}

\subsection{Generalities on unipotent spectra}\label{unispec2.1}
Let $A$ be a fixed ordinary commutative ring. We start by recalling the definition of affine stacks due to \cite{MR2244263}.

\begin{definition}
Let $X \in \on{St}_{A}$. We say $X$ is an affine stack if there is an equivalence of presheaves 
\[
X(-) \simeq \Map_{\on{DAlg}_A}(B, -) 
\]
for some $B \in \on{DAlg}^{\on{ccn}}_A$. Here $\on{DAlg}^{\on{ccn}}_A$ denotes the $\infty$-category of coconnective derived rings over $ A $, equivalently the underlying $\infty$-category of cosimiplicial commutative rings over $ A $ by \cite{soon1}. We let $\mathrm{AffSt}_A$ denote the category of affine stacks over $A.$ 
\end{definition}

\begin{remark}
Note that $\mathrm{AffSt}_A$ is a category with all limits and colimits. Further, the natural functor $\mathrm{AffSt}_A \to \mathrm{St}_A$ preserves all limits. 
\end{remark}

For the purposes of this paper, the category $\mathrm{AffSt}_A$ should be thought of as the category of unipotent homotopy types over $A.$
 
\begin{construction}[Unipotent spectra]\label{unispectradef}
We will construct the category of unipotent stable homotopy types. For brevity, we will instead call them the category of unipotent spectra and denote it by $\mathrm{Sp}^\mathrm{U}_A.$ It is constructed as follows:

Let $\mathrm{AffSt}_{A*}$ denote the category of pointed affine stacks. We define $\mathrm{Sp}^\mathrm{U}_A$ to be the inverse limit of the tower of $\mathbb{Z}$-indexed $\infty$-categories
$$\ldots \to \mathrm{AffSt}_{A*}\xrightarrow{\Omega} \mathrm{AffSt}_{A*}\to\ldots \,.$$
By \cite[Proposition 1.4.2.25]{luriehigher}, we may equivalently define $\mathrm{Sp}^\mathrm{U}_A$ as the $ \infty $-category of spectrum objects in $ \mathrm{AffSt}_{A} $. 
\end{construction}

\begin{remark}
    Given any stable presentable $ \infty $-category $\mathcal{C}$, the functor $ \Sigma^\infty_+ \colon \mathrm{AffSt}_{A}\to \mathrm{Sp}_A^\mathrm{U} $ induces an equivalence between exact colimit-preserving functors $ \mathrm{Sp}_A^\mathrm{U}  \to \mathcal{C} $ and colimit-preserving functors $ \mathrm{AffSt}_A \to \mathcal{C} $ \cite[Corollary 1.4.4.5]{luriehigher}. 
\end{remark}{}

\begin{remark}\label{obs:basechange}
    Given a map of commutative rings $ A \to B $, the base change functor $(\cdot \otimes_A B) \colon  \mathrm{AffSt}_A \to \mathrm{AffSt}_B $ preserves limits, whence it induces a functor $\mathrm{Sp}^\mathrm{U}_A \to \mathrm{Sp}^\mathrm{U}_B $. 
\end{remark}

\begin{remark}
We can alternatively define the $\infty$-category $\mathrm{Sp}^{\mathrm{U}}_A$ as the opposite category to the $\infty$-category of \emph{cospectrum} objects in coconnective derived rings over $ A $, in view of  \cite[Corollaire 2.2.3]{MR2244263} and \cite{soon1}. This gives a purely algebraic description for the $\infty$-category of unipotent spectra. However, the geometric perspective developed in our paper will play a crucial role in elucidating this notion. 
\end{remark}

\begin{remark}
Let $\mathcal{X}$ be an $\infty$-topos. Then a group-like $\mathbb{E}_\infty$-monoid in $\mathcal{X}$ is also naturally a spectrum object of $\mathcal{X}.$ However, $ \mathrm{AffSt}_{A*} $ is not an $\infty$-topos and this breaks down. 
For example, the affine stack $\mathbb{G}_m$ can be given the structure of a group like $E_\infty$-monoid in $\mathrm{AffSt}_A$, but can not be given the structure of a unipotent spectrum by \cref{rep}.
\end{remark}{}

By construction, $\mathrm{Sp}^\mathrm{U}_A$ is a stable $\infty$-category. There is a canonical limit preserving functor $$\Omega^\infty: \mathrm{Sp}^\mathrm{U}_A \to \mathrm{AffSt}_{A*}.$$ 
It follows from \cite[Remark 1.4.2.4]{luriehigher} that $\mathrm{Sp}^\mathrm{U}_A$ is presentable and $ \Omega^\infty $ is accessible. 
By the adjoint functor theorem, $ \Omega^\infty $ admits a left adjoint
$$\Sigma^\infty: \mathrm{AffSt}_{A*} \to \mathrm{Sp}^\mathrm{U}_A.$$
Note that the canonical limit preserving functor 
$\Omega^\infty: \mathrm{Sp}^\mathrm{U}_A \to \mathrm{AffSt}_{A}$ also admits a left adjoint, which we will denote by $\Sigma^\infty_+: \mathrm{AffSt}_{A} \to \mathrm{Sp}^\mathrm{U}_A.$ 

\begin{remark}\label{rmk:sp_uni_to_spectra_in_stacks}
Let $\mathrm{Sp}(\mathrm{St}_A)$ denote the category of spectrum objects of the $\infty$-topos $\mathrm{St}_A$. By construction, we have a fully faithful limit preserving functor 
$$\mathrm{Sp}^\mathrm{U}_A \to \mathrm{Sp}(\mathrm{St}_A).$$ The essential image is spanned by objects $E \in \mathrm{Sp}(\mathrm{St}_A)$ such that $\Omega^{\infty - n}E:=\Omega^\infty (E[n])$ is an affine stack for all $n \ge 0$ (equivalently, for all $n \in \mathbb Z$). 
\end{remark}{}
 \begin{remark}[Unipotent completion of ordinary spectra]\label{unicomofspectra}
Note that $\mathrm{Sp}^\mathrm{U}_A \to \mathrm{Sp}(\mathrm{St}_A)$ admits a left adjoint $ (-)^u \colon \mathrm{Sp}(\mathrm{St}_A) \to \mathrm{Sp}^\mathrm{U}_A $, which can be regarded as ``unipotent completion" of an object of $\mathrm{Sp}(\mathrm{St}_A)$. For any spectrum $G \in \mathrm{Sp}$, we can associate the constant sheaf of spectra $G \in \mathrm{Sp}(\mathrm{St_A})$, whose unipotent completion $G^\uu$ is naturally an object of $\mathrm{Sp}^\U_A.$
\end{remark}{}

\begin{notation}\label{notabelow}
 Let $\mathrm{Sp}^{\mathrm{U-}}_A$ denote the full subcategory of unipotent spectra over $A$ spanned by objects $E$ such that $\pi_i (E) = 0$ for $i \ll 0.$ We will call this the category of bounded below unipotent spectra, which is also a stable $\infty$-category with finite limits and finite colimits.
\end{notation}{}

Let us now specialize to the case where $A = k$ is a field. We will show that in that case, $\mathrm{Sp}^{\mathrm{U-}}_k$ admits a very well-behaved $t$-structure. First we note the following:

\begin{proposition}\label{rep}
 Let $k$ be a field. A bounded below object $E \in \mathrm{Sp}(\mathrm{St}_k)$ is a unipotent spectrum (i.e. belongs to the essential image of the functor described in \cref{rmk:sp_uni_to_spectra_in_stacks}) if and only if for all $i \in \mathbb Z$, $\pi_i(E)$ is representable by a unipotent affine commutative group scheme over $k.$  
\end{proposition}{}

\begin{proof}
  Suppose that $E \in \mathrm{Sp}(\mathrm{St}_k)$ as in the proposition is a unipotent spectrum. Since $E$ is bounded below, for $n \gg 0,$ we can look at $\Omega^\infty (\tau_{\ge -n}{ E})[n] \in \mathrm{Sp}(\mathrm{St}_k)$, which is a pointed \emph{connected} affine stack. Therefore, its homotopy groups must be representable by commutative unipotent affine group schemes. This implies that $\pi_i (E)$ is representable by commutative unipotent affine group schemes for any $i \ge -n$, so in fact for all $i.$

Conversely, under our assumptions on $E$, we need to prove that $\Omega^{\infty - i} E$ is an affine stack for all $i \ge 0.$ For $n \gg 0, $ we note that $\Omega^{\infty-n }E$, by assumption, is a pointed connected stack whose homotopy sheaves are representable by unipotent affine group schemes. Therefore, for $n \gg 0,$ $\Omega^{\infty-n }E$ is an affine stack. Applying the loop construction repeatedly, we see that for $n \gg 0,$ $\Omega^{\infty - i} E$ is an affine stack for all $i \le n,$ so in fact, for all $i,$ as desired. 
\end{proof}{}

\begin{proposition}\label{tstruc}
Let $ k $ be a field. 
Let $(\mathrm{Sp}^\mathrm{U-}_k)_{\le 0}$ denote the full subcategory of $(\mathrm{Sp}^\mathrm{U-}_k)$ spanned by $K \in (\mathrm{Sp}^\mathrm{U-}_k)$ such that $\Omega^{\infty}(K[-1])$ is contractible. 
Then $(\mathrm{Sp}^\mathrm{U-}_k)_{\le 0}$ determines a $t$-structure on $(\mathrm{Sp}^\mathrm{U-}_k)$, where the connective objects $(\mathrm{Sp}^\mathrm{U-}_k)_{\ge 0}$ are given by $L \in (\mathrm{Sp}^\mathrm{U-}_k)$ such that $\pi_i (L) = 0$ for $ i <0 $. 
Moreover, this t-structure is left-separated. 
\end{proposition}{}

\begin{proof}
Note that for the $t$-structure defined as above, any object $L \in (\mathrm{Sp}^\mathrm{U-}_k)$ such that $\pi_i (L) = 0$ for $ i <0$ is connective. It suffices to prove that if $L \in (\mathrm{Sp}^\mathrm{U-}_k)_{\ge 0}$, then it has the property that $\pi_i (L) = 0$ for $i <0.$ Let $n$ be the integer minimal with respect to the property that $n \ge 1$ and $\pi_{-k} (L)= 0$ for all $k \ge n.$ Such an $n$ exists since $L$ is bounded below as an object of $\mathrm{Sp(St}_k)$. It suffices to show that $n=1.$ Suppose that $n >1.$ By construction, the mapping space $\mathrm{Map}(L, \pi_{-(n-1)}(L)[-n+1])$ must be contractible, since by \cref{rep}, $\pi_{-(n-1)}(L)$ is a unipotent spectrum and $n>1$. However, $\mathrm{Map}(L, \pi_{-(n-1)}(L)[-n+1]) \simeq \mathrm{Map}(\pi_{-(n-1)}(L), \pi_{-(n-1)}(L))$, so we conclude that $\pi_{-(n-1)}(L) = 0.$ But that contradicts the minimality of $n,$ which finishes the proof. 

Suppose we are given $ P \in (\mathrm{Sp}^\mathrm{U-}_k)_{\ge 0} $ which is $ n $-connective for all $ n $. 
To show that the t-structure is left-separated, it suffices to show that $ \Omega^\infty P $ is the trivial pointed affine stack. 
However, this follows from hypercompleteness of affine stacks (see \cite[Remark 2.1.14]{soon}). 
\end{proof}{}
\begin{corollary}\label{prc}
 Let $k$ be a field. The category of bounded below unipotent spectra $\mathrm{Sp}^{\mathrm{U-}}_k$ is equipped with a natural $t$-structure (from \cref{tstruc}) for which the heart is equivalent to the category of commutative unipotent affine group schemes over $k$.  
\end{corollary}

We have seen that an object $Y \in \mathrm{Sp}(\mathrm{St}_k)_{\ge 0}$ whose underlying stack is an affine stack may not define an unipotent spectrum (e.g., one may take $Y= \mathbb{G}_m$). Below, we will show that the only obstruction is due to $\pi_0 (Y)$ not being representable by a unipotent affine group scheme.
\begin{proposition}\label{beach}
 Let $k$ be a field. Let $Y \in \mathrm{Sp}(\mathrm{St}_k)$ be such that $\Omega^\infty Y$ is an affine stack. Then $\pi_i (Y)$ is representable by unipotent affine commutative group schemes for $i >0.$ 
\end{proposition}

\begin{proof}
Follows from \cite[Lemma 4.3]{Toee23}.
\end{proof}{}

\begin{corollary}\label{cor:affine_loops_implies_unipotent_spectrum}
   Let $k$ be a field. Let $Y \in \mathrm{Sp}(\mathrm{St}_k)_{\ge 0}$ be such that $\Omega^\infty Y$ is an affine stack and $\pi_0 (Y)$ is representable by unipotent affine commutative group scheme. Then $Y$ is a unipotent spectrum.
\end{corollary}{}

\begin{proof}
 Follows from \cref{rep} and \cref{beach}.   
\end{proof}{}

\begin{example}\label{introex1}
Let $G$ be a commutative unipotent group scheme over a field $k$. 
Then the Eilenberg–MacLane stacks $B^nG:= K(G,n)$ are all affine stacks for $n \ge 1$. 
Since $\Omega B^n G \simeq B^{n-1}G$, by \cref{cor:affine_loops_implies_unipotent_spectrum} the sequence of affine stacks $\left \{B^nG \right \}_{n \ge 0}$ defines a unipotent spectra over $k$. 
We will simply denote this by $G \in \mathrm{Sp}^{\U}_{k}.$ 
\end{example}

\begin{proposition}\label{beach1}
Let $ k $ be a field. The category $(\mathrm{Sp}^{\mathrm{U-}}_k)_{\ge 0}$ has all small limits and the inclusion functor $\iota \colon (\mathrm{Sp}^{\mathrm{U-}}_k)_{\ge 0} \to \mathrm{Sp}({\mathrm{St}}_k)_{\ge 0}$ preserves small limits.
\end{proposition}
\begin{proof}
 Let $F : I \to (\mathrm{Sp}^{\mathrm{U-}}_k)_{\ge 0}$ be a diagram and let $Y$ be the limit of this diagram in $\mathrm{Sp}({\mathrm{St}}_k).$ We can think of $Y$ as the data of an infinite loop object $(\ldots, Y_2, Y_1, Y_0).$ By construction, $Y_n$ is an affine stack for all $ n \in \mathbb{Z} $, since affine stacks are closed under limits. Note that $Y_n = \Omega^\infty (Y[n]).$ By \cref{beach}, it follows that $\pi_i(Y_n)$ is unipotent for $ i>0.$ The limit of the diagram $F$ in $\mathrm{Sp}({\mathrm{St}}_k)_{\ge 0}$ is given by the infinite loop object $(\ldots, \tau_{\ge 2} Y_2, \tau_{\ge 1} Y_1, Y_0).$ It suffices to prove that $\tau_{\ge n} Y_n$ is an affine stack. For $n=0$ the claim follows directly. For $n \ge 1,$ the stack $\tau_{\ge n}Y_n$ is pointed, connected and the homotopy sheaves are unipotent. 
 By \cref{cor:affine_loops_implies_unipotent_spectrum}, each $\tau_{\ge n} Y_n$ is affine, which ends the proof. 
\end{proof}

\subsection{Symmetric monoidal structure on unipotent spectra}\label{monoidal}
In this section, we discuss the construction of a symmetric monoidal structure on ind-unipotent spectra. We will work over a fixed base field $k.$ 
Let $\mathrm{AffSt}_{k*}$ denote the category of pointed affine stacks over $k$. 
First, we will explain how to equip $\mathrm{AffSt}_{k*}$ with the structure of a symmetric monoidal $\infty$-category. 
Next, we show that $ \mathrm{Ind}\left(\mathrm{AffSt}_{k*}\right) $ inherits a symmetric monoidal structure which preserves all small colimits separately in each variable. 
Finally we show that this endows the stabilization of $ \mathrm{Ind}\left(\mathrm{AffSt}_{k*}\right) $ with a symmetric monoidal structure with the same property. 

Note that the left adjoint to the inclusion $\mathrm{AffSt}_{k*} \to \mathrm{St}_{k*}$ is not very well-behaved; namely, it does not commute with finite products. 
This causes difficulties in showing that the natural symmetric monoidal structure $ \mathrm{St}_{k*} $ induces one on $\mathrm{AffSt}_{k*}$. 
Let us illustrate a related issue in the remark below.

\begin{remark}\label{productissue}
 Let $X, Y \in \mathrm{St}_k,$ such that $Y$ is affine. Then it is not necessarily true that the mapping stacks $\underline{\mathrm{Map}}(X, Y)$ and $\underline{\mathrm{Map}} (\mathrm{U}(X), Y)$ are isomorphic. The issue originates from the fact that $\mathrm{U}(X \times \mathrm{Spec}\, A)$ is not in general isomorphic to $\mathrm{U}(X) \times \mathrm{Spec}\, A$; this can be seen by taking $X$ to be an infinite disjoint union of $\mathrm{Spec} \,k$.
\end{remark}

Below, we will impose a certain general condition on $X$ to resolve the issue in \cref{productissue}.

\begin{definition}\label{almostfin}
We call an object $X \in \mathrm{St}_k$ \emph{almost finitary} if for all $ n $, $ \tau_{\leq n}X $ is generated by affine schemes under finite colimits in the category $\tau_{\le n}\mathrm{St}_k.$ Let us denote the full subcategory of $\mathrm{St}_k$ spanned by almost finitary objects to be $\mathrm{St}_k^{\mathrm{afin}}.$
\end{definition}

\begin{proposition}\label{prop1}
The category $\mathrm{St}_k^{\mathrm{afin}}$ has finite colimits.    
\end{proposition}

\begin{proof}
 This follows from the definition since the functor $\tau_{\le n}: \mathrm{St}_k \to \tau_{\le n} \mathrm{St}_k$ preserves colimits (being a left adjoint). 
\end{proof}{}

\begin{proposition}\label{prop2}
 The category $\mathrm{St}_k^{\mathrm{afin}}$ is stable under finite products.
\end{proposition}

\begin{proof}
 Let $X , Y \in \mathrm{St}_k^{\mathrm{afin}}.$ Since $\tau_{\le n} (X \times Y) \simeq \tau_{\le n} X \times \tau_{\le n} Y$, it suffices to show that the full subcategory of $\tau_{\le n} \mathrm{St}_k$ generated under finite colimits by affine schemes, which we will denote by $\mathcal{C}$, is closed under products. 
 This essentially follows because colimits in an $ \infty $-topos are universal and products of affine schemes are affine. 

To this end, first we claim that if $Z \in \mathcal{C},$ then $Z \times \mathrm{Spec}\, A \in \mathcal{C}.$  Note that the full category $\mathcal{C}_{A}$ of $\tau_{\le n}\mathrm{St}_k$ spanned by $Z_0 \in \tau_{\le n}\mathrm{St}_k$ such that $Z_0 \times \mathrm{Spec}\, A \in \mathcal{C}$ contains all affine schemes and is stable under finite colimits (taken in $\tau_{\le n}\mathrm{St}_k)$. By definition of $\mathcal{C}$, this implies that there is a natural embedding $\mathcal{C} \subseteq \mathcal{C}_A$, which implies the claim.

Now we claim that for $U, V \in \mathcal{C}$, we have $U \times V \in \mathcal{C}.$ Note that the full category $\mathcal{C}_{V}$ of $\tau_{\le n}\mathrm{St}_k$ spanned by $Z_0$ such that $Z_0 \times V \in \mathcal{C}$ contains all affine schemes (by the previous paragraph) and is stable under finite colimits. By definition of $\mathcal{C}$, this implies that there is a natural embedding $\mathcal{C} \subseteq \mathcal{C}_V$, which implies the claim. This finishes the proof.
\end{proof}

Note that the category $\mathrm{St}_{k *}$ has a natural symmetric monoidal structure given by the smash product $X \wedge Y,$ defined as a pushout 
\begin{center}
\begin{tikzcd}
X \vee Y \arrow[d] \arrow[r] \ar[rd,phantom,"{\rotatebox{180}{$\lrcorner$}}",very near end] & * \arrow[d] \\
X \times Y \arrow[r]         & X \wedge Y  \end{tikzcd}
\end{center}{}
where $X, Y \in \mathrm{St}_{k *}.$

\begin{proposition}\label{prop:smash_product_stacks_afin}
 The unit of $\mathrm{St}_{k*}$ belongs to $\mathrm{St}_{k*}^{\mathrm{afin}}$. 
 If $X, Y \in \mathrm{St}_{k*}^\mathrm{afin}$, then $X \wedge Y \in \mathrm{St}_{k*}^{\mathrm{afin}}$. 
 In particular, the smash product equips $\mathrm{St}_{k*}^{\mathrm{afin}}$ with the structure of a symmetric monoidal $\infty$-category. 
\end{proposition}

\begin{proof}
That the unit belongs to $ \mathrm{St}_{k*}^{\mathrm{afin}} $ is evident. 
That $\mathrm{St}_{k*}^{\mathrm{afin}}$ is closed under the smash product $ \wedge $ follows from \cref{prop1} and \cref{prop2}. The final part follows from the previous statements and \cite[Remark 2.2.1.2]{luriehigher}.    
\end{proof}

\begin{remark}
 Note that the category $\mathrm{St}_{k*}^{\mathrm{afin}}$ is not closed under taking mapping stacks in $\mathrm{St}_{k_*}^{\mathrm{afin}}$. The mapping stack $\underline{\mathrm{Map}}\left(\mathbb{A}^1_k, \mathbb{A}^1_k\right)$ can be identified with the ind-(affine) scheme $\mathbb{A}^\infty$ which is not almost finitary.  
\end{remark}

\begin{proposition}\label{prop:aff_stack_is_almost_finitary}
 An affine stack $X \in \mathrm{AffSt}_k$ is almost finitary. 
 In other words, the inclusion $ \mathrm{AffSt}_k \to \mathrm{St}_k $ factors through the full subcategory $ \mathrm{St}_k^{\mathrm{afin}} $. 
\end{proposition}{}

\begin{proof}
    By the proof of \cite[Theorem~2.2.9]{MR2244263}, there exists a simplicial scheme $ X_\bullet  =  \Spec A_\bullet $ so that $ \colim_{\Delta^\op} X_\bullet \simeq X $. 
    Now for each $ n $, $ \tau_{\leq n}\left(\colim_{\Delta^\op} X_\bullet\right) \simeq \colim_{\Delta^\op} \tau_{\leq n} X_\bullet $ is equivalent to a finite colimit. 
\end{proof}{}

\begin{lemma}\label{truncaa}
 Let $X \in \mathrm{St}_k$ and $n \ge 0$ be an integer. Then we have a natural isomorphism
\begin{equation}
 \tau_{\ge 0} (R\Gamma(X, \mathcal{O})[n]) \simeq \mathrm{Map}_{\tau_{\le n}\mathrm{St}_k}(\tau_{\le n}X, K(\mathbb{G}_a, n)).   
\end{equation}
   
\end{lemma}{}

\begin{proof}
 Note that we have $ \tau_{\ge 0} (R\Gamma(X, \mathcal{O})[n]) \simeq \mathrm{Map}_{\mathrm{St}_k}(X, K(\mathbb{G}_a, n))$. The lemma now follows because $K(\mathbb{G}_a, n)$ is $n$-truncated.
\end{proof}

\begin{remark}\label{coro}
As a corollary, we obtain a natural isomorphism $ \tau_{\ge 0} (R\Gamma(X, \mathcal{O})[n]) \simeq  \tau_{\ge 0} (R\Gamma(\tau_{\le n}X, \mathcal{O})[n]).$   
\end{remark}

\begin{lemma}\label{filcolim}
Let $X \in \mathrm{St}_k.$ We have a natural isomorphism
$$\varinjlim R\Gamma (\tau_{\le n}X, \mathcal{O}) \simeq R\Gamma (X, \mathcal{O}).$$
\end{lemma}{}

\begin{proof}
 By \cref{coro}, $\varinjlim \mathrm{cofib} (R\Gamma(\tau_{\le n}X, \mathcal{O}) \to R\Gamma(X, \mathcal{O})) \simeq 0,$ which yields the claim.  
\end{proof}{}

\begin{proposition}\label{keyob1}
 Let $X \in \mathrm{St}_k^\mathrm{afin}$ and let $\mathrm{Spec}\, A$ be any affine scheme. Then we have a natural isomorphism $R\Gamma (X, \mathcal{O}) \otimes_k A \simeq R\Gamma (X \times_k \mathrm{Spec}\, A, \mathcal{O}). $ \end{proposition}

\begin{proof}
Let $\mathcal{C}$ denote the full subcategory of $\tau_{\le n} \mathrm{St}_k$ generated under finite colimits by affine schemes. We claim the following:
\begin{itemize}
    \item Let $Y \in \mathcal{C}.$ Then \begin{equation}\label{claaim}
        \tau_{\ge -n}R\Gamma (Y, \mathcal{O}) \otimes_k A \simeq \tau_{\ge -n} R\Gamma (Y \times_k \mathrm{Spec}\, A, \mathcal{O}).
    \end{equation}{}
\end{itemize}
Let $\mathcal{C}_A$ denote the full subcategory spanned by objects $ Y $ of $\tau_{\le n}\mathrm{St}_k$ for which \cref{claaim} holds. Then $\mathcal{C}_A$ contains all affine schemes. Thus to prove our claim, it suffices to prove that $\mathcal{C}_A$ is stable under finite colimits (taken in $\tau_{\le n}\mathrm{St}_k$). Let $\mathcal{Y}: {I} \to \mathcal{C}_A$ denote a finite colimit diagram. For $i \in \mathrm{ob}(I)$, we use $Y_i$ to denote $\mathcal{Y}(i) \in \mathcal{C}_A$. Let $Y:= \mathrm{colim}_I\mathcal{Y} \in \tau_{\le n}\mathrm{St}_k$. We wish to prove that $Y \in \mathcal{C}_A$.  By \cref{truncaa}, we have
\begin{equation}\label{eqq1}
\tau_{\ge 0} (R\Gamma (Y \times_k \mathrm{Spec}\, A, \mathcal{O})[n]) \simeq  \mathrm{Map}_{\tau_{\le n}\mathrm{St}_k}(Y \times_k \mathrm{Spec}\, A, K(\mathbb{G}_a, n)).   
\end{equation}
Using that colimits are universal in $\tau_{\le n} \mathrm{St}_k,$ we have 
\begin{equation}\label{eqq2}
 \mathrm{Map}_{\tau_{\le n}\mathrm{St}_k}(Y \times_k \mathrm{Spec}\, A, K(\mathbb{G}_a, n)) \simeq  \lim_{i \in I} \mathrm{Map}_{\tau_{\le n}\mathrm{St}_k}(Y_i \times_k \mathrm{Spec}\, A, K(\mathbb{G}_a, n)).  
\end{equation}
Now we note that
\begin{align*}
\lim_{i \in I} \mathrm{Map}_{\tau_{\le n}\mathrm{St}_k}(Y_i \times_k \mathrm{Spec}\, A, K(\mathbb{G}_a, n))&\simeq \lim_{i \in I} \tau_{\ge 0}((R\Gamma (Y_i, \mathcal{O}) \otimes_k A)[n])\\&\simeq \lim_{i \in I}\tau_{\ge 0} R\Gamma ((Y_i, \mathcal{O})[n]) \otimes_k A \\& \simeq \left (\lim_{i \in I} \mathrm{Map}_{\tau_{\le n}\mathrm{St}_k} (Y_i, K(\mathbb{G}_a, n))\right)\otimes_k A \\&\simeq \mathrm{Map}_{\tau_{\le n}\mathrm{St}_k}(Y, K(\mathbb{G}_a, n)) \otimes_k A \\&\simeq \tau_{\ge 0} (R\Gamma(Y, \mathcal{O})[n])\otimes_k A.
\end{align*}
In the above, the first isomorphism uses the hypothesis that $Y_i \in \mathcal{C}_A$, the second one uses that the functor $(\cdot) \otimes_k A$ is essentially a filtered colimit (since we are working over a field) and filtered colimits commute with truncation and finite limits, the third and fifth one uses \cref{truncaa} and finally the fourth one simply uses that $Y \simeq \mathrm{colim}_{i \in I}\mathcal{Y}_i$ in $\tau_{\le n} \mathrm{St}_k$.

Combining the above chain of isomorphisms with \cref{eqq1}
and \cref{eqq2} we see that $Y$ satisfies \cref{claaim}, i.e., $Y \in \mathcal{C}_A$. This proves our claim that if $Y \in \mathcal{C},$ then $Y$ satisfies \cref{claaim}.

 Now, for $X \in \mathrm{St}_k^{\mathrm{afin}}$, note that by \cref{filcolim}, we have $R\Gamma (X \times_k \mathrm{Spec}\, A, \mathcal{O})\simeq \varinjlim_n R\Gamma (\tau_{\le n}X \times_k \mathrm{Spec}\, A, \mathcal{O}),$ which is naturally isomorphic to $\varinjlim_n \tau_{\ge -n}R\Gamma (\tau_{\le n}X \times_k \mathrm{Spec}\, A, \mathcal{O})$.
By the claim we proved above, since $\tau_{\le n} X \in \mathcal{C}$, we have $$\tau_{\ge -n}R\Gamma (\tau_{\le n}X \times_k \mathrm{Spec}\, A, \mathcal{O}) \simeq \tau_{\ge -n}R\Gamma (\tau_{\le n}X, \mathcal{O}) \otimes_k A.$$ By taking filtered colimit over $n$ and using \cref{filcolim} again, we obtain
\begin{align*}
R\Gamma (X \times_k \mathrm{Spec}\, A, \mathcal{O}) &\simeq  \varinjlim_n \tau_{\ge -n}R\Gamma (\tau_{\le n}X \times_k \mathrm{Spec}\, A, \mathcal{O}) \\& \simeq \varinjlim_n \tau_{\ge -n}R\Gamma (\tau_{\le n}X, \mathcal{O}) \otimes_k A \\& \simeq R\Gamma(X, \mathcal{O})\otimes_k A.
\end{align*}{}
This proves the proposition.
\end{proof}{}

\begin{proposition}[K\"unneth formula]\label{keyob2}
 Let $X,X' \in \mathrm{St}_k^\mathrm{afin}$. Then we have a natural isomorphism $R\Gamma (X, \mathcal{O}) \otimes_k R\Gamma(X', \mathcal{O}) \simeq R\Gamma (X \times_k X', \mathcal{O}). $ \end{proposition}

\begin{proof}
 Let $\mathcal{C}$ denote the full subcategory of $\tau_{\le n} \mathrm{St}_k$ generated under finite colimits by affine schemes. Let $X \in \mathrm{St}_k^{\mathrm{afin}}$ be fixed as in the proposition. We claim the following:
\begin{itemize}
    \item Let $ Y \in \mathcal{C}$ and $A:= R\Gamma (X, \mathcal{O})$. Then the natural map  
    \begin{equation}\label{claaimu}
     \tau_{\ge -n}  (\tau_{\ge -n}R\Gamma (Y, \mathcal{O}) \otimes_k\tau_{\ge -n} A) \to \tau_{\ge -n} R\Gamma (Y \times_k X, \mathcal{O}).
    \end{equation} is an isomorphism.
\end{itemize}   

Let $\mathcal{C}_X$ denote the full subcategory spanned by those objects $ Y $ of $\tau_{\le n}\mathrm{St}_k$ for which the natural map \cref{claaimu} is an equivalence. By \cref{keyob1}, $\mathcal{C}_X$ contains all affine schemes. Thus to prove our claim, it suffices to prove that $\mathcal{C}_X$ is stable under finite colimits (taken in $\tau_{\le n}\mathrm{St}_k$). Let $Y:= \mathrm{colim}_I{Y}_i \in \tau_{\le n}\mathrm{St}_k$, where $Y_i \in \mathcal{C}_X$. We wish to prove that $Y \in \mathcal{C}_X$. By \cref{truncaa}, we have
\begin{equation}\label{eqq1}
\tau_{\ge 0} (R\Gamma (Y \times_k X, \mathcal{O})[n]) \simeq  \mathrm{Map}_{\tau_{\le n}\mathrm{St}_k}(Y \times_k X, K(\mathbb{G}_a, n)).   
\end{equation}
Using that colimits are universal in $\tau_{\le n} \mathrm{St}_k,$ we have 
\begin{equation}\label{eqq2}
 \mathrm{Map}_{\tau_{\le n}\mathrm{St}_k}(Y \times_k X, K(\mathbb{G}_a, n)) \simeq  \lim_{i \in I} \mathrm{Map}_{\tau_{\le n}\mathrm{St}_k}(Y_i \times_k X, K(\mathbb{G}_a, n)).  
\end{equation}
Now we note that
\begin{align*}
\lim_{i \in I} \mathrm{Map}_{\tau_{\le n}\mathrm{St}_k}(Y_i \times_k X, K(\mathbb{G}_a, n))&\simeq \lim_{i \in I} \tau_{\ge 0}((\tau_{\ge -n}R\Gamma (Y_i, \mathcal{O}) \otimes_k \tau_{\ge -n}A)[n])\\&\simeq \tau_{\ge 0} \left(R\lim_{i \in I} \left( \tau_{\ge -n}R\Gamma (Y_i, \mathcal{O}) \otimes_k (\tau_{\ge -n}A)[n] \right)\right)\\&\simeq \tau_{\ge 0}\left( \left(R\lim_{i \in I}\tau_{\ge -n}R\Gamma (Y_i, \mathcal{O})[n] \right) \otimes_k \tau_{\ge -n}A\right) \\& \simeq \tau_{\ge 0}\left( \left(\lim_{i \in I}\tau_{\ge -n}R\Gamma (Y_i, \mathcal{O})[n]\right) \otimes_k \tau_{\ge -n}A\right)\\& \simeq \tau_{\ge 0}\left (\lim_{i \in I}  \left (\mathrm{Map}_{\tau_{\le n}\mathrm{St}_k} (Y_i, K(\mathbb{G}_a, n))\right)\otimes_k \tau_{\ge -n}A \right) \\&\simeq \tau_{\ge 0} \left(\mathrm{Map}_{\tau_{\le n}\mathrm{St}_k}(Y, K(\mathbb{G}_a, n)) \otimes_k \tau_{\ge -n}A \right)\\&\simeq \tau_{\ge 0} \left (\tau_{\ge 0}(R\Gamma(Y, \mathcal{O})[n])\otimes_k \tau_{\ge -n}A \right) \\& \simeq \tau_{\ge 0} ((\tau_{\ge -n}R\Gamma(Y, \mathcal{O}) \otimes \tau_{\ge -n} A) [n]),
\end{align*}
where the first isomorphism follows from our hypothesis that $Y_i \in \mathcal{C}_X,$ the second isomorphism uses that connective cover is a right adjoint ($R\lim$ denotes limits in spectra), the third isomorphism uses that finite limits in spectra commute with tensor products, the fourth isomorphism uses that tensor product of coconnective objects are coconnective (since we are working over a field), the fifth and seventh isomorphisms follows from \cref{truncaa}, the sixth isomorphism uses the hypothesis that $Y:= \mathrm{colim}_I{Y}_i \in \tau_{\le n}\mathrm{St}_k$, and the last one is clear. 
This proves that the natural map \cref{claaimu} is an equivalence. Now the proposition follows in a way entirely similar to the last paragraph in the proof of \cref{keyob1} by noting that $\tau_{\le 
 n}X' \in \mathcal{C}_X$. This finishes the proof. \end{proof}{}

\begin{proposition}\label{keyob3}
Let $X,X' \in \mathrm{St}_k^\mathrm{afin}$. 
Then the canonical map $\mathrm{U}(X\times_k X') \to \mathrm{U}(X) \times_k \mathrm{U}(X') $ is an equivalence which is moreover natural in $ X $ and $ X' $. 
\end{proposition}

\begin{proof}
 Follows from \cref{keyob2}.   
\end{proof}{}

\begin{proposition}\label{keyob4}
Let $L: \mathrm{St}_{k*} ^{\mathrm{afin}} \to \mathrm{AffSt}_{k*}$ denote the left adjoint to the inclusion functor (see \cref{prop:aff_stack_is_almost_finitary}). Then for any $X, Y \in \mathrm{St}_{k*} ^{\mathrm{afin}}$, we have a natural isomorphism
$$L(X \wedge Y) \simeq L(X \wedge L(Y)).$$
\end{proposition}

\begin{proof}
 By adjunction it suffices to check that for any $Z \in \mathrm{AffSt}_{k*}$, the map 
$$\mathrm{Map}(X \wedge Y, Z) \to \mathrm{Map}(X \wedge L(Y), Z )$$ induced by the counit $ L(Y) \to Y $ is an equivalence, where the mapping spaces can be taken in $\mathrm{St}_{k*}$, which naturally contains $\mathrm{St}_{k*}^{\mathrm{afin}}.$ Note that by construction, the monoidal structure on $\mathrm{St}_{k*}^{\mathrm{afin}}$ is compatible with the (closed) monoidal structure on $\mathrm{St}_{k*}$. Therefore, we have 
$$\mathrm{Map}(X \wedge Y, Z) \simeq \mathrm{Map} (X, \underline{\mathrm{Map}}(Y,Z)).$$ 

By \cref{keyob3}, $ L(Y) \to Y $ induces an equivalence $$\mathrm{Map} (X, \underline{\mathrm{Map}}(Y,Z)) \simeq \mathrm{Map} (X, \underline{\mathrm{Map}}(L(Y),Z)).$$
However, the right hand side is naturally equivalent to $\mathrm{Map}(X \wedge L(Y), Z).$ This finishes the proof.
\end{proof}

\begin{proposition}\label{monoidalaff}
There is a natural symmetric monoidal structure on $\mathrm{AffSt}_{k*}$ such that the left adjoint $L: \mathrm{St}_{k*}^{\mathrm{afin}} \to \mathrm{AffSt}_{k*}$ to the inclusion (see \cref{prop:aff_stack_is_almost_finitary}) is symmetric monoidal, where the symmetric monoidal structure on the former is from \cref{prop:smash_product_stacks_afin}.  
In particular, the symmetric monoidal structure on $ \mathrm{AffSt}_{k*} $ preserves finite colimits separately in each variable. 
\end{proposition}{}

\begin{proof}
 Let $L'$ denote the composite functor $\mathrm{St}_{k*}^{\mathrm{afin}} \xrightarrow{L} \mathrm{AffSt}_{k*} \to \mathrm{St}_{k*}^{\mathrm{afin}}.$ Then $L'$ is a localization functor in the sense of \cite[Example 4.8.2.3]{luriehigher}. The claim now follows from \cref{keyob4} and \cite[Proposition 2.2.1.9]{luriehigher} (also see Example 2.2.1.7 \emph{loc. cit.}). 
\end{proof}{}

\begin{corollary}\label{ind-unispect}
    Let $ k $ be a field.  
    Then 
    \begin{enumerate}
        \item the $ \infty $-category $ \mathrm{Ind}\left(\mathrm{AffSt}_{k*}\right) $ has a symmetric monoidal structure which preserves small colimits separately in each variable. 
        \item the stable $ \infty $-category $ \mathrm{Sp} \,\mathrm{Ind}\left(\mathrm{AffSt}_{k*}\right) $ has a symmetric monoidal structure which preserves small colimits separately in each variable. 
    \end{enumerate} 
\end{corollary}
\begin{proof}
    The first point follows from \cref{monoidalaff} and \cite[Corollary 4.8.1.14(2')]{luriehigher}. 
    The latter point follows from the former and \cite[Propositions 4.8.2.7 \& 4.8.2.18]{luriehigher}. 
\end{proof}

\subsection{Unipotent homology and a profiniteness theorem}\label{applyschemes}

Note that the functor $\Omega^\infty: \mathrm{Sp}({\mathrm{St}}_k)_{\ge 0} \to \mathrm{St}_k$ preserves limits, so by \cref{beach1}, the composite functor $$(\mathrm{Sp}^{\mathrm{U-}}_k)_{\ge 0} \to \mathrm{St}_k$$ also preserves limits. Therefore, there is a left adjoint $$\Sigma^\infty_+ : \mathrm{St}_k \to (\mathrm{Sp}^{\mathrm{U-}}_k)_{\ge 0} .$$ Similarly, there is a left adjoint $$\Sigma^\infty : \mathrm{St}_{k*} \to (\mathrm{Sp}^{\mathrm{U-}}_k)_{\ge 0}. $$

\begin{definition}[Unipotent stable homotopy type] \label{defn:unipotent_stable_htpy_type} Let $Y \in \mathrm{St}_k.$ Then we call $\Sigma^\infty_+ Y$ the unipotent \textit{stable} homotopy type of $Y.$
\end{definition}{}

\begin{definition}[Unipotent stable homotopy groups] Let $Y \in \mathrm{St}_k$ (resp. $ Y \in \mathrm{St}_{k*} $). By \cref{rep}, $\pi_i (\Sigma^\infty_+ Y)$ (resp. $\pi_i(\Sigma^\infty Y) $) is represented by a commutative unipotent affine group scheme for all $i \ge 0.$ We call these the unipotent \textit{stable} homotopy groups of $Y$. 
\end{definition}{}

\begin{remark}\label{rmk:affinization_stabilization_commute}
 Let $Y \in \mathrm{St}_k.$ Let $\bU(Y
)$ denote the unipotent homotopy type of $Y$ in the sense of \cite[\S3]{soon}. 
Then we have a canonical isomorphism 
$ \left(\Sigma^\infty_+ Y\right)^u \simeq \Sigma^\infty_+ \bU(Y).$  
This holds because the diagram of right adjoints
\begin{equation*}
\begin{tikzcd}
   \mathrm{Sp}(\mathrm{St}_k) \ar[r,"{\Omega^\infty}"] & \mathrm{St}_k \\
    (\mathrm{Sp}^{\mathrm{U-}}_k)_{\ge 0}  \ar[u] \ar[r] & \mathrm{AffSt}_k \ar[u] 
\end{tikzcd}
\end{equation*}
commutes. 
\end{remark}

Now, let $Y \in \mathrm{St}_{k*}$ be a pointed $n$-connected stack for $n \ge 0$. Then by the Freudenthal suspension theorem for affine stacks (see the first part of the proof of \cite[Prop. 3.4.10]{soon}), it follows that for $i \le 2n$, the natural map $Y \to \Omega \mathbb U (\Sigma Y)$ induces an isomorphism $$\pi_i(Y) \to \pi_{i+1} (\mathbb U (\Sigma Y)).$$ Therefore, for any $Y \in \mathrm{St}_{k*}$ and  $i \in \mathbb{Z}$, the direct system of homotopy group schemes $\left \{\pi_{i+k} ((\mathbb{U} \Sigma)^k Y) \right \}_k$ is constant for $k \ge \mathrm{max}\left \{i+2, 0 \right \}.$ Therefore, one may define $$\pi_i^{\mathrm{st, U}}(Y) := \varinjlim_{k}\pi_{i+k} ((\mathbb{U} \Sigma)^k (Y)).$$ By definition, it follows that $\pi_i^{\mathrm{st, U}} (Y) =0$ for $i<0$.

\begin{proposition}\label{specfreudenthalcomp}
For any pointed stack $Y$ over a field $k$, we have a natural isomorphism of unipotent group schemes
$$\pi_i^{\mathrm{st, U}}(Y) \simeq \pi_i (\Sigma^\infty Y).$$
\end{proposition}{}

\begin{proof}
By \cref{rmk:affinization_stabilization_commute}, we may without loss of generality assume that $Y$ is an affine stack. Let $ u: \mathrm{Spec} \, A \to \mathrm{Spec}\, k$ be a map of affine schemes. Note that by the adjoint functor theorem, the category of unipotent spectra $\mathrm{Sp}^{\U}_A$ can be equivalently described as the colimit of the following $\mathbb{Z}$-indexed diagram
\begin{equation}\label{unispectracolim}
 \ldots \to \mathrm{AffSt}_{A*} \xrightarrow{ \mathbb{U}\Sigma}  \mathrm{AffSt}_{A*} \to \ldots .  
\end{equation} Let $'\Sigma^\infty: \mathrm{AffSt}_{k*} \to \mathrm{Sp}^\U_k$ denote the functor that sends a pointed affine stack to the zeroth level of the above diagram. By construction, it follows that $'\Sigma^\infty$ admits a right adjoint which is given by $\Omega^\infty: \mathrm{Sp}^\U_k \to \mathrm{AffSt}_{k*}$. Further, we claim that there is a commutative diagram of the following form:

\begin{equation}\label{diagramus}
 \begin{tikzcd}
\mathrm{AffSt}_{A*} \arrow[rr, "\mathbb{U}\Sigma"]                                             &  & \mathrm{AffSt}_{A*}                                                 \\
\mathrm{AffSt}_{k*} \arrow[u, "{(\cdot)\times \mathrm{Spec}\,A}"] \arrow[rr, "\mathbb{U}\Sigma"'] &  & \mathrm{AffSt}_{k*} \arrow[u, "{(\cdot)\times \mathrm{Spec}\, A}"']
\end{tikzcd} 
\end{equation}

To this end, note that for any pointed stack $Y$, we have $\Sigma (Y \times \mathrm{Spec}\, A) \simeq (\Sigma Y) \times \mathrm{Spec}\, A$. Further, if $Y$ is an affine stack, by \cref{prop1}, $\Sigma Y$ is a pointed \emph{almost finitary} stack. Therefore, by \cref{keyob3}, we have $\mathbb{U} ( (\Sigma Y) \times \mathrm{Spec}\, A) \simeq (\mathbb{U} \Sigma Y) \times \mathrm{Spec}\, A.$ This checks the existence of the above commutative diagram. The above diagram induces a functor $$ u^*: \mathrm{Sp}^{\U}_k \to \mathrm{Sp}^\U_A$$ which corresponds to taking the pullback when we view unipotent spectra as certain sheaves of spectra on the sites $\mathrm{Aff}_A$ and $\mathrm{Aff}_k$ respectively. Our discussion implies that 
\begin{equation}\label{comp5}
  {u^*} (' \Sigma^\infty Y) \simeq ('\Sigma^\infty (Y\times \mathrm{Spec}\, A)).  
\end{equation}
Let us denote $\mathcal{F} \coloneqq\,
    '\Sigma^\infty Y$. We will show that 
$\pi_i^{\mathrm{st, U}}(Y) \simeq \pi_i (\mathcal{F})$. To this end, note that $\pi_i (\mathcal{F})$ is the sheafification of the group valued presheaf on $\mathrm{Aff}_k$ that sends $$\mathrm{Spec}\, A \mapsto \pi_0\mathrm{Map}_{\mathrm{Sp}^\U_A}(\mathbb{S}^\uu[i], u^* \mathcal{F}),$$ where $\mathbb{S}^\uu$ denotes the unipotent completion of the sphere spectrum (see \cref{unicomofspectra}). By \cref{comp5}, we have 
$$\pi_0\mathrm{Map}_{\mathrm{Sp}^\U_A}(\mathbb{S}^\uu[i], u^* \mathcal{F}) \simeq \pi_0\mathrm{Map}_{\mathrm{Sp}^\U_A}(\mathbb{S}^\uu[i],\, '\Sigma^\infty (Y \times \mathrm{Spec}\, A)).$$By the description of the category $\mathrm{Sp}^\U_A$ from \cref{unispectracolim}, the right hand side above is equivalent to $$\varinjlim_k \pi_0\mathrm{Map}_{\mathrm{AffSt}_{A*}}(\mathbb{U}(S^{i+k}), (\mathbb{U}\Sigma)^k (Y \times \mathrm{Spec}\, A)).$$ By \cref{diagramus} and adjunction, the above is equivalent to 
$$\varinjlim_k \pi_0\mathrm{Map}_{\mathrm{St}_{A*}}(S^{i+k}, (\mathbb{U}\Sigma)^k (Y )\times \mathrm{Spec}\, A).$$ Since sheafification is a left adjoint, it follows that $\pi_i(\mathcal{F})$ is equivalent to the following direct limit (in the category of sheaves)
$$\varinjlim_k \pi_{i+k}((\mathbb{U}\Sigma)^k (Y)).$$ However, by the discussion before \cref{specfreudenthalcomp}, the above direct system is ind-constant; further, the direct limit is naturally isomorphic to $\pi_i^{\mathrm{st, U}}(Y)$. This shows that $\pi_i^{\mathrm{st, U}}(Y) \simeq \pi_i (\mathcal{F})$, as desired. Finally, the latter isomorphism implies that $\pi_i ('\Sigma^\infty Y)=0$ for $i<0$, i.e., $'\Sigma^\infty Y$ is connective for the $t$-structure in \cref{tstruc}. By the property of adjunction, it follows that $\Sigma^\infty Y \simeq\, '\Sigma^\infty Y$. This gives $\pi_i^{\mathrm{st, U}}(Y) \simeq \pi_i (\Sigma^\infty Y) ,$ which finishes the proof.
\end{proof}{}

\begin{remark}\label{use}
 Let $L: \mathrm{Sp}(\mathrm{St}_k)_{\ge 0} \to (\mathrm{Sp}^{\mathrm{U-}}_k)_{\ge 0}$ denote the left adjoint of the functor in \cref{beach1}. Let $G$ be a commutative affine group scheme over a field $k$ viewed as an object of $\mathrm{Sp}(\mathrm{St}_k)_{\ge 0}.$ Then $L (G) \simeq G^\mathrm{uni},$ where $G^{\mathrm{uni}}$ denotes the universal unipotent, commutative group scheme that receives a map from $G$. We sketch the argument. By considering the kernel of the (surjective) map $G \to G^{\mathrm{uni}}$, one can without loss of generality assume that $G$ is such that $G^{\mathrm{uni}}=0.$ It would suffice to prove that $L(G) \simeq 0.$ By regarding the spectrum $G$ as an infinite loop object $(\ldots, B^2G, BG, G)$, it would suffice to show that $\mathbb{U}(B^nG) \simeq \mathrm{Spec}\, k$ for $n \ge 1.$ This amounts to showing that $R\Gamma (B^nG, \mathcal{O}) \simeq k$ for $n \ge 1$. Applying descent along $* \to B^n G$, we reduce checking the latter claim to $n=1$. Moreover, by base change, we can assume that the field $k$ is algebraically closed. In that case, the group scheme $G$ must be multiplicative which allows us to further reduce to the cases when $G= \mathbb{G}_m$ or $G= \mu_n$ for $n \in \mathbb{N}.$ In these cases, $\mathrm{QCoh}(BG)$ identifies with $\mathbb{Z}$ or $\mathbb{Z}/n\mathbb{Z}$ graded $k$-vector spaces, which implies that the global section functor is exact. This shows that $R\Gamma (BG,\mathcal{O}) \simeq k$, which finishes the argument.
\end{remark}{}

\begin{definition}\label{intro-mod-uni}
 The category $\mathrm{Sp}_A^\mathrm{U}$ is a presentable stable $\infty$-category. In particular, for any $\mathbb{E}_\infty$-ring spectrum $E$, one can talk about the category of (left) $E$-modules in $\mathrm{Sp}_A^\mathrm{U}$ (\cite[Definition 4.2.1.13 \& Remark 4.8.2.20]{luriehigher}). 
 We will denote this category by $E\mathrm{-Mod}^{\mathrm{U}}_A,$ and call it the category of \emph{unipotent $E$-modules} (over $A$).  
\end{definition}

In what follows, we will be most interested in the case when $E = \mathbb{Z}$ or $ E = \mathbb{Z}/p $, and when $A$ is a field $k.$ Note that there is a natural limit-preserving functor 
$$\mathbb{Z}\mathrm{-Mod}^{\mathrm{U}}_k \to \mathrm{Sp}^\mathrm{U}_k. $$
Define the full subcategory of $\mathbb{Z}\mathrm{-Mod}^{\mathrm{U}}_k$ denoted by $\mathbb{Z}\mathrm{-Mod}^{\mathrm{U}-}_k$ which is spanned by objects whose underlying unipotent spectrum is bounded below. Define the full subcategory of $\mathbb{Z}\mathrm{-Mod}^{\mathrm{U}-}_k$ denoted by $(\mathbb{Z}\mathrm{-Mod}^{\mathrm{U}-}_k)_{\ge 0}$ which is spanned by objects whose underlying unipotent spectrum is connective. Similarly, define the full subcategory of $\mathbb{Z}\mathrm{-Mod}^{\mathrm{U}}_k$ denoted by $(\mathbb{Z}\mathrm{-Mod}^{\mathrm{U}-}_k)_{\le 0}$ which is spanned by objects whose underlying unipotent spectrum is coconnective.
\begin{example}\label{ex:eilenberg_maclane_unipotent_spectra}
    Let $G$ be a commutative unipotent group scheme over a field $k$. 
    Then the unipotent spectrum $ G $ over $k$ of \cref{introex1} admits a canonical lift to unipotent $ \bZ $-modules. 
\end{example}
\begin{proposition}\label{ts}
The pair $\left(\left(\mathbb{Z}\mathrm{-Mod}^{\mathrm{U}-}_k\right)_{\ge 0},\left(\mathbb{Z}\mathrm{-Mod}^{\mathrm{U}-}_k\right)_{\le 0}\right)$ define a $t$-structure on $\mathbb{Z}\mathrm{-Mod}^{\mathrm{U}-}_k.$
\end{proposition}{}

\begin{proof}
    Similar to \cref{tstruc}.
\end{proof}{}

There is a natural limit-preserving functor 
$$\left(\mathbb{Z}\mathrm{-Mod}^{\mathrm{U}-}_k\right)_{\ge 0} \to \mathrm{St}_k,$$ whose left adjoint will be denoted by $H^{\uu}_*(\cdot).$

\begin{definition}[Unipotent homology]\label{def2.49}
Let $Y \in \mathrm{St}_k.$ 
We will call $$H^{\uu}_*(Y) \in (\mathbb{Z}\mathrm{-Mod}^{\mathrm{U}-}_k)_{\ge 0}$$ the \emph{unipotent homology} of $Y.$ If $ (Y,y) $ is a pointed stack over $ k $, then the \emph{reduced} unipotent homology of $ Y $ is the cofiber $ \widetilde{H}^{\uu}_*(Y) := \mathrm{cofib}\left( H^{\uu}_*(\{y\}) \to H^{\uu}_*(Y) \right) $. 

For each $i \ge 0,$ the unipotent group scheme $\pi_i (H^{\uu}(Y))$ will be denoted by $H_i^{\uu}(Y)$ and will be called the $i$-th unipotent homology group scheme of $Y.$   
\end{definition}{}

Since $\mathbb{Z}\mathrm{-Mod}^{\mathrm{U}}_k$ is a ($\mathbb{Z}$-linear) stable $\infty$-category, for $M, N \in \mathbb{Z}\mathrm{-Mod}^{\mathrm{U}}_k,$ there is a natural mapping ($\mathbb{Z}$-module) spectrum that we denote by $R\mathrm{Hom}(M,N).$ 
If $ G, H $ are commutative unipotent groups over $ k $, regarded as unipotent $ \bZ $-modules via \cref{ex:eilenberg_maclane_unipotent_spectra}, we will write $ \mathrm{Ext}^i(G,H) $ for $ \pi_{-i}R\mathrm{Hom}(G,H) $. 
\begin{proposition}\label{dual}
 Let $Y \in \mathrm{St}_k.$ Let $G$ be a commutative unipotent group scheme over $k$, which we regard as a unipotent $ \bZ $-module via \cref{ex:eilenberg_maclane_unipotent_spectra}. 
 Then we have a natural isomorphism
$$R\mathrm{Hom}\left(H^\uu_* (Y), G\right) \simeq  R\Gamma_{fl}(Y, G).$$   
\end{proposition}

\begin{proof}
 For $n \ge 0,$ we have $$\mathrm{Map}_{\mathrm{St}_k}(Y, K(G,n)) \simeq \tau_{\ge 0} (R\Gamma _{fl}(Y, G)[n]) \simeq \Omega^{\infty - n}R\Gamma _{fl}(Y, G).$$ By adjunction, we have 
 $$\mathrm{Map}_{\left(\mathbb{Z}\mathrm{-Mod}^{\mathrm{U}}_k\right)_{\geq 0}}\left(H^\uu_* (Y), G[n]\right) \simeq \mathrm{Map}_{\mathrm{Sp}^{\mathrm{U}}_k}\left(\Sigma^\infty_+Y, G[n]\right) \simeq \mathrm{Map}_{\mathrm{St}_k}(Y, K(G,n)).$$ This implies that $R\mathrm{Hom}(H^\uu_* (Y), G) \simeq  R\Gamma_{fl}(Y, G),$ as desired.
\end{proof}{}

\begin{remark}\label{filtrationhomology}
Using \cref{dual} and the Postnikov filtration on $H^\uu_* (Y),$ one can obtain a new filtration on $R\Gamma_{fl} (Y, G),$ which we call the ``homology filtration". This gives a (cohomological) spectral sequence
\begin{equation}\label{ss1}
 E_2^{p,q}:= \mathrm{Ext}^p \left(H^\uu_{q}(Y), G\right) \implies H^{p+q}_{fl}(Y, G). 
\end{equation}{}

\end{remark}{}

\begin{remark}
Let $Y = \mathrm{Spec}\, k \in \mathrm{St}_k,$ where $k$ is a field of characteristic $p>0$. By universal properties, it follows that $H^{\uu}_*(\mathrm{Spec}\, k) \simeq L (\mathbb{Z}) \simeq \mathbb{Z}^\mathrm{uni} \simeq \mathbb{Z}_p$ (see \cref{use1}). Here, $\mathbb{Z}_p$ is thought of as the profinite group scheme $\varprojlim \mathbb{Z}/p^k\mathbb{Z}.$ In particular, we see that $\mathrm{Ext}^i (\mathbb{Z}_p, \mathbb{G}_a) \simeq H^i (\mathrm{Spec}\, k, \mathcal{O}),$ which is zero for $i>0.$ If $k$ is assumed to be of characteristic zero, then $H^{\uu}_{*} (\mathrm{Spec}\, k) \simeq \mathbb{Z}^{\mathrm{uni}} \simeq \mathbb{G}_a.$
\end{remark}

\begin{remark} \label{rmk:zeroth_unipotent_homology_counts_components} 
Let $k$ be a field of characteristic $p$. Let $Y \in \mathrm{St}_k$ be such that $H^0 (Y, \mathcal{O}) = k,$ i.e., $Y$ is cohomologically connected. Then by the spectral sequence \cref{ss1}, we have $\varinjlim \mathrm{Hom} (H^\uu_0 (Y), W_n) \simeq \varinjlim H^0 (Y, W_n).$ Note that $H^0 (Y, W_n) \simeq \mathrm{Hom}(Y, W_n).$ By universal property of mapping to affine schemes, any map $Y \to W_n$ factors uniquely through $\mathrm{Spec}\, H^0 (Y, \mathcal{O}) \to W_n.$ Thus, by our assumption that $H^0 (Y, \mathcal{O}) = k$, it follows that the Dieudonn\'e module of $H^\uu_0 (Y)$ is given by $\varinjlim W_n(k).$
This implies that $H^\uu_0 (Y) \simeq \mathbb{Z}_p.$

\begin{remark}\label{rmkabtred}
 If $ k $ is a field of arbitrary characteristic, a similar argument (by replacing $W_n$ with an arbitrary commutative unipotent group scheme $G$) shows that for any pointed, cohomologically connected stack $ Y $, we have an isomorphism $H^\uu_0 (*) \simeq H^\uu_0 (Y.)$ Therefore, one has $\widetilde{H}^{\uu}_0(Y)=0.$
\end{remark}

\end{remark}{}

\begin{proposition}\label{abelianizationofhomotopy}
    Let $ k $ be a field. 
    Let $Y $ be a pointed, cohomologically connected stack. Then we have a natural isomorphism
    $$H^\uu_{1}(Y) \simeq \pi_1^{\mathrm{U}}(X)^\mathrm{ab} $$ of unipotent group schemes over $ k $.  
\end{proposition}{}

\begin{proof}
    Let $ G $ be an arbitrary commutative unipotent group scheme over $ k $; regard $ G $ as a  unipotent $ \mathbb{Z} $-module via \cref{ex:eilenberg_maclane_unipotent_spectra}. 
    Since $ Y $ is cohomologically connected, by \cite[Lemma 3.1.6]{soon}, we have equivalences $$ \mathrm{Map}_{\mathrm{St}_{k*}}(Y, BG) \simeq \mathrm{Map}_{\mathrm{St}_{k*}}(B \pi_1^{\mathrm{U}}(Y), BG) \simeq \mathrm{Hom}(\pi_1^{\mathrm{U}} (Y)^{\mathrm{ab}}, G)\, , $$ where the latter $ \mathrm{Hom} $ denotes maps of unipotent group schemes over $ k $. 
    On the other hand, we have 
    \begin{equation*}
    \begin{split}
        \mathrm{Map}_{\mathrm{St}_{k*}}(Y, BG) &\simeq \mathrm{Map}_{\left(\mathbb{Z}\mathrm{-Mod}^{\mathrm{U}}_k\right)_{\geq 0}}(\widetilde{H}^{\uu}_*(Y), G[1]) \\
        &\simeq \mathrm{Map}_{\left(\mathbb{Z}\mathrm{-Mod}^{\mathrm{U}}_k\right)_{\geq 1}}(\widetilde{H}^{\uu}_*(Y), G[1]) \\
        &\simeq \mathrm{Map}_{\left(\mathbb{Z}\mathrm{-Mod}^{\mathrm{U}}_k\right)_{\geq 1}}(\widetilde{H}^{\uu}_1(Y)[1], G[1])
    \end{split}
    \end{equation*}  
    where the second equivalence follows from \cref{rmkabtred}. 
    Now the category of 1-connective, 1-truncated unipotent $ \mathbb{Z} $-modules over $ k $ is equivalent to the category of \emph{commutative} unipotent group schemes over $ k $ by \cref{ts}, so $$\mathrm{Map}_{\mathrm{Mod}_{\mathbb{Z}}^{\geq 1}\left(\mathrm{Sp}^{\mathrm{U}}_k\right)}(\widetilde{H}^{\uu}_1(Y)[1], G[1])  \simeq \mathrm{Map}(H^{\uu}_1(Y), G).$$ 
    Therefore, there is a natural isomorphism $ H^{\uu}_1(Y) \simeq \pi_1^\mathrm{U}(Y)^{\mathrm{ab}} $.
\end{proof}{}
By universal properties, for any $X \in \mathrm{St}_{k*},$ there is a natural map
$$\mathbb{U}(X) \to H^\uu_*(X),$$ where the target is regarded as a pointed stack via the functor $$(\mathbb{Z}\mathrm{-Mod}^{\mathrm{U}-}_k)_{\ge 0} \to \mathrm{St}_{k*}.$$  This induces natural maps
$$\pi^\mathrm{U}_n(X) \to H^\uu_n(X),$$ which we call the Hurewicz map.

\begin{proposition}[Hurewicz theorem]\label{hurehomology}
Let $Y$ be a pointed and cohomologically connected stack over a field $k$. Let $n \ge 1$ be an integer such that $\mathbb{U}(Y)$ is $n$-connected. Then $H^\uu_{i}(Y)= 0$ for $0<i < n+1$ and the Hurewicz map
$$\pi_{n+1}^\U (X) \to H^\uu_{n+1}(X)$$ is an isomorphism.
\end{proposition}{}

\begin{proof}

Let $0<i<n+1 $ be an integer. By \cite[Proposition~3.2.11]{soon}, it follows that $H^i (Y, \mathscr{O})=0$. By the spectral sequence in \cref{filtrationhomology}, it (inductively) follows that $\mathrm{Hom}(H^\uu_{i}(Y), \mathbb{G}_a)=0$. Since $H^\uu_{i}(Y)$ is unipotent, we must have $H^\uu_{i}(Y) = 0$ for $0 < i <n+1$. For the second part of the proposition, let $G$ be an arbitrary commutative unipotent group scheme over $k$. Similar to the proof of \cref{abelianizationofhomotopy}, it follows that $$ \mathrm{Map}_{\mathrm{St}_{k*}}(Y, B^{n+1}G) \simeq \mathrm{Map}_{\mathrm{St}_{k*}}(\tau_{\le n+1} \mathbb{U}(Y), B^{n+1}G) \simeq \mathrm{Hom}(\pi_{n+1}^{\mathrm{U}} (Y)^{\mathrm{}}, G)\,.$$ On the other hand, we have 
    \begin{equation*}
    \begin{split}
        \mathrm{Map}_{\mathrm{St}_{k*}}(Y, B^{n+1}G) &\simeq \mathrm{Map}_{\mathrm{Mod}_{\mathbb{Z}}^{\geq 0}\left(\mathrm{Sp}_k^{\mathrm{U}}\right)}(\widetilde{H}^{{\uu}}_*(Y), G[n+1]) \\
        &\simeq \mathrm{Map}_{\mathrm{Mod}_{\mathbb{Z}}^{\geq n+1}\left(\mathrm{Sp}^{\mathrm{U}}_k\right)}(\widetilde{H}^{{\uu}}_*(Y), G[n+1]) \\
        &\simeq \mathrm{Map}_{\mathrm{Mod}_{\mathbb{Z}}^{\geq 1}\left(\mathrm{Sp}^{\mathrm{U}}_k\right)}(\widetilde{H}^{{\uu}}_{n+1}(Y)[n+1], G[n+1]) \\&\simeq \mathrm{Hom}(\widetilde{H}^{{\uu}}_{n+1}(Y), G).
    \end{split}
    \end{equation*}  This proves the desired claim.
\end{proof}

\begin{corollary}
   Let $Y \in \mathrm{St}_k$ be cohomologically connected and pointed. Let $n \ge 1$ be an integer such that $H^i (Y, \mathcal{O})=0$ for $0<i<n+1$. Then $H^\uu_{i}(Y)= 0$ for $0<i < n+1$ and there is a natural isomorphism
$$\mathrm{Hom}(H^\uu_{n+1}(Y), \mathbb{G}_a) \simeq H^{n+1}(Y, \mathcal{O}).$$
\end{corollary}{}
\begin{proof}
  Follows from \cite[Proposition~3.2.11]{soon} and \cref{hurehomology}.  
\end{proof}

\begin{lemma}\label{bd1}
Let $M_j$ be an inverse system of commutative affine group schemes over $k.$ Then for all $i \ge 0,$ we have a natural isomorphism
$$\varinjlim_j \mathrm{Ext}^i \left(M_j, \mathbb{G}_a\right) \simeq \mathrm{Ext}^i \left(\varprojlim_{j}M_j, \mathbb{G}_a\right).$$   
\end{lemma}{}

\begin{proof}
Let $M:= \varprojlim_{j}M_j$. Since $M$ is affine, by the Breen--Deligne resolution, $R\mathrm{Hom}(M, \mathbb{G}_a)$ is naturally isomorphic to a complex 
\begin{equation}\label{breendeligne}
\mathcal{O}(M) \to \mathcal{O}(M)^{\otimes 2}\to \ldots \to \oplus_{j=1}^{n_i} \mathcal{O}(M)^{\otimes r_{i,j}} \to \ldots .    
\end{equation}
Since $\mathcal{O}(M) \simeq \varinjlim_j \mathcal{O}(M_j),$ the functoriality of the Breen--Deligne resolution shows that
$$\varinjlim_j R\mathrm{Hom}(M_j, \mathbb{G}_a) \simeq R\mathrm{Hom}(M, \mathbb{G}_a).$$ Taking cohomology yields the desired result.
\end{proof}{}

\begin{lemma}\label{useful}
    Let $M_j$ be an inverse system of commutative affine group schemes over $k.$ Let $G$ be a finite type commutative unipotent group scheme over $k.$ Then for all $i \ge 0,$ we have a natural isomorphism
$$\varinjlim_j \mathrm{Ext}^i \left(M_j, G\right) \simeq \mathrm{Ext}^i \left(\varprojlim_{j}M_j, G\right).$$
\end{lemma}

\begin{proof}
 We have a natural map
$$\varinjlim_j R\mathrm{Hom}\left(M_j, {G}\right) \simeq R\mathrm{Hom}(M, {G}).$$ Since $G$ is finite type, $V_G^n=0$ for some $n$. To prove that the above natural map is an isomorphism, by using the short exact sequence $0 \to VG \to G \to G/VG \to 0,$ one may reduce to the case when $V_G=0.$ In that case one may write $G$ as 
$$0 \to G \to \prod_I \mathbb{G}_a \to \prod_J \mathbb{G}_a \to 0,$$ where $I$ and $J$ are finite sets. To obtain such an exact sequence one may use the classification of finite type unipotent group schemes in terms of $k_\sigma[F]$-modules (see \cite[IV, § 3, Corollary 6.7]{Demazure:1970} and \cite[Lemma~4.2.32]{soon}). Using this exact sequence, one may further reduce to $G=\mathbb{G}_a$, which follows from \cref{bd1}. 
\end{proof}{}

\begin{lemma}\label{huseful}
    Let $M_j$ be an inverse system in $\left(\mathbb{Z}\mathrm{-Mod}^{\mathrm{U}-}_k\right)_{\ge 0}$, and denote its inverse limit in $\left(\mathbb{Z}\mathrm{-Mod}^{\mathrm{U}-}_k\right)_{\ge 0}$ by $\varprojlim_{j} M_j.$ Let $G$ be a finite type commutative unipotent group scheme over $k$, regarded as a unipotent $ \bZ $-module via \cref{ex:eilenberg_maclane_unipotent_spectra}. 
    Then we have a natural isomorphism
    $$\varinjlim_j \mathrm{RHom} \left(M_j, G\right) \simeq \mathrm{RHom} \left(\varprojlim_{j}M_j, G\right).$$
\end{lemma}
\begin{proof}
The case when $G = \mathbb{G}_a$ follows in a way similar to \cref{bd1} by applying the Breen--Deligne resolution in an animated form. The case of a general finite type unipotent group scheme $G$ is deduced in a way similar to the proof of \cref{useful}.
\end{proof}{}

\begin{lemma}\label{bd2}
Let $M$ be a finite group scheme over $k.$ Then for all $i \ge 0,$ the $k$-vector space $\mathrm{Ext}^i (M,\mathbb{G}_a)$ is finite-dimensional.    
\end{lemma}

\begin{proof}
 By the Breen--Deligne resolution, $\mathrm{Ext}^i (M, \mathbb{G}_a)$ is the $i$-th cohomology of the complex \cref{breendeligne}. Since $\mathcal{O}(M)$ is a finite dimensional $k$-algebra, we obtain the desired claim.
\end{proof}{}
Recall that a left $ k_\sigma[F] $-module $ M $ is \emph{torsion} if each $ m \in M $ is contained in a $ k_\sigma[F] $-submodule $ N_m $ so that $ N_m $ is finite-dimensional as a $ k $-vector space. 
\begin{lemma}\label{use1}
 Let $M$ be a profinite commutative unipotent group scheme over $k.$ Then $\mathrm{Ext}^i (M, \mathbb{G}_a)$ is a torsion $k_\sigma[F]$-module for each $i \ge 0.$ 
\end{lemma}

\begin{proof}
Follows from \cref{bd1}  and \cref{bd2}. 
\end{proof}

\begin{proposition}[Profiniteness] \label{profinitenessthm}Let $X$ be a stack over $k$ such that $H^i (X, \mathcal{O})$ is a torsion $k_\sigma [F]$-module for each $i \ge 0.$ Then $H^\uu_{i}(X)$ is a profinite unipotent commutative group scheme for each $i \ge 0.$ 
    
\end{proposition}{}

\begin{proof}
 We use \cref{ss1} when $G= \mathbb{G}_a.$ Since $H^i (X, \mathcal{O})$ is a torsion $k_\sigma[F]$-module (and the filtration is compatible with the Frobenius), it follows that $E^{0,i}_\infty$ is naturally a torsion $k_{\sigma}[F]$-module. Our goal is to prove that $E^{0,i}_2 = \mathrm{Hom}\left(H^\uu_{i}(X), \mathbb{G}_a\right)$ is a torsion $k_\sigma[F]$-module. The claim is clear from the spectral sequence \cref{ss1} when $i=0.$ We will prove by descending induction on $r$ and ascending induction on $i$ that $E_r^{0,i}$ is a torsion $k_\sigma[F]$-module for $r \ge 2, i \ge 0.$ For a fixed $i>0,$ note that $E^{0,i}_{i+2} = E^{0,i}_\infty,$ and therefore, is a torsion $k_\sigma[F]$-module. Note that we have an exact sequence 
\begin{equation}\label{ssexact}
0 \to E_{r+1}^{0,i}  \to E_{r}^{0,i} \to E_{r}^{r, i-r+1}
\end{equation}for all $r \ge 2.$ 
Since $i-r+1 <i,$ by induction, $E_2^{0, i-r+1}$ is a torsion $k_\sigma[F]$-module, or equivalently, $H^\uu_{i-r+1}(X)$ is a profinite group scheme. By \cref{use}, $E^{r, i-r+1}_2 = \mathrm{Ext}^r (H^\uu_{i-r+1}(X), \mathbb{G}_a)$ is a torsion $k_\sigma[F]$-module. Therefore, $E^{r, i-r+1}_r$ is also a torsion $k_\sigma[F]$-module. By descending induction on $r$, we can suppose that $E^{0,i}_{r+1}$ is a torsion $k_\sigma[F]$-module. The exact sequence \cref{ssexact} therefore implies that $E_r^{0,i}$ is a torsion $k_\sigma[F]$-module. Therefore, by induction, we obtain the desired claim that $E_2^{0,i}$ is a torsion $k_\sigma[F]$-module. This finishes the proof.
\end{proof}

\begin{definition}[Unipotent local homology]
\label{defn:unipotent_local_homology}
Let $X$ be a scheme over $k.$ Let $Y$ be a closed subscheme of $X$ and let $U := X - Y.$ We define $$H^\uu_{*,Y}(X):= \mathrm{cofib}\left(H^\uu_*(U) \to H^\uu_*(X)\right),$$ where the cofiber is taken in the stable $\infty$-category $\mathbb{Z}\mathrm{-Mod}^\U_k.$
It follows that $H^\uu_{*, Y}(X) \in (\mathbb{Z}\mathrm{-Mod}^\U_k)_{\ge 0};$ we will call this object unipotent local homology. 
\end{definition}{}

The following definition is classical.

\begin{definition}[Local cohomology]Let $X$ be a scheme over $k.$ Let $Y$ be a closed subscheme of $X$ and let $U := X - Y.$ Let $G$ be a commutative unipotent group scheme over $k.$ One defines 
$$R\Gamma_Y(X, G):=\mathrm{fib}\left(R\Gamma(X, G) \to R\Gamma(U, G)\right).$$    
\end{definition}

\begin{proposition}\label{dual2}
 Let $X$ be a scheme over $k.$ Let $Y$ be a closed subscheme of $X$ and let $U := X - Y.$ Let $G$ be a commutative unipotent group scheme over $k.$ Then we have a natural isomorphism

$$R\mathrm{Hom}\left(H^\uu_{*, Y}(X), G\right) \simeq R\Gamma_Y (X, G). $$
\end{proposition}{}

\begin{proof}
 Follows from \cref{dual}.   
\end{proof}

\begin{remark}\label{eat}
Using \cref{dual2} and the Postnikov filtration on $H^\uu_{*,Y} (X),$ one can obtain a new filtration on $R\Gamma_{Y} (X, G),$ which we call the ``local homology filtration". This gives a (cohomological) spectral sequence 
\begin{equation}\label{ss2}
 E_2^{p,q}:= \mathrm{Ext}^p \left(H^\uu_{q,Y}(X), G\right) \implies H^{p+q}_{Y}(X, G).   
\end{equation}{}

\end{remark}{}

To prove certain standard properties about unipotent local homology, the following results will be useful.

\begin{lemma}\label{uselemma}
 Let $f: P \to Q$ be a morphism in $\left(\mathbb{Z}-\mathrm{Mod}^\mathrm{U}_k\right)_{\ge 0}.$ Suppose that for every commutative unipotent group scheme $G$, the induced map
$$R\mathrm{Hom}(Q, G) \to R\mathrm{Hom}(P, G)$$ is an isomorphism. Then $f$ is an isomorphism.
\end{lemma}{}

\begin{proof}
 By passing to the cofiber of $P\to Q$, we can without loss of generality assume that $P =  0.$ Then by hypothesis, $R\mathrm{Hom}(Q, G)=0$ for every unipotent group scheme $G.$ Note that $\Omega^\infty Q$, being an affine stack, is hypercomplete. Note that $\pi_i (Q)$ is representable by a commutative unipotent affine group scheme for all $i \ge 0.$ Since $\mathrm{Map}(Q, G)=0,$ setting $G = \pi_0(Q)$ shows that $Q[-1] \in \left(\mathbb{Z}-\mathrm{Mod}^\mathrm{U}_k\right)_{\ge 0}$. Repeating this argument with $Q'= Q[-1]$ shows that $Q[-2] \in \left(\mathbb{Z}-\mathrm{Mod}^\mathrm{U}_k\right)_{\ge 0}.$ Inductively, we obtain that $Q[-n]$ is connective for all $n$; or in other words, $Q$ is $\infty$-connective. 
 Since the t-structure on $ \mathbb{Z}-\mathrm{Mod}^{\mathrm{U}-}_k $ is left-separated by \cref{tstruc}, it follows that $ Q \simeq 0 $. 
 This finishes the proof.
\end{proof}

\begin{proposition}\label{chekisom}
   Let $f: P \to Q$ be a morphism in $\left(\mathbb{Z}-\mathrm{Mod}^\mathrm{U}_k\right)_{\ge 0}.$ Suppose that the induced map
$$R\mathrm{Hom}(Q, \mathbb{G}_a) \to R\mathrm{Hom}(P, \mathbb{G}_a)$$ is an isomorphism. Then $f$ is an isomorphism.  
\end{proposition}{}

\begin{proof}
By \cref{uselemma}, it is enough to show that for every commutative unipotent group scheme $G$, the induced map
$$R\mathrm{Hom}(Q, G) \to R\mathrm{Hom}(P, G)$$is an isomorphism. By our hypothesis, the above map is an isomorphism when $G= \mathbb{G}_a^I$, where $I$ is some index set. If the Verschiebung $V_G$ on $G$ is zero, then there is a fiber sequence
$G \to \mathbb{G}_a^I \to \mathbb{G}_a^J.$ Thus the map is an isomorphism when $V_G=0.$ Using induction and arguing using the filtration induced by $V_G$, the map is an isomorphism for $G/V_G^n.$ Since $G$ is unipotent, $G \simeq \varprojlim_n G/V_G^n.$ Thus the map is an isomorphism for any commutative unipotent group scheme $G.$ This finishes the proof.
\end{proof}{}

\subsection{Recognition theorem for unipotent spectra} \label{subsection:recognition_thm}
Let $ k $ be a field. 
In \cite[\S4.2]{Toee23}, To\"en shows that the category of $ \mathbb{Z} $-modules in unipotent spectra over $ k $ is equivalent to modules over the endomorphism ring spectrum of $ \mathbb{G}_a $; this may be regarded as a variation on Dieudonn\'e theory for unipotent group schemes (see \emph{loc. cit.} for subtleties). 
Let $\mathrm{Sp}^\mathrm{U-}_k$ denote the category of bounded below unipotent spectra over $k$, which is a stable $\infty$-category. 
Consider the object $ \mathbb{G}_a \in \mathrm{Sp}^\mathrm{U-}_k$ (\cref{introex1}). 
Then the endomorphism spectrum $$R:=\mathrm{End}_{\mathrm{Sp}^\mathrm{U-}_k}(\mathbb{G}_a)$$ can naturally be viewed as an $\mathbb{E}_1$-ring. 
For any $E \in \mathrm{Sp}^\mathrm{U-}_k,$ the mapping spectrum denoted by $R\mathrm{Hom}(E, \mathbb{G}_a)$ can naturally be viewed as a right module over $R$. The assignment $E \mapsto R\mathrm{Hom}(E, \mathbb{G}_a)$ promotes to a functor 
$$M: \mathrm{Sp}^\mathrm{U-}_k \to \mathrm{RMod}_R^{\mathrm{op}},$$ where the right-hand side denotes modules in the category of spectra. 
Our goal in this subsection is to prove the following.

\begin{proposition}\label{recog}
The functor $M: \mathrm{Sp}^\mathrm{U-}_k \to \mathrm{RMod}_R^{\mathrm{op}}$ constructed above is fully faithful.  
\end{proposition}{}

In order to prove the above proposition, we will need some preparations. 

\begin{lemma}\label{genbylim}
Let $ \mathcal{C} \subseteq \left(\mathrm{Sp}^\mathrm{U-}_k\right)_{\geq 0} $ be the full subcategory of connective unipotent spectra over $ k $ generated under arbitrary limits by the collection $ \left\{\mathbb{G}_a[n]\right\}_{n \geq 0} $. 
Then $ \mathcal{C} = \left(\mathrm{Sp}^\mathrm{U-}_k\right)_{\geq 0} $. 
\end{lemma}

\begin{proof}
    Take $ E \in \left(\mathrm{Sp}^\mathrm{U-}_k\right)_{\geq 0} $; we will show that $ E $ belongs to $ \mathcal{C} $. 
    By writing $E \simeq \varprojlim \tau_{\le n} E$, we may assume that $E$ is connective and bounded. Moreover, by devissage, we can assume that $E$ lies in the heart of the t-structure on $\mathrm{Sp}^\mathrm{U-}_k$ (\cref{tstruc}). 
    By \cref{prc}, such an $E$ arises from a commutative unipotent group scheme over $k$ as in \cref{introex1}. 
    Since one can write $E \simeq \varprojlim_n E/V_E^n$, where $V_E$ denotes the Verschiebung on $E$, we may further assume by devissage that $E$ is killed by $V_E$. 
    In that case, there is a short exact sequence $$0 \to E \to \prod_I \mathbb{G}_a \to \prod_J \mathbb{G}_a \to 0,$$ where $I$ and $J$ are (possibly infinite) index sets. This finishes the proof. 
\end{proof}{}

We will additionally need the following lemma, which can be thought of as a spectral refinement of the Breen--Deligne resolution.

\begin{lemma}[Spectral Breen--Deligne resolution]\label{breendel} There exists a sequence of functors $F_i: \mathrm{Sp}_{\ge 0} \to \mathrm{Sp}_{\ge 0}$ for $i \ge -1$ with natural transformations $$0= F_{-1} \to F_0 \to F_1 \to \ldots $$ such that we have \begin{enumerate}

\item $F_i/ F_{i-1}$ is naturally isomorphic to a finite direct sum of functors of the form $\Sigma^i \Sigma^\infty_+ \prod_{t=1}^{n_{d,i}} \Omega^\infty (\cdot)$ for some fixed $n_{d,i} \in \mathbb{N}.$ 
    \item $\varinjlim F_i \simeq \mathrm{id}$ as endofunctors of $\mathrm{Sp}_{\ge 0}.$
\end{enumerate}{}
    
\end{lemma}{}

\begin{proof}
    We will freely use the results from \cite[\S4.1]{MR4280864}. Let $\mathcal{C} := \mathrm{Sp}_{\ge 0}$, the category of connective spectra. Let $F: \mathcal{C} \to \mathrm{Sp}_{\ge 0}$ be the identity functor. Define $\mathcal{D}$ to be the full subcategory of $\mathcal{C}$ spanned by suspension spectrum of finite sets. By Proposition 4.7 loc. cit. we would be done if we can prove that $F$ is $\mathcal{D}$-pseudocoherent, which we do below.

Since $\mathcal D$ is closed under finite products, by Proposition 4.10, loc. cit. it is enough to show that $\Sigma^\infty_+ \Omega^\infty F$ is $\mathcal D$-pseudocoherent. To this end, by Def. 4.4, loc. cit. it is enough to prove that  $\Sigma^\infty_+ \Omega^\infty F$  is $\mathcal D$-perfect. This is however clear from the definition, as 
$$\Sigma^\infty_+ \Omega^\infty F (\cdot) \simeq  \Sigma^\infty_+ \mathrm{Map}_{\mathcal{C}} (\mathbb S, \cdot),$$ where $\mathbb S = \Sigma^\infty_+ \left \{*\right \},$ the sphere spectrum, which is in $\mathcal D$ by construction.
\end{proof}

\begin{proposition}\label{cocompacthard}
Fix an integer $n \ge 0$. Then $ \mathbb{G}_a[n] $ is cocompact as an object of $ \left(\mathrm{Sp}^{\mathrm{U}-}_{k}\right)_{\geq 0} $.     
\end{proposition}

\begin{proof}
Suppose that $E$ is a connective unipotent spectrum. By functoriality of the spectral Breen--Deligne resolution, we obtain a direct system $$0 = F_{-1}(E) \to F_0(E) \to F_1(E) \to \ldots \to F_j (E) \to \ldots$$   of sheaf (on $\mathrm{Alg}_k^{\mathrm{op}}$) of connective spectra such that $\varinjlim F_j(E) \simeq E$ and each cofiber $F_{r}(E)/F_{r-1}(E)$ is naturally isomorphic to a finite direct sum of objects of the form $\Sigma^r \Sigma^{\infty}_+ \prod_{t=1}^{n_{d,r}} \Omega^\infty E.$ 

We claim that \begin{equation}\label{theequ}
    \mathrm{Map}\left(E, \mathbb{G}_a[n]\right) \simeq \mathrm{Map}\left(F_{n+1}(E), \mathbb{G}_a[n]\right),
\end{equation}where the mapping spaces are taken in sheaves of connective spectra. To this end, note that 
\begin{equation*}
 \mathrm{Map}(E, \mathbb{G}_a[n]) \simeq \varprojlim \mathrm{Map}(F_j(E), \mathbb{G}_a[n]);   
\end{equation*}
therefore, it suffices to show that the maps
\begin{equation*}
   \mathrm{Map}(F_{j+1}(E), \mathbb{G}_a[n]) \to  \mathrm{Map}(F_j(E), \mathbb{G}_a[n])  
\end{equation*}
are isomorphisms for $j \ge n+1$. However, this follows the description of $F_i(E)/F_{i-1}(E)$ and the fact that the mapping spectrum $$R\mathrm{Hom} (\Sigma^j \Sigma^\infty_+ \prod_{t=1}^{n_{d,j}} \Omega^\infty E, \mathbb{G}_a[n]) $$ has a vanishing $\pi_{-1}$ for $j \ge n+2$. This proves the claim \cref{theequ}.

Further, we claim that if $G \simeq \lim_i G_i$ is a cofiltered limit diagram of connective unipotent spectra and $m \ge 0$, then for any fixed $j \ge 0$, we have \begin{equation}\label{secondclaim}
\varinjlim_i \mathrm{Map}(F_j (G_i), \mathbb{G}_a[m]) \simeq  \mathrm{Map}(F_j(G), \mathbb{G}_a[m]).  
\end{equation}
We will prove this claim by induction on $j$. When $j=0$, the claim follows from the fact that $F_0(E)$ is naturally isomorphic to a finite direct sum of objects of the form $\Sigma^{\infty}_+ \prod_{t=1}^{n_{d,0}} \Omega^\infty E$. Indeed, $$\varinjlim_i \mathrm{Map}\left(\Sigma^{\infty}_+ \prod_{t=1}^{n_{d,0}} \Omega^\infty G_i, \mathbb{G}_a[m]\right) \simeq \varinjlim_i \mathrm{Map}\left( \prod_{t=1}^{n_{d,0}} \Omega^\infty G_i, K(\mathbb{G}_a,m)\right),$$ where the latter mapping space can be considered in the category of affine stacks, as $\Omega^\infty G_i$ is an affine stack. However, $K(\mathbb{G}_a, m)$ is a cocompact object in the category of affine stacks. Therefore, we have $$\varinjlim_i \mathrm{Map}\left( \prod_{t=1}^{n_{d,0}} \Omega^\infty G_i, K(\mathbb{G}_a,m)\right) \simeq \mathrm{Map}\left( \prod_{t=1}^{n_{d,0}} \Omega^\infty G, K(\mathbb{G}_a,m)\right).$$ The latter is isomorphic to $\mathrm{Map}(F_0(G), \mathbb{G}_a[m])$ by adjunction, which proves the case when $j=0$.

Now we suppose the claim in \cref{secondclaim} holds for a fixed $j \ge 0$; we will check that it holds for $j+1$. Let $\mathrm{gr}_r(E):= F_r(E)/F_{r-1}(E).$ By arguing in a manner similar to the above paragraph, we obtain
\begin{equation}\label{greq}\varinjlim_i \mathrm{Map}\left(\mathrm{gr}_r(G_i), \mathbb{G}_a[m]\right) \simeq \mathrm{Map}\left(\mathrm{gr}_r(G), \mathbb{G}_a[m]\right)\end{equation} for all $r,m \ge 0$. Note that we have a map of fiber sequences 
\begin{center}
\begin{tikzcd}
{\varinjlim_{i}R\mathrm{Hom}\left(F_j(G_i), \mathbb{G}_a\right)} \arrow[r] \arrow[d] & {\varinjlim_{i}R\mathrm{Hom}\left(F_{j+1}(G_i), \mathbb{G}_a\right)} \arrow[r] \arrow[d] & {\varinjlim_{i}R\mathrm{Hom}(\mathrm{gr}_{j+1}(G_i), \mathbb{G}_a)} \arrow[d] \\
{R\mathrm{Hom}(F_j(G), \mathbb{G}_a)} \arrow[r]                           & {R\mathrm{Hom}(F_{j+1}(G), \mathbb{G}_a)} \arrow[r]                           & {R\mathrm{Hom}(\mathrm{gr}_{j+1}(G), \mathbb{G}_a)} \,.                         
\end{tikzcd}                        
 \end{center}
 The left vertical map is an isomorphism by our inductive hypothesis. The right vertical map is an isomorphism by \cref{greq}. Therefore the middle vertical map is also an isomorphism, implying our claim in \cref{secondclaim} for $j+1$. This completes the induction and proves the claim \cref{secondclaim} for all $j \ge 0$.

 Now we can deduce the cocompactness of $\mathbb{G}_a[n]$. Indeed, given a cofiltered limit diagram of connective unipotent spectra $G \simeq \lim_i G_i$,  
 \begin{align*}
  \varinjlim_i \mathrm{Map} (G_i, \mathbb{G}_a[n] ) & \simeq \varinjlim_i \mathrm{Map}(F_{n+1}(G_i), \mathbb{G}_a[n]) \\& \simeq \mathrm{Map}(F_{n+1}(G), \mathbb{G}_a[n]) \\& \simeq \mathrm{Map} (G, \mathbb{G}_a[n]), 
 \end{align*}where the first and the third isomorphism follow from \cref{theequ}; the second one follows from \cref{secondclaim}. This finishes the proof.
\end{proof}{}

\begin{corollary}\label{qin}
Let $E \simeq \varprojlim_i E_i$ be a cofiltered limit diagram in $\mathrm{Sp}^{\mathrm{U}-}_k$ where $E, E_i$ are all connective. Then the natural map
$$\varinjlim _i R\mathrm{Hom}\left(E_i, \mathbb{G}_a\right) \simeq R\mathrm{Hom}(E, \mathbb{G}_a)$$ is an isomorphism.
\end{corollary}

\begin{proof}
 Follows from \cref{cocompacthard}.   
\end{proof}{}

\begin{proof}[Proof of \cref{recog}]
 There is a natural colimit-preserving embedding $\left(\mathrm{Sp}^\mathrm{U-}_k\right)_{\ge 0} \to \mathrm{Sp}^\mathrm{U-}_k$, where the source denotes the category of connective unipotent spectra. It suffices to prove that the restricted functor $$M': \left(\mathrm{Sp}^\mathrm{U-}_k\right)_{\ge 0} \to \mathrm{RMod}_R^{\mathrm{op}}$$ is fully faithful. Since both of the categories involved above admit small colimits and $M'$ preserves them, by the adjoint functor theorem, it follows that $M'$ has a right adjoint $$D: \mathrm{RMod}_R^{\mathrm{op}} \to (\mathrm{Sp}^\mathrm{U-}_k)_{\ge 0}.$$ 
To prove that $M'$ is fully faithful, it is enough to prove that the unit map $\mathrm{id} \to D \circ M' $ is an equivalence. By \cref{genbylim}, \cref{qin}, and the fact that $D$ preserves small limits, it suffices to show that the natural map $\mathbb{G}_a \to D(M'(\mathbb{G}_a))$ is an isomorphism. To this end, note that for any $Z \in \left(\mathrm{Sp}^\mathrm{U-}_k\right)_{\ge 0}$, we have
$$\mathrm{Map}(Z, D(M'(\mathbb{G}_a))) \simeq \mathrm{Map}_{R} (R, M'(Z)) \simeq \mathrm{Map}(Z, \mathbb{G}_a),$$ where the first isomorphism follows from adjunction and the second one follows from the construction of $M'$. This shows that the desired map $\mathbb{G}_a \to D(M'(\mathbb{G}_a))$ is an isomorphism, finishing the proof.
\end{proof}{}
\newpage
  
\section{Coniveau filtration on unipotent homology}\label{section3}

\subsection{Graded pieces of the coniveau filtration and unipotent local homology}
 Let $k$ be a field and $\mathrm{Sch}_k$ denote the category of $ k $-schemes. 
 For any scheme $X$ of dimension $n$, we will discuss a filtration on $H^{\uu}_*(X)$ constructed by To\"en \cite[Definition 3.1]{Toee23}. 
 We explain how the filtration on $H^{\uu}_*(X)$ is related to the coniveau filtration (which leads to the Cousin complex) on cohomology. This perspective is then used to directly deduce certain desired properties of the filtration on $H^\uu_*(X)$ from well-known properties of the coniveau filtration and local cohomology.

\begin{definition}[Coniveau filtration]\label{filtrationdeff}
    Let $Z_i$ be the set of closed subschemes of $X$ of codimension $\ge i.$ We say $C_1 \le C_2$ for $C_1, C_2 \in Z_i$ if $C_1 \subseteq C_2$. We will equip $Z_i$ with this partial order and view it as a category. First, define $$T_i:= \lim_{C\in Z_i^\mathrm{op}}H^\uu_{*} (X-C),$$ where the limit is taken in $\left(\mathbb{Z}-\mathrm{Mod}^\mathrm{U}_k \right)_{\ge 0}$. Note that $T_0 \simeq 0$ and $T_{n+1} \simeq H^\uu_*(X).$ Now let $$F_i:= T_{i+1}.$$ Then $F_i$ defines a finite, increasing filtration on $H^\uu_*(X)$, which we denote as $F^* H^\uu_*(X).$
\end{definition}{}

The main result of this section is an identification of $\mathrm{gr}^iH^\uu_*(X).$ 
Before delving into that, we discuss the relationship between the filtration of \cref{filtrationdeff} and the existing coniveau filtration on cohomology. Note that $$T^*_i:=R\mathrm{Hom} (T_i, \mathbb{G}_a) \simeq \mathrm{colim}_{C\in Z_i^\mathrm{}}R\Gamma (X-C, \mathcal{O}).$$There is a natural map $R\Gamma(X, \mathcal{O}) \to T_i^*.$ Let $$R_i:=\mathrm{colim}_{C\in Z_i^\mathrm{}} R\Gamma_C(X, \mathcal{O}),$$ where the latter denotes local cohomology. Note that $R_i$ equips $R\Gamma (X, \mathcal{O})$ with a decreasing finite filtration denoted as $F^*_{\mathrm{coniv}} R\Gamma (X, \mathcal{O})$, which one classically calls the \emph{coniveau filtration}. Note that $R_{n+1} \simeq 0$ and $R_0 \simeq R\Gamma (X, \mathcal{O}).$ The fiber sequence $$R\Gamma_C(X, \mathcal{O}) \to R\Gamma (X, \mathcal{O}) \to R\Gamma (X-C, \mathcal{O})$$ then induces a fiber sequence\begin{equation}\label{s}
  R_i \to R\Gamma (X, \mathcal{O}) \to T_i^*.  
\end{equation}
The following result is classical (e.g., see \cite{bloch-ogus}). 
\begin{proposition}\label{classical}
 In the above notation, $$R_i/R_{i+1} \simeq \mathrm{gr}^i_{\mathrm{coniv}} R\Gamma (X, \mathcal{O}) \simeq \bigoplus_{ x \in X^{(i)}} R\Gamma_x (X_x, \mathcal{O}),$$ where $X^{(i)}$ denotes the set of points of $X$ codimension $i.$   
\end{proposition}{}

The diagram below (where the horizontal arrows are fiber sequences; see \cref{s})
\begin{center}
\begin{tikzcd}
R_{i+1} \arrow[d] \arrow[r] & {R\Gamma(X, \mathcal O)} \arrow[r] \arrow[d] & T_{i+1}^* \arrow[d] \\
R_i \arrow[r]               & {R\Gamma(X, \mathcal O)} \arrow[r]           & T_i^*              
\end{tikzcd}
\end{center}induces a fiber sequence
\begin{equation}\label{eq1}
R_i/R_{i+1} \to 0 \to T_i^*/ T_{i+1}^*.    
\end{equation}
Further, we have a fiber sequence $F_{i-1} \to F_{i} \to \mathrm{gr}^iH^\uu_*(X).$ This gives a fiber sequence (we recall that $F_i = T_{i+1}$)
\begin{equation}\label{eq2}
 \mathrm{RHom}(\mathrm{gr}^iH^\uu_*(X), \mathbb{G}_a) \to  T_{i+1}^* \to T_i^*.  
\end{equation}
Combining \cref{eq1} and \cref{eq2}, we have
\begin{equation}\label{checkisom1}
 \mathrm{RHom}(\mathrm{gr}^iH^\uu_*(X), \mathbb{G}_a) \simeq R_i/R_{i+1}.   
\end{equation}{}

Now we are ready to prove the following.

\begin{proposition}\label{descriptionofgradedpiece}
 Let $X$ be a scheme over $k$ of dimension $n.$ 
 Then the associated graded of the coniveau filtration on the unipotent homology of $ X $ (\cref{filtrationdeff}) may be identified as 
$$\mathrm{gr}^i H^\uu_*(X) \simeq \prod_{x \in X^{(i)}} H^\uu_{*, x}(X_x),$$ where $X^{(i)}$ denotes the set of points of $X$ of codimension $i.$

\end{proposition}
\begin{proof}
We will first construct a natural map from the right hand side to the left hand side; this part follows the same reasoning as in Proposition 3.2 of \cite{Toee23}. Let $D \subseteq C$ be two closed subsets of $X$ such that $\mathrm{codim}_X (C) \ge i$ and $\mathrm{codim}_X (D) \ge i+1.$ Let $x \in (C-D)$ be such that $x \in X^{(i)};$ we will use $(C-D)^{(i)}$ to denote the set of such points. For an $x \in (C-D)^{(i)}$, it follows that $X_x - \left \{x \right \} \subseteq X-C.$ This gives the following commutative diagram in $\mathrm{Sch}_k:$
\begin{center}
\begin{tikzcd}
\coprod_{x \in (C-D)^{(i)}}X_x - \left \{x \right \} \arrow[d] \arrow[rr] &  & X-C \arrow[d] \\
\coprod_{x \in (C-D)^{(i)}}X_x \arrow[rr]                                 &  & X-D \,.         
\end{tikzcd}
\end{center} Since $(C-D)^{(i)}$ is finite by \emph{loc. cit.}, on applying unipotent homology and taking cofibers, we obtain a map
$$\prod_{x \in (C-D)^{(i)}} H^\uu_{*, x}(X_x) \to H^\uu_* (X-D)/ H^\uu_* (X-C).$$

Taking limits over pairs $D \subseteq C$ such that $C \in Z_i$ and $D \in Z_{i+1}$, we obtain the desired map
$$\prod_{x \in X^{(i)}} H^\uu_{*, x}(X_x) \to \mathrm{gr}^i H^\uu_*(X).$$
    
We need to check that the above map is an isomorphism. However, that follows from \cref{huseful}, \cref{chekisom}, \cref{checkisom1} and \cref{classical}.
\end{proof}{}

\begin{remark}\label{specseq}Let $X$ be a scheme over $ k $ of dimension $n.$ As a consequence of \cref{descriptionofgradedpiece}, we obtain the following (homological) spectral sequence converging to unipotent homology:

\begin{equation*}
 E_1^{p,q} = \prod_{x \in X^{(p)}} H_{p+q,x}^\uu (X_x) \implies H_{p+q}^\uu(X).    
\end{equation*}{}
     
\end{remark}{}

\begin{proposition}[Purity]\label{connect}
 Let $X$ be a Cohen--Macaulay scheme, i.e., for every $x \in X,$ the local ring $\mathcal{O}_{X,x}$ is Cohen--Macaulay. Then for any $x \in X^{(i)}$, the object $H^\uu_{*, x}(X_x) \in (\mathbb{Z}\mathrm{-Mod}^{\mathrm{U}}_k)_{\ge 0}$ is $i$-connective.
\end{proposition}{}

\begin{proof}
Using \cref{eat}, we have the following spectral sequence
$$ E_2^{p,q}:= \mathrm{Ext}^p (H^\uu_{q,x}(X_x), \mathbb{G}_a) \implies H^{p+q}_{x}(X_x, \mathcal{O}). $$ Since $X$ is Cohen--Macaulay and $x \in X$ is of codimension $i$, it follows that $$H^n_x (X_x, \mathcal{O})=0$$ for $n <i.$ Using the above spectral sequence and induction on $q,$ we obtain $$\mathrm{Hom}(H^\uu_{q,x}(X_x), \mathbb{G}_a)=0$$ for $q<i.$ Since $H^\uu_{q,x}(X_x)$ is unipotent, we have $H^\uu_{q,x}(X_x)=0$ for $q<i$, as desired.
\end{proof}

\subsection{Reformulation in terms of Beilinson $t$-structures}\label{sec3.2}
                                 
\cref{connect} above admits a slick reformulation in the language of Beilinson $t$-structure on filtered stable $\infty$-categories, which we recall below.
\begin{notation}
Suppose $ \mathcal{C} $ is an $ \infty $-category. 
Let $F^* \mathcal{C}:= \mathrm{Fun}((\mathbb{Z},\leq), \mathcal{C}).$ We think of $F^*\mathcal{C}$ as the category of increasing $\mathbb{Z}$-indexed filtered objects.   
\end{notation}{} 
\begin{definition} [{\cite{MR4043467,MR923133}}]
 Let $\mathcal{C}$ be a stable $\infty$-category equipped with a $t$-structure. 
 Define $(F^*\mathcal{C})_{\ge 0}$ to be the full subcategory of $F^*\mathcal{C}$ spanned by objects $U$ such that for all $i$, $\mathrm{gr}^iU \in \mathcal{C}$ is $i$-connective. Define $(F^*\mathcal{C})_{\le 0}$ to be the full subcategory of $F^*\mathcal{C}$ spanned by objects $U$ such that for all $i$, $\mathrm{gr}^iU \in \mathcal{C}$ is $i$-coconnective. Then the pair $((F^*\mathcal{C})_{\ge 0}, (F^*\mathcal{C})_{\le 0})$ defines a $t$-structure on $F^*\mathcal{C}$, which we refer to as the Beilinson $t$-structure on $F^*\mathcal{C}$.   
\end{definition}{}                          
\begin{notation}
The truncation functors with respect to the Beilinson $t$-structure will be denoted as $\tau_{\ge n}^\mathrm{B}$ and $\tau_{\le n}^\mathrm{B}$. The functor $\tau_{\le n}^\mathrm{B} \tau_{\ge n}^\mathrm{B}$ will be denoted by $\pi_n^\mathrm{B}.$    
\end{notation}{}

 We will apply this construction to $\mathbb{Z}\mathrm{-Mod}^\mathrm{U-}_k$ equipped with the $t$-structure from \cref{ts}.                                                     \begin{proposition}\label{filtrationuse}
 Let $X$ be a Cohen--Macaulay scheme. Then $$F^* H^\uu_{*}(X) \in F^* \mathbb{Z}\mathrm{-Mod}^\mathrm{U-}_k $$ is connective with respect to the Beilinson $t$-structure.
\end{proposition}{}                 
\begin{proof}
 Follows from \cref{descriptionofgradedpiece} and \cref{connect}.   
\end{proof}

\begin{remark}\label{whynotaddalabel} Note that for a scheme $X$ of dimension $n$, the filtered object $\pi_0^\mathrm{B} (F^* H^\uu_* (X))$ can be identified as a chain complex of unipotent group schemes, equipped with its naive increasing filtration (see \cite[Tag~12.15~(2)]{stacks-project}). This identifies with the chain complex $E^{\bullet, 0}_1$ from \cref{specseq}. Concretely, the complex is the following
\begin{equation}\label{grothendieck}
  0 \to \prod_{x \in X^{(n)}} H^\uu_{n,x}(X_x) \to  \prod_{x \in X^{(n-1)}} H^\uu_{n-1,x}(X_x) \to \ldots \to \prod_{x \in X^{(0)}} H^\uu_{0,x}(X_x) \to 0.
\end{equation}{}
\end{remark}{}
\begin{notation}\label{ntn:filtered_generalized_loc_jacobian}
 The chain complex in \cref{grothendieck} will be denoted by $J_*^\uu(X).$ 
 It can be viewed as a filtered object by equipping it with the naive filtration, which we will denote by $F^*J_*^\uu(X)$. 
\end{notation}

 In the situation when $X$ is Cohen--Macaulay, since $F^* H^\uu_*(X)$ is connective, we obtain a map of filtered objects
\begin{equation}\label{filteredmap}
  F^* H^\uu_*(X) \to F^*J_*^\uu(X).  
\end{equation}
Now let $G$ be a commutative unipotent group scheme over $k.$ 
The identification $R\Gamma (X, G) \simeq R\mathrm{Hom}(H^\uu_*(X), G )$ of \cref{dual} and the filtration $F^* H^\uu_* (X)$ allow us to endow $R\Gamma (X, G)$ with a decreasing filtration. 
Composing with \cref{filteredmap}, we obtain a map of filtered objects
$$R\mathrm{Hom}(F^*J_*^\uu(X), G ) \to R\mathrm{Hom}(F^*H^\uu_*(X), G ).$$  Let $D(\mathrm{Uni})$ denote the derived category of the abelian category of unipotent commutative group schemes over $k.$ We have a natural map of filtered objects
$$R\mathrm{Hom}_{D(\mathrm{Uni})}(F^*J_*^\uu(X), G ) \to R\mathrm{Hom}(F^*J_*^\uu(X), G ).$$ This induces a map of filtered objects
\begin{equation}\label{filunder}
R\mathrm{Hom}_{D(\mathrm{Uni})}(F^*J_*^\uu(X), G ) \to   R\mathrm{Hom}(F^*H^\uu_*(X), G ).  
\end{equation}{}

\subsection{Cohomology of Cohen--Macaulay schemes}\label{sec3.3}

We will prove that at the level of underlying objects, \cref{filunder} induces an isomorphism. More precisely,

\begin{proposition}\label{animpresult}
 Let $X$ be a Cohen--Macaulay scheme over $k$. For any commutative unipotent group scheme $G$ over $k$, we have an isomorphism (induced by \cref{filunder})
\begin{equation}
R\mathrm{Hom}_{D(\mathrm{Uni})}(J_*^\uu(X), G ) \xrightarrow{\sim}   R\mathrm{Hom}(H^\uu_*(X), G ) \xrightarrow{\sim} R\Gamma (X, G) 
\end{equation}
\end{proposition}{}                                                         Only the left isomorphism needs to be proven since the other one follows from \cref{dual}. We first note the following lemmas.

\begin{lemma}\label{lw0}
    Let $X$ be a Cohen--Macaulay scheme over $k.$  Then for any $i \ge 0$, $x \in X^{(i)}$ and any commutative unipotent group scheme $G$ over $k$, we have
    $$H^i_x(X_x, {G}) \simeq \mathrm{Hom}\left(H^\uu_{i,x}(X_x), {G}\right)$$ and $H^t_x(X_x, G)=0$ for $t<i.$
\end{lemma}

\begin{proof}
By \cref{dual2}, we have $ \mathrm{RHom}(H^\uu_{*, x}(X_x), G) \simeq R\Gamma_{x}(X_x, G).$ The claim now follows directly from \cref{connect}.    
\end{proof}

\begin{proposition}\label{lw}
  Let $X$ be a Cohen--Macaulay scheme over $k.$  Then for any $i \ge 0$, $x \in X^{(i)}$ and any commutative unipotent group scheme $G$ over $k$, we have 
$$H^{i+j}_x(X_x, G)=0 $$ for $j \ge 2.$  
\end{proposition}

\begin{proof}
First we prove the result when $G$ satisfies the property that $V_G=0$. Then there exists a short exact sequence
\begin{equation}\label{fckd1}
0 \to G \to \prod_I \mathbb{G}_a \to \prod_J \mathbb{G}_a \to 0,    
\end{equation}
where $I$ and $J$ are possibly infinite index sets. Note that $$H^{i+j}_x\left(X_x, \prod_I \mathbb{G}_a\right) \simeq \prod_I H^{i+j}_x(X_x, \mathbb{G}_a)=0 $$ for $j \ge 1$ (and similarly for $J$).
The desired claim now follows from a long exact sequence chase.  

Suppose now that $G$ satisfies $V_G^\ell=0$ for some $ \ell \geq 2.$ 
The claim in this case follows from induction on $ \ell $ and chasing the long exact sequence on cohomology associated to the short exact sequence $0 \to VG \to G \to G/VG \to 0.$

Now for a general commutative unipotent group scheme $G,$ we have $G \simeq \varprojlim G/V^\ell G.$ By the previous paragraph, for $j \ge 2$, we have $H^{i+j}_x (X_x, G/V^\ell G)=0$ for all $ \ell .$ Note that for $j \ge 2,$ the induced maps 
$$H^{i+j-1}_x(X_x, G/V^{\ell} G) \to H^{i+j-1}_x(X_x, G/V^{\ell-1} G)$$ are surjective, since $H^{i+j} _x(X_x, V^{\ell-1}G/V^\ell G )=0.$ In particular, $R^1\varprojlim H^{i+j-1}_x(X_x, G/V^{\ell} G)=0.$ The claim in the lemma now follows from Milnor sequences.
\end{proof}

\begin{proposition}\label{micro1}
  Let $X$ be a Cohen--Macaulay scheme over $k.$  Then for any $i \ge 0$, $x \in X^{(i)}$ and any commutative unipotent group scheme $G$ over $k$, we have a natural isomorphism
$$H^{i+1}_{x}(X_x, G) \xrightarrow{\sim} \mathrm{Ext}^1 \left(H^\uu_{i,x}(X_x), G\right).$$
\end{proposition}                                                                                \begin{proof}
  By \cref{dual2}, we have $ \mathrm{RHom}(H^\uu_{*, x}(X_x), G) \simeq R\Gamma_{x}(X_x, G).$ By \cref{connect}, $H^\uu_{*,x}(X_x)$ is $i$-connective. This gives a natural truncation map $H^\uu_{*,x}(X_x) \to H^\uu_{i,x}(X_x)[i].$ Thus, we have a natural map
\begin{equation*}
\mathrm{RHom}\left(H^\uu_{i, x}(X_x), G[-i]\right) \to \mathrm{RHom}\left(H^\uu_{*, x}(X_x), G\right).    
\end{equation*}
This induces natural maps
$$\theta_j: \mathrm{Ext}^{j} \left(H^\uu_{i,x}(X_x), G\right) \to H^{i+j}_{x}(X_x, G).$$ By \cref{lw0}, the map $\theta_0$ is an isomorphism. We would like to show that $\theta_1$ is an isomorphism. Note that by construction and a long exact sequence chase, it follows that $\theta_1$ is injective. It is thus an isomorphism when $G = \mathbb{G}_a$, since the target of $\theta_1$ vanishes in this case. It follows that $\theta_1$ is also an isomorphism when $G= \prod_I \mathbb{G}_a$, where $I$ is a possibly infinite index set. In particular, $\mathrm{Ext}^1 (H^\uu_{i,x}(X_x), \prod_I \mathbb{G}_a )=0.$ Now we pause to prove the following lemma.

\begin{lemma}\label{fckd}
 Let $X$ be a Cohen--Macaulay scheme over $k.$ Then for any $i \ge 0,$ $x \in X^{(i)}$ and any commutative unipotent group scheme $G$ over $k$, we have $$\mathrm{Ext}^j_{D(\mathrm{Uni})}(H^\uu_{i,x}(X_x), G)=0$$ for $j \ge 2.$  
\end{lemma}

\begin{proof}
 Since $H^\uu_{i,x}(X_x)$ is unipotent, by \cite[Proposition~V-1 5.1 and 5.2]{Demazure:1970}, we have $\mathrm{Ext}^j (H^\uu_{i,x}(X_x), \mathbb{G}_a)=0$ for $j \ge 2.$ It follows that for any index set $I$, we have $\mathrm{Ext}^j \left(H^\uu_{i,x}(X_x), \prod_I \mathbb{G}_a\right)=0$ for $j \ge 2.$ 
 If $G$ is such that $V_G=0,$ observe that  the short exact sequence \cref{fckd1} induces a long exact sequence on Ext groups. 
 By \cite[Proposition~V-1 5.1 and 5.2]{Demazure:1970}, we have $\mathrm{Ext}^1 \left(H^\uu_{i,x}(X_x), \prod_I \mathbb{G}_a \right)=0$; applying this to the aforementioned long exact sequence and using exactness gives the desired result. 
 
 Suppose now that $G$ satisfies $V_G^\ell =0$ for some $\ell \geq 2.$ The desired vanishing follows inductively from the short exact sequence $0 \to VG \to G \to G/VG \to 0$ and the argument of the previous paragraph. 
 
 For a general unipotent group scheme $G,$ we have $G \simeq \varprojlim G/V^\ell G.$ Since we have proven the statement of the lemma for unipotent group schemes killed by a power of $V$, it follows from a long exact sequence chase that the maps 
$$\mathrm{Ext}^{j-1}\left(H^\uu_{i,x}(X_x), G/V^\ell\right) \to \mathrm{Ext}^{j-1}\left(H^\uu_{i,x}(X_x), G/V^{\ell-1}\right)$$ are surjective for $j \ge 2.$ Similarly to \cref{lw}, by a Milnor sequence argument, we obtain the desired vanishing in general.
\end{proof}{}

We return to the proof of \cref{micro1}. 
Proceeding in a manner similar to the proof of \cref{lw} using \cref{fckd1} shows that $\theta_1$ is an isomorphism when $G$ has the property $V_G=0.$ Suppose now that $G$ satisfies $V_G^\ell =0$ for some $\ell \geq 2.$ The long exact sequences associated to the short exact sequence $0 \to VG \to G \to G/VG \to 0$ along with \cref{fckd} (which implies that the map $\mathrm{Ext}^1(H^\uu_{i,x}(X_x),G) \to \mathrm{Ext}^1 (H^\uu_{i,x}(X_x),G/VG)$ is surjective) and five lemma implies that $\theta_1$ is an isomorphism in that case. The case of a general unipotent group scheme follows from the fact that $G \simeq \varprojlim G/V^\ell G$ and using Milnor sequences along with the five lemma.
\end{proof}{}                                                             \begin{lemma}\label{b1}
Let $X$ be a Cohen--Macaulay scheme over $k.$  Then for any $i \ge 0$ and any finite type commutative unipotent group scheme $G$ over $k$, we have
$$\bigoplus_{x \in X^{(i)}}H^i_x(X_x, {G}) \simeq  \mathrm{Hom}\biggl(\prod_{x \in X^{(i)}} H^\uu_{i,x}(X_x), {G} \biggl).$$    
\end{lemma}    
\begin{proof}
 Follows from \cref{useful} and \cref{lw0}.   
\end{proof}{}
\begin{lemma}\label{b2}
      Let $X$ be a Cohen--Macaulay scheme over $k.$  Then for any $i \ge 0$ and any finite type commutative unipotent group scheme $G$ over $k$, we have a natural isomorphism
$$\bigoplus_{x \in X^{(i)}}H^{i+1}_{x}(X_x, G) \xrightarrow{\sim}  \mathrm{Ext}^1 \biggl(\prod_{{x \in X^{(i)}}} H^\uu_{i,x}(X_x), G \biggl).$$
\end{lemma}

\begin{proof}
Follows from \cref{useful} and \cref{micro1}.    
\end{proof}

\begin{lemma}\label{fckd2}
 Let $X$ be a Cohen--Macaulay scheme over $k.$ Then for any $i \ge 0$ and any finite type commutative unipotent group scheme $G$ over $k$, we have $$\mathrm{Ext}^j_{D(\mathrm{Uni})}\biggl(\prod_{x \in X^{(i)}}H^\uu_{i,x}(X_x), G\biggl)=0$$ for $j \ge 2.$  
\end{lemma}

\begin{proof}
 Since $\prod_{x \in X^{(i)}}H^\uu_{i,x}(X_x)$ is unipotent, by \cite[Proposition~V-1 5.1 and 5.2]{Demazure:1970}, $$\mathrm{Ext}^j \biggl(\prod_{x \in X^{(i)}}H^\uu_{i,x}(X_x), \mathbb{G}_a \biggl)=0$$ for $j \ge 2.$ Suppose now that $V_G=0.$ Since $G$ is finite type, we have an exact sequence of the form $$0 \to G \to \prod_I \mathbb{G}_a \to \prod_J \mathbb{G}_a \to 0,$$ where $I, J$ are finite sets. By \cref{b2}, we have $\mathrm{Ext}^1 (\prod_{{x \in X^{(i)}}} H^\uu_{i,x}(X_x), \mathbb{G}_a)=0.$ Thus, by a long exact sequence chase we obtain the desired claim when $V_G=0.$ Since $G$ is finite type, $V_G^n=0$ for some $n$. Thus, to prove the claim in general, by using the short exact sequence $0 \to VG \to G \to G/VG \to 0,$ one may reduce to the case when $V_G=0.$ This finishes the proof.
\end{proof}

Finally, we are ready to give a proof of \cref{animpresult}.
\begin{proof}[Proof of \cref{animpresult}]
As discussed before the statement of \cref{animpresult}, there is a map of filtered objects $$R\mathrm{Hom}_{D(\mathrm{Uni})}(F^*J_*^\uu(X), G ) \to R\mathrm{Hom}(F^*H^{\uu}_*(X), G ).$$Note that the filtered object on the left induces a convergent spectral sequence
$$'E_1^{i,j}= \mathrm{Ext}^j_{D(\mathrm{Uni})}\biggl(\prod_{x \in X^{(i)}}H^\uu_{i,x}(X_x), G\biggl) \implies \mathrm{Ext}^{i+j}_{D(\mathrm{Uni})} (J^\uu_* (X), G).$$The filtered object on the right induces a convergent spectral sequence
$$E_1^{i,j}= \mathrm{Ext}^{i+j}(\mathrm{gr}^i H^\uu_*(X), G) \implies \mathrm{Ext}^{i+j} (H^\uu_* (X), G).$$ The map of filtered objects induces natural morphisms between the above two spectral sequences, and to prove the proposition it suffices to prove that the natural maps $'E_1^{i,j} \to E_1^{i,j}$ are isomorphisms. Using \cref{huseful} and \cref{descriptionofgradedpiece}, it follows that we need to prove that the natural maps 
$$\mathrm{Ext}^j_{D(\mathrm{Uni})}\biggl(\prod_{x \in X^{(i)}}H^\uu_{i,x}(X_x), G \biggl) \to \bigoplus_{x\in X^{(i)}}H^{i+j}_{x} (X_x, G)$$ are isomorphisms. 

To this end, note that for $j \le 0,$ the map is an isomorphism by \cref{lw0}. For $j=1$, the isomorphism follows from \cref{b2}, and for $j \ge 2$, it follows from the vanishings from \cref{lw} and \cref{fckd2}.
\end{proof}{}            

\section{Artin--Mazur formal groups}\label{section4}                 
In \cite{soon}, the authors explained how to recover Artin--Mazur formal groups from the unipotent homotopy group schemes introduced in loc. cit. under certain hypotheses on vanishing of cohomology groups. 
In this section, we will explain how to recover these Artin--Mazur formal groups in general from unipotent homology groups studied in this paper, with \emph{no such vanishing assumptions}.  To this end, let us recall their definition.

\begin{definition}[{\cite{MR457458}}]\label{artinmazur}
Let $k$ be a field and $X$ be a smooth proper scheme over $k$. Let $\mathrm{Art}_k$ be the category of Artinian $k$-algebras. Define an abelian group valued functor $\Phi^n_X \colon (\mathrm{Art}/k)^{\mathrm{op}} \to \mathrm{Ab}$ as
\[ A \mapsto \mathrm{Ker}\bigl(H^n_{\mathrm{\acute{e}t}}(X_A,\mathbb{G}_m) \to H^n_{\mathrm{\acute{e}t}}(X,\mathbb{G}_m)\bigr). \]   
\end{definition}{}
Note that when $n=1$, $\Phi^n_X$ is the formal completion of the Picard scheme of $X$. When $n=2$, it is the formal Brauer group. In general, the above functor is not pro-representable. Artin and Mazur gave certain conditions regarding pro-representability for this functor. 
Recently, Bragg--Olsson gave a new proof \cite[Theorem~10.8]{BOL} of the following result of Raynaud \cite[Proposition 2.7.5]{MR563468}. 
\begin{theorem}[Raynaud, Bragg--Olsson]\label{braols}
Let $X$ be a smooth proper scheme over $k.$ Let $(\Phi^n_X)^\mathrm{fl}$ denote the sheafification of $\Phi^n_X$ for the fppf topology on $\mathrm{Art}_k^\mathrm{op}.$ Then $(\Phi^n_X)^\mathrm{fl}$ is pro-representable for every $n.$ 
\end{theorem}{}
Following Bragg--Olsson, we will refer to $(\Phi^n_X)^\mathrm{fl}$ as the $n$-th flat Artin--Mazur formal group. We will actually recover the \emph{flat} Artin--Mazur formal groups from unipotent homology. Recall that (\cref{specseq}) for a $ k $-scheme $X$, we have the following (homological) spectral sequence (arising from the coniveau filtration of \cref{filtrationdeff}) converging to its unipotent homology:
\begin{equation}\label{fromres}
 E_1^{p,q} = \prod_{x \in X^{(p)}} H_{p+q,x}^\uu (X_x) \implies H_{p+q}^\uu(X).    
\end{equation}{}

We will prove the following.
\begin{theorem}\label{isitmainthm}
  Let $X$ be a smooth proper scheme over a perfect field $k$ of positive characteristic. 
  Then the Cartier dual of the flat Artin--Mazur formal group $(\Phi_X^p)^\mathrm{fl}$ is canonically isomorphic to the unipotent group scheme $E_2^{p,0}$ obtained by turning the page of the spectral sequence \cref{fromres}. 
\end{theorem}

\begin{proof}
 Note that $E_2^{p,0},$ by definition, is the $p$-th homology of the chain complex $E_1^{\bullet, 0}$ of unipotent group schemes, which we denoted by $J^u_* (X)$ (see \cref{ntn:filtered_generalized_loc_jacobian}).  
 We will begin by computing the Dieudonn\'e module of $E^{p,0}_2$, which is given by $\varinjlim \mathrm{Hom}\left(E_2^{p,0}, W_n\right)$. By \cref{animpresult}, we have 
$$R\mathrm{Hom}_{D(\mathrm{Uni})}\left(J_*^\uu(X), W_n \right)\xrightarrow{\sim} R\Gamma (X,W_n) .$$ This yields a spectral sequence where we may identify the $ E_2 $-page: 
$$ \mathrm{Ext}^q _{D(\mathrm{Uni})}\left(E_2^{p,0}, W_n\right) \implies H^{q+p}(X, W_n).$$ Since $ \mathrm{Ext}^q _{D(\mathrm{Uni})}\left(E_2^{p,0}, W_n\right) =0$ for $q \ge 2$, the spectral sequence degenerates on the second page, and we obtain exact sequences 

$$0 \to \mathrm{Ext}^1 \left(E_2^{p-1,0}, W_n\right) \to H^p (X, W_n) \to \mathrm{Hom}\left(E_2^{p,0}, W_n\right) \to 0.$$
Compatibility of these exact sequences for varying $n$ implies that we have an exact sequence
$$0 \to \varinjlim_n \mathrm{Ext}^1 \left(E_2^{p-1,0}, W_n\right) \to \varinjlim_n H^p (X, W_n) \to \varinjlim_n \mathrm{Hom}\left(E_2^{p,0}, W_n\right) \to 0.$$
By \cref{usethislemma}, we obtain 
$$\varinjlim_n H^p (X, W_n) \simeq \varinjlim_n \mathrm{Hom}\left(E_2^{p,0}, W_n\right).$$ Note that the flat Artin--Mazur formal group $(\Phi^p_X)^{\mathrm{fl}}$ is connected, so its Cartier dual $((\Phi^p_X)^{\mathrm{fl}})^\vee$ is a commutative unipotent group scheme over $ k $. Further, by \cite[Theorem~12.1]{BOL} (cf. \cite[Proposition 8.1]{MR800174}), we have
$$\varinjlim_n \mathrm{Hom} \left(((\Phi^p_X)^{\mathrm{fl}})^\vee, W_n\right) \simeq \varinjlim_{n} H^p (X, W_n). $$
Therefore, by Dieudonn\'e theory, we conclude that $$ ((\Phi^p_X)^{\mathrm{fl}})^\vee \simeq E^{p,0}_2,$$ which finishes the proof.
\end{proof}{}
The following lemma was used in the above proof.

\begin{lemma}\label{usethislemma}
Let $G$ be a commutative unipotent group scheme over a perfect field $k$. 
Then $$\varinjlim_n \mathrm{Ext}^1(G, W_n) =0.$$ 
\end{lemma}

\begin{proof}
Suppose that $\gamma \in \mathrm{Ext}^1 (G, W_t)$ for some $t \in \mathbb{N}.$ Suppose that $ \gamma $ is classified by an extension 
\begin{equation}\label{ifil}
    0 \to W_t \to H \to G \to 0.
\end{equation}Note that on the category of commutative unipotent group schemes over $ k $, the functor $\varinjlim_{n} \mathrm{Hom}(\cdot, W_n)$ is exact \cite[V-1 Théorème 4.3 b)]{Demazure:1970}. 
Applying this to \cref{ifil}, we obtain an exact sequence
\begin{equation}\label{ifil2}
 0 \to \varinjlim_{n}\mathrm{Hom}(G, W_n) \to  \varinjlim_{n}\mathrm{Hom}(H, W_n) \to  \varinjlim_{n}\mathrm{Hom}(W_t, W_n)\to 0.   
\end{equation}
The exactness implies that there exists a map $v: H \to W_s$ for some $s > t$ such that the composition $W_t \to H \to W_s$ is the canonical map. The class $\gamma$ induces a class in $\mathrm{Ext}^1 (G, W_s)$ which can be described as an exact sequence
\begin{equation}\label{ifil3}
 0 \to W_s \to H' \to G \to 0,   
\end{equation}
where $H'$ is the pushout  of $W_t \to H$ along the canonical map $W_t \to W_s$. Using the pushout description of $ H' $, the map $v : H \to W_s$ induces a map $H' \to W_s$ which splits the exact sequence \cref{ifil3}. This proves that $\varinjlim_n \mathrm{Ext}^1(G, W_n) =0$, as desired.
\end{proof}{}

\newpage

\section{Perfect unipotent spectra and duality theorems}\label{section5}
In this section, we fix a perfect field $ k $ of characteristic $ p>0 $ and study perfect unipotent spectra over $ k $. 
In \cref{subsection:perfect_qft_group_schemes}, we study a certain finiteness condition on perfect group schemes, leading to the notion of quasi-finite type perfect group schemes. 
We define perfect unipotent spectra in \cref{subsection:perfect_unipotent_spectra} and record a recognition theorem in \cref{subsection:perfect_uni_spectra_recognition}, in parallel with the results of \cref{subsection:recognition_thm}. 
In \cref{subsection:duality_perfect_unipotent_spectra}, we use the finiteness condition described in \cref{subsection:perfect_qft_group_schemes}, and prove that a certain  full subcategory $ \mathbb{Z}/p $-modules in perfect unipotent spectra admits a good theory of duality (\cref{duality}). 
In \cref{subsection:syntomic_duality}, we show that the weight $ i $ syntomic cohomology (modulo $p$) of a proper lci scheme $ X $ admits a canonical enhancement $ \mathbb{Z}/p(i)_X^{\mathrm{uni}} $ to perfect unipotent spectra over $ k $ for each $ i $; if moreover $ X $ is smooth, the unipotent spectrum $ \mathbb{Z}/p(i)_X^{\mathrm{uni}} $ has good finiteness properties. 
Finally, we show that when $ X $ is a smooth proper $ k $-scheme of dimension $ d $, there is an equivalence of unipotent spectra $ \mathbb{Z}/p(i)^{\mathrm{uni}}_X \simeq ({\mathbb{Z}/p(d-i)^{\mathrm{uni}}_X})^\vee [-2d] $ (under the duality from \cref{duality}) refining Milne's duality theorem \cite[Theorem 1.9]{Milne111}. 

\subsection{Preliminaries on quasi-finite type perfect group schemes}\label{subsection:perfect_qft_group_schemes}
In this section, we discuss the foundations on quasi-finite type perfect group schemes, which will be used later in the context of perfect unipotent spectra. 

\begin{definition}
    Let $G$ be an affine group scheme over a perfect field $k$ of positive characteristic. The \emph{perfection} of $G$ is defined to be $G^{\mathrm{perf}}:= \varprojlim_{\varphi} G$; it is a perfect affine group scheme over $k$. 
\end{definition}

\begin{definition}[Quasi-finite type perfect group schemes]\label{qftgp}
A perfect affine group scheme $G$ over $k$ is called \emph{quasi-finite type} if $G$ is the perfection of some finite type group scheme over $k$.    
\end{definition}

\begin{remark}
Note that the category of commutative group schemes over $k$ is an abelian category. The full subcategory of perfect commutative group schemes over $k$ forms an abelian subcategory of the former. As we will prove in \cref{rmk:quasi_ft_perf_gp_sch_are_abelian_cat}, the full subcategory of commutative group schemes spanned by perfect commutative quasi-finite type group schemes also naturally forms an abelian category.  
\end{remark}
The following proposition will give an intrinsic reformulation of \cref{qftgp}.

\begin{proposition}\label{cocom} 
A perfect affine group scheme $G$ over $k$ is quasi-finite type if and only if $G$ is a cocompact object in the category of perfect affine group schemes over $k$.
\end{proposition}{}
\begin{proof}
Let $H \simeq H_0^{\mathrm{perf}}$ where $H_0$ is a finite type affine group scheme over $ k $. We will show that $H$ is a cocompact object in the category of perfect affine group schemes. Let $G \simeq \varprojlim G_i$ in the category of perfect group schemes. By adjunction, we have
$\mathrm{Hom}(G, H) \simeq \mathrm{Hom}(G, H_0)$. Since $H_0$ is finite type, we have $\mathrm{Hom}(G, H_0) \simeq \varinjlim_i \mathrm{Hom}(G_i, H_0) \simeq \varinjlim_i\mathrm{Hom}(G_i, H)$, as desired.

Conversely, we will show that a cocompact object $G$ is quasi-finite type. 
One can write $ G $ as a cofiltered limit $G \simeq \varprojlim G^0_{i}$, where each $G^0_{i}$ is a finite type quotient of $G$. 
Passing to perfection induces an equivalence $G \simeq \varprojlim G_i$, where $G \to G_i$ is a surjection of perfect group schemes for each $i$. Since $G$ is cocompact, we have
$$\mathrm{Hom}(G, G) \simeq \varinjlim_i \mathrm{Hom}(G_i, G).$$This implies that there exists a map $G_i \xrightarrow{f} G$ such that the composition $G \to G_i \xrightarrow{f} G $ is the identity map. It follows that the surjection $G \to G_i$ must also be an injection. Thus $G \simeq G_i \simeq (G^0_i)^{\mathrm{perf}}$, which finishes the proof since $G^0_i$ was finite type by choice.
\end{proof}{}

\begin{corollary}\label{cccc}
 A perfect affine group scheme $G$ over a field $k$ is quasi-finite type if and only if $G$ is the perfection of a finite type quotient of $G$.   
\end{corollary}

\begin{proof}
Follows from \cref{cocom} and its proof.    
\end{proof}

\begin{remark}\label{c1}
Note that a cocompact object in the category of perfect, affine, commutative group schemes over $k$ can be equivalently regarded as a perfect, affine, quasi-finite type group scheme over $k$ that also happens to be commutative. This follows in a manner similar to the proof of \cref{cocom}. Further, by \cref{cccc}, any such group scheme $G$ is isomorphic to perfection of a finite type quotient of $G$, which is necessarily commutative. 
\end{remark}{}

\begin{remark}\label{frobpower}
Let $G$ and $H$ be two perfect quasi-finite type group schemes over $k$. Suppose that $G\simeq G_0^{\mathrm{perf}}$ and $ H \simeq H_0^{\mathrm{perf}}$. Then it follows that 
\begin{equation}\label{indpro}
 \mathrm{Hom}(G, H) \simeq  \mathrm{Hom} (G, H_0) \simeq \varinjlim_{\varphi} \mathrm{Hom}(G_0, H_0 ).   
\end{equation} In other words, for every $f: G \to H$, there is a $k \ge 0$ such that $f \varphi^k$ is induced from $f_0: G_0 \to H_0$ via perfection.   
\end{remark}{}

\begin{example}\label{basicex}
The perfection of $\mathbb{G}_a$ (resp.~$\mathbb{G}_m$ ) is a group scheme denoted by $\mathbb{G}_a^{\mathrm{perf}}$, whose underlying scheme is isomorphic to $\mathrm{Spec}\, k[x^{1/p^\infty}]$ (resp.~$\mathrm{Spec}\, k[x^{\pm 1/p^\infty}]$). By definition, $\mathbb{G}_a^{\mathrm{perf}}$ (resp.~$\mathbb{G}_m^{\mathrm{perf}}$ ) is a perfect quasi-finite type group scheme.   
\end{example}{}

\begin{example}
 The profinite group scheme $\underline{\mathbb{Z}_p}:= \lim_n \underline{\mathbb{Z}/p^n}$ is a perfect group scheme, but not of quasi-finite type. 
\end{example}

\begin{example}\label{basicex2}
Let $\mu_n:= \mathbb{G}_m [n].$ If $n$ is a power of $p$, it follows that $(\mu_n)^{\mathrm{perf}} \simeq *$. If $n$ is coprime to $p$, then $(\mu_n)^{\mathrm{perf}} \simeq \mu_n$.  
\end{example}{}
\begin{example}\label{basicex3}
The group scheme $ \alpha_p = \mathbb{G}_a[p] $ satisfies $ \alpha_p^{\mathrm{perf}} \simeq 0 $. 
\end{example}{}

\begin{proposition}\label{uniqf}
Let $G$ be a perfect affine group scheme over a perfect field $k$. Then $G$ is unipotent and quasi-finite type (resp.~commutative) if and only if $G$ is the perfection of some unipotent and finite type (resp.~commutative) group scheme.    
\end{proposition}{}
\begin{proof}
Follows from \cref{cccc} and \cref{c1} since the category of unipotent group schemes is closed under inverse limits and quotients. 
\end{proof}

\begin{lemma}\label{exactperfection}
 Let $G$ be a finite type commutative affine group scheme over a perfect field $k$. Let $F_G$ denote the Frobenius map on $G$. Then $$R^1 \varprojlim _{F_G} G \simeq 0.$$ 
\end{lemma}{}

\begin{proof}   
Note that since $G$ is finite type, the image of $F_G^k$ stabilizes for  large enough $k$. Let $G':= \mathrm{Im}(F_G^k)$ for $k \gg 0$. It suffices to prove that $\displaystyle{R^1 \varprojlim_{F_{G'}} G'} \simeq 0$. However, by construction, $F_{G'}$ is a surjection. This finishes the proof.
\end{proof}{}

\begin{proposition}\label{rmk:quasi_ft_perf_gp_sch_are_abelian_cat}
The category of perfect commutative quasi-finite type group schemes over $k$ is an abelian subcategory of the category of commutative group schemes over $k$ that is closed under extensions.   
\end{proposition}

\begin{proof}
Let $G$ and $H$ be two perfect commutative quasi-finite type group schemes and let $f: G \to H$ be a map. For the abelian subcategory part, it will suffice to prove that kernel and cokernel of $f$ in the category of commutative affine group schemes is already perfect and of quasi-finite type. The perfectness follows directly. To prove that they are of quasi-finite type, by \cref{frobpower}, we can assume without loss of generality that $f$ is induced from a map $f_0: G_0 \to H_0$ of finite type algebraic groups via perfection. Since kernel and cokernel of $f_0$ are both finite type, our claim follows from \cref{exactperfection}. 

Now we will prove the closure under extension property. In what follows, we work in the derived category of fpqc abelian sheaves over $k$. By \cref{exactperfection}, it follows that ${R\varprojlim_{F_T} T \simeq T^{\mathrm{perf}}}$ for any finite type commutative affine group scheme $T$. Let us consider an extension $0 \to H \to E \to G \to 0$ of group schemes, where $G,H$ are perfect and of quasi-finite type. It follows directly that $E$ is also perfect. Our goal now is to prove that $E$ is of quasi-finite type. Suppose that $G_0, H_0$ are finite type group schemes such that $G \simeq G_0^{\mathrm{perf}}$ and $H \simeq H_0^{\mathrm{perf}}$. It follows that $R\mathrm{Hom}(G, H) \simeq R\varprojlim_{\varphi} R\mathrm{Hom}(G, H_0) \simeq  R\mathrm{Hom}(G, H_0)$, where the latter isomorphism follows because $G$ is perfect. Since $H_0$ is finite type, it further follows that 
$R\mathrm{Hom}(G, H_0) \simeq \varinjlim R\mathrm{Hom}(G_0, H_0)$. In particular, we have $\mathrm{Ext}^1 (G, H) \simeq \varinjlim \mathrm{Ext}^1 (G_0, H_0).$ Therefore, without loss of generality, we may assume that the extension $0 \to H \to E \to G \to 0$ arises as perfection of an extension $ 0 \to H_0 \to E_0 \to G_0 \to 0$. Thus $E \simeq E_0^{\mathrm{perf}}$, which finishes the proof since $E_0$ must be of finite type.
\end{proof}

\begin{corollary}
\label{rmk:uni_quasi_ft_perf_gp_sch_are_abelian_cat}
The category of perfect commutative quasi-finite type unipotent group schemes over $k$ is an abelian subcategory of the category of commutative group schemes over $k$ that is closed under extensions.    
\end{corollary}

\begin{proof}
Follows from \cref{uniqf} and \cref{rmk:quasi_ft_perf_gp_sch_are_abelian_cat}.    
\end{proof}

\begin{remark}
 When $k$ is algebraically closed, the category of quasi-finite type perfect group schemes is equivalent to the category of \emph{quasi-algebraic} group schemes due to Serre \cite[\S1]{MR118722}. This follows from \cite[Proposition~2]{MR118722}.  
\end{remark}

\begin{proposition}\label{filtrationonuni}
 Let $G$ be a perfect quasi-finite type commutative unipotent group scheme over a perfect field $k$. Then $G$ has a finite filtration where the graded pieces are all perfect, quasi-finite type, unipotent, closed subgroup schemes of $\mathbb{G}_a^{\mathrm{perf}}$.
\end{proposition}{}

\begin{proof}\label{structure} By \cref{uniqf}, $G$ is perfection of a finite type, unipotent, commutative, affine group scheme $G_0$. Any such $G_0$ has a finite filtration where the graded pieces are subgroup schemes of $\mathbb{G}_a$. The result follows from taking perfection.     
\end{proof}
\begin{corollary}
Let $G$ be a perfect quasi-finite type commutative unipotent group scheme over a perfect field $k$. Then $G$ is killed by a power of $p$.    
\end{corollary}{}
\begin{proof}
Follows from using the filtration in \cref{filtrationonuni}.    
\end{proof}

\begin{proposition}\label{filtrationuni1}
  Let $G$ be a perfect quasi-finite type commutative unipotent group scheme over an algebraically closed field $k$ of characteristic $p>0$. Then $G$ has a finite filtration where the graded pieces are isomorphic to either $\mathbb{G}_a^{\mathrm{perf}}$ or $\mathbb{Z}/p$.   
\end{proposition}
\begin{proof}
\label{structure} By \cref{uniqf}, $G$ is perfection of a finite type, unipotent, commutative, affine group scheme $G_0$ over $k$. Since $k$ is algebraically closed, any such $G_0$ has a filtration where the graded pieces are isomorphic to $\mathbb{G}_a,\, \alpha_p$, and $\mathbb{Z}/p$. The result follows from taking perfection and using \cref{basicex3}.
\end{proof}

\begin{proposition}[Galois descent] \label{galdescent}
 Let $G$ be a perfect affine group scheme over a perfect field $k$. Let $\overline{k}$ be an algebraic closure of $k$. Suppose that $G_{\overline{k}}:=G \times_{\mathrm{Spec}\,k} \mathrm{Spec}\,\overline{k}$ is quasi-finite type over $\overline{k}$. Then $G$ is quasi-finite type over $k$.  
\end{proposition}{}

\begin{proof}
Let us choose an isomorphism $f: G_ {\overline{k}} \, {\simeq} \,H^{\mathrm{perf}}$, where $H$ is a finite type group scheme over $\overline{k}$. By adjunction, $f$ is induced from a canonical map $f': G_{\overline{k}} \to H$ of group schemes over $ \overline{k} $. 
Using the spreading out technique, by possibly replacing $k$ by a finite extension, we can without loss of generality assume that the finite type group scheme $H$ is isomorphic to $H'_{\overline{k}}$ where $H'$ is a finite type group scheme defined over $k$. In the category of affine group schemes over $k$, $$\mathrm{Hom} (G_{\overline{k}}, H') \simeq  \mathrm{Hom}\left(\varprojlim_{[L: k] <\infty}G_L, H' \right)\simeq \varinjlim_{[L: k] <\infty}\mathrm{Hom} \left(G_L, H' \right),$$ where the latter isomorphism follows because $H'$ is a cocompact object, since it is a finite type affine group scheme over $k$. This implies that there exists a finite extension $L$ of $k$, and a map $f'_0: G_L \to H'_L$ such that $f'$ is the pullback of $ f'_0 $ along $\mathrm{Spec}\, \overline{k} \to \mathrm{Spec}\, L.$ Note that since $k$ is perfect, the finite extension $L$ is also perfect and $G_L$ is a perfect group scheme over $L$. Therefore, we have a map $f_0: G_L \to (H'_L)^{\mathrm{perf}} \simeq (H'^{\mathrm{perf}})_L$ which induces an isomorphism when base changed along $\mathrm{Spec}\, \overline{k} \to \mathrm{Spec}\, L $. Therefore, $f_0$ is an isomorphism. This implies that $G_L$ is a perfect, quasi-finite type group scheme over $L$.

We will now show that $G$ is a cocompact object in the category of perfect affine group schemes over $k$, which will imply that it is of quasi-finite type by \cref{cocom}. To this end, let $T \simeq \varprojlim T_i$, where $(T_i)_{i \in I}$ is an inverse system of perfect affine group schemes over $k$. Then we also have $T_L:= \varprojlim (T_i)_L.$ Note that
$$ \mathrm{Hom}_k(T, G) \simeq \mathrm{Hom}(T_L, G_L)^{\mathrm{Gal}(L/k)} \simeq \left (\varinjlim \mathrm{Hom}((T_i)_L, G_L) \right)^{\mathrm{Gal}(L/k)},$$ where the latter isomorphism follows from the fact that $G_L$ is a cocompact object in the category of perfect affine group schemes by the previous paragraph and \cref{cocom}. Moreover, since $\mathrm{Gal}(L/k)$ is a finite group, taking fixed points commutes with filtered colimits. 
Therefore, we have $$\left (\varinjlim \mathrm{Hom}((T_i)_L, G_L) \right)^{\mathrm{Gal}(L/k)} \simeq \varinjlim \left (\mathrm{Hom}((T_i)_L, G_L) \right)^{\mathrm{Gal}(L/k)} \simeq \varinjlim \mathrm{Hom}(T_i, G_L).$$ 
This proves the desired cocompactness of $G$ which finishes the proof.
\end{proof}{}

\subsection{Perfect affine stacks and perfect unipotent spectra}\label{subsection:perfect_unipotent_spectra}
In this section, we restrict our attention to the category of perfect schemes over $ k $ and introduce a notion of perfect affine stacks and perfect unipotent spectra over $ k $. First we discuss a few relevant preliminaries regarding perfect algebras, and recall their derived analogues (\cref{defn:perfect_derived_algebra}).  
The starting point here is Breen's theorem on the vanishing of higher Ext-groups of the additive group over the site of perfect $k$-schemes.  As a consequence, upon restricting to the site of perfect schemes, the recognition theorem can be refined to modules over a far less complicated $ \mathbb{E}_1$-algebra than the one we encountered in \cref{subsection:recognition_thm}.

 \begin{definition}
Let $k$ be a perfect field of characteristic $ p > 0 $. Let $\mathrm{Alg}_k^{\mathrm{perf}}$ denote the category of perfect  $k$-algebras. The inclusion of categories
\[
 \mathrm{Alg}^{\mathrm{perf}}_k  \hookrightarrow\mathrm{Alg}_k 
\]
admits a left adjoint, given by $A \mapsto A_{\mathrm{perf}}:= \varinjlim _{\varphi}A$, and a right adjoint, given by $A^\flat:= \varprojlim _{\varphi}A$, where $\varphi$ is the Frobenius endomorphism of $A$. 
\end{definition}

The notion of being perfect also extends to derived commutative rings over $ k $:

\begin{definition}\label{defn:perfect_derived_algebra}
Note that every object of $\on{DAlg}_k$ admits a Frobenius endomorphism (e.g., see \cite[Construction~2.4.1]{holeman2023}) extending the Frobenius on $ \on{Alg}_k $. 
Let $\on{DAlg}^{\perf}_k \subseteq \on{DAlg}_k$ be the full subcategory spanned by those objects for which the Frobenius map is an isomorphism of derived rings. We will refer to objects of $ \on{DAlg}^{\perf}_k $ as \emph{perfect derived $ k $-algebras}. 
\end{definition}

\begin{remark}
For a derived $\mathbb{F}_p$-algebra $B$, the Frobenius map induces the zero map on $\pi_i(B)$ for $i >0$. This implies that a perfect derived ring is always coconnective. 
\end{remark}

\begin{definition}[Perfect prestacks]
Let $\mathrm{Aff}^{\mathrm{perf}}_k$ denote the category of perfect affine schemes over $k$. We let $\mathrm{PSt}_k^{\mathrm{perf}}:= \mathrm{Fun}\left(\mathrm{Aff}_k^{\mathrm{op}}, \mathcal{S}\right)$, and call it the $ \infty $-category of \emph{perfect prestacks}.  
\end{definition}

\begin{remark}\label{pushpull!}
Note that there is a natural restriction functor $u^{*}: \mathrm{PSt}_k \to \mathrm{PSt}_k^{\mathrm{perf}}$. 
Precomposition with $(\cdot)_{\mathrm{perf}}: \mathrm{Alg}_k \to \mathrm{Alg}_k^{\mathrm{perf}}$ defines a canonical functor denoted by $u_*: \mathrm{PSt}_k^{\mathrm{perf}} \to \mathrm{PSt}_k $. By construction, $u_*$ is right adjoint to $u^*$. Similarly, precomposition with $ (\cdot)^{\flat}: \mathrm{Alg}_k \to \mathrm{Alg}_k^{\mathrm{perf}}$ defines a canonical functor denoted as $u_!: \mathrm{PSt}_k^{\mathrm{perf}} \to \mathrm{PSt}_k $. By construction, $u_!$ is left adjoint to $u^*$.
\end{remark}{}

\begin{lemma}\label{lemmatopoi}
 The functor $u_!$ is fully faithful and its essential image is given by the full subcategory of $\mathrm{PSt}_k$ spanned by $X \in \mathrm{PSt}_k$ such that the natural map $X(A^\flat) \to X(A)$ is an equivalence for every $A \in \mathrm{Alg}_k$. Moreover, for any $B \in \mathrm{Alg}_k^{\mathrm{perf}}$, we have $u_! u^*\mathrm{Spec}\,{B} \simeq \mathrm{Spec}\, B.$   
\end{lemma}

\begin{proof}
 Full faithfulness of $u_!$ follows from the observation that the natural map $\mathrm{id} \to u^*u_!$ is an equivalence. 
 The rest follows from the fact that $(\cdot)^\flat: \mathrm{Alg}_k \to \mathrm{Alg}^{\mathrm{perf}}_k$ is right adjoint to the inclusion functor.
\end{proof}{}

\begin{notation}
 For $B \in \mathrm{Alg}_k^\perf$, $u^* \mathrm{Spec}\, B$ will be simply denoted by $\mathrm{Spec}\, B \in \mathrm{Pst}_k^\perf.$   
\end{notation}

\begin{construction}\label{globalperf}
Note that left Kan extension of the global section functor $\mathrm{Aff}_{k}^{\mathrm{perf}}  \to \mathrm{DAlg}_k^{\op}$ along $\mathrm{Aff}_{k}^{\mathrm{perf}} \to \mathrm{PSt}_k^{\mathrm{perf}}$ produces a functor that we denote as 
$$R\Gamma' (\,\cdot\,, \mathcal{O}):  \mathrm{PSt}_k^{\mathrm{perf}} \to \mathrm{DAlg}_k^{\op}.$$
Note that for $X \in \mathrm{PSt}_k^{\mathrm{perf}},$ we have
$R\Gamma (u_!X, \mathcal{O}) \simeq R\Gamma' (X, \mathcal{O}).$
To see this, note that
$$R\Gamma (u_!X, \mathcal{O}) \simeq \lim_{\substack{A \in \Alg_k;\\(\mathrm{Spec}\, A \to u_!X) \in \mathrm{PSt}_k }}A \simeq \lim_{\substack{A \in \mathrm{Alg}_k;\\(\mathrm{Spec}\, A^\flat \to X) \in \mathrm{Pst}_k^\perf }}A ,$$ where the limits are taken in $\mathrm{DAlg}_k$.
However, the latter is equivalent to
$$ \lim_{\substack{A^\flat \in \mathrm{Alg}_k;\\(\mathrm{Spec}\, A^\flat \to X) \in \mathrm{Pst}_k^\perf }}A \simeq \lim_{\substack{B \in \mathrm{Alg}^\perf_k;\\(\mathrm{Spec}\, B \to X) \in \mathrm{Pst}_k^\perf }}B \simeq R\Gamma'(X, O).$$
Therefore, for $X \in \mathrm{Pst}_k^\perf$, we will simply use $R\Gamma (X,\mathcal{O})$ to denote $R\Gamma' (X, \mathcal{O}).$
\end{construction}

\begin{definition}[Perfect affine stacks]
By the Yoneda embedding, we have a functor $$ (\mathrm{DAlg}_k^{\perf})^\mathrm{op} \to \mathrm{Fun}\left(\mathrm{DAlg}_k^{\perf}, \mathcal{S}\right).$$ Composing with $\mathrm{Alg}_k^{\perf} \to \mathrm{DAlg}_k^{\perf}$, we obtain a functor
$$\mathrm{Spec}^\mathrm{pf}: (\mathrm{DAlg}_k^{\perf})^\mathrm{op} \to \mathrm{PSt}_k^\perf.$$ We define the essential image of this functor to be the category of \emph{perfect affine stacks} over $k$ and denote it by $\mathrm{AffSt}_k^{\mathrm{perf}}$.
\end{definition}

\begin{remark}\label{adjprfder}
 Note that we have an adjunction  \[
 R\Gamma(\, \cdot\, , \mathcal{O}): \mathrm{PSt}_k^{\mathrm{perf}} \leftrightarrows  (\mathrm{DAlg}^{\mathrm{perf}}_k)^{\mathrm{op}} :   \mathrm{Spec}^{\mathrm{pf}},
 \] where the left adjoint $R\Gamma(\, \cdot\, , \mathcal{O})$ is as defined in \cref{globalperf}. 
\end{remark}

\begin{remark}
By definition, for $B \in \mathrm{DAlg}_k^\perf$, we have $\mathrm{Spec}^\mathrm{pf} \, B \simeq u^* \Spec\,B$.    
\end{remark}

\begin{remark}
 Let $\mathrm{St}_k^\wedge$ (resp. $(\mathrm{St}^\perf_k)^\wedge$) denote the full subcategory of $\mathrm{PSt}_k$ (resp. $\mathrm{PSt}^\perf_k$) that satisfies hyperdescent for the fpqc topology on $\mathrm{Aff}_k$ (resp. $\mathrm{Aff}^\perf_k$). The functor $u^*$ from \cref{pushpull!} restricts to a functor again denoted as $u^*: \mathrm{St}_k^\wedge \to (\mathrm{St}^\perf_k)^\wedge.$ Note that the functor $u_*$ also restricts to a functor $u_*: (\mathrm{St}^\perf_k)^\wedge \to \mathrm{St}_k^\wedge $; this follows from the observation that if $A \to B$ is faithfully flat map of $\mathbb{F}_p$-algebras, then $A_{\perf} \to B_{\perf}$ is also faithfully flat. By \cref{pushpull!}, $u_*$ is right adjoint to $u^*$. Further, $u^*$ also admits a left adjoint, (obtained as hypersheafification of $u_!$) which we denote by $u^\sharp_!$. Similarly to \cref{lemmatopoi}, we have the following.
\end{remark}

\begin{lemma}\label{lemmatopoi2}
The functor $u^\sharp_!:  \mathrm{St}_k^\wedge \to (\mathrm{St}^\perf_k)^\wedge$ is fully faithful. Moreover, for any $B \in \mathrm{DAlg}_k^{\mathrm{perf}}$, we have $$ B \simeq R\Gamma \left(u^\sharp_! \mathrm{Spec}^{\mathrm{pf}}\,{B} , \mathcal{O}\right).$$       
\end{lemma}{}

\begin{proof}
For $S \in \mathrm{Alg}_k^{\perf}$, any faithfully flat map $S \to T$, where $T \in \mathrm{Alg}_k$, the map $S \to T_\perf$ is also faithfully flat and factors through $S \to T$. This property implies that the natural map $\id \to u^* u^\sharp_!$ is an equivalence. Therefore, the full faithfulness of $u^\sharp_!:  \mathrm{St}_k^\wedge \to (\mathrm{St}^\perf_k)^\wedge$ follows from adjunction. 

For the second part, note that there is a natural map $u^\sharp_! \Spec^{\mathrm{pf}} B \to \Spec B$ by adjunction, which induces a map $B \to R\Gamma (u^\sharp_! \Spec^{\mathrm{pf}} B,\mathcal{O})$. We will prove that this is an isomorphism. Since $\Spec B$ is an affine stack, we can write it as a colimit of a simplicial affine scheme $\Spec A_{\bullet}$ in $\mathrm{St}_k^\wedge.$ Since $u^*$ is left adjoint to $u_*$, it preserves small colimits. Therefore, $ \mathrm{Spec}\, ^\mathrm{pf} B$ is a colimit of $u^* \Spec A_\bullet \simeq \Spec^{\mathrm{pf}}(A_\bullet)_\perf $ in $(\mathrm{St}^\perf_k)^\wedge$. Since $u^\sharp_!$ is left adjoint to $u^*$, it also preserves small colimits. Therefore, $u^\sharp_! \Spec^{\mathrm{pf}} B$ is a colimit of $u^\sharp_! \Spec^{\mathrm{pf}}(A_\bullet)_\perf .$ However, for an $A \in \Alg_k^\perf$, by \cref{lemmatopoi}, $u^\sharp_! \Spec^\mathrm{pf} A \simeq u_! \Spec^\mathrm{pf} A \simeq \Spec A $. Therefore, the simplicial object $u^\sharp_! \Spec^{\mathrm{pf}}(A_\bullet)_\perf$ is isomorphic to $\Spec (A_\bullet)_\perf.$ This implies that
$$R\Gamma (u^\sharp_! \Spec^{\mathrm{pf}} B, \mathcal{O}) \simeq \mathrm{Tot} (A_\bullet)_{\perf}.$$ Since filtered colimits commute with totalizations of coconnective objects, it follows that the latter is isomorphic to 
$$(\mathrm{Tot} A_\bullet)_{\perf} \simeq B_\perf \simeq B,$$ which finishes the proof.
\end{proof}
\begin{proposition}[Embedding of perfect derived rings]\label{embedingofperf} Let $k$ be a perfect field of characteristic $p>0$. The functor 
$$\mathrm{Spec}^{\mathrm{pf}}: (\mathrm{DAlg}^{\mathrm{perf}}_k)^{\mathrm{op}} \to \mathrm{PSt}_k^\perf$$ is fully faithful. 
    
\end{proposition}{}

 \begin{proof}By virtue of adjunction from \cref{adjprfder} and \cref{globalperf}, it will be enough to prove that $B \simeq R\Gamma (u_! \Spec^\mathrm{pf} B, \mathcal{O})$ for $B \in \mathrm{DAlg}_k^{\mathrm{perf}}.$ This follows from \cref{lemmatopoi2}, as we have $B \simeq R\Gamma (u^\sharp_! \Spec^\mathrm{pf} B, \mathcal{O}) \simeq R\Gamma (u_! \Spec^\mathrm{pf} B, \mathcal{O}).$  
 \end{proof}

 \begin{corollary}
Let $\mathrm{AffSt}_k^{\perf'}$ denote the full subcategory of $\mathrm{AffSt}_k$ spanned by $X \in \mathrm{AffSt}_k$ such that $R\Gamma (X, \mathcal{O})$ is a perfect derived ring. Then the functor $u^*$ induces an equivalence
$$\mathrm{AffSt}_k^{\perf'} \simeq \mathrm{AffSt}_k^{\perf} $$
 \end{corollary}
 \begin{proof}
  The above functor is the composition of 
  $$\mathrm{AffSt}_k^{\perf'} \simeq (\mathrm{DAlg}_k^\perf)^\mathrm{op} \to \mathrm{AffSt}_k^{\perf}. $$ By \cref{embedingofperf}, the latter functor is an equivalence, which finishes the proof.
 \end{proof}

\begin{corollary}\label{cor:perfect_conn_affine_stacks_from_perfect_homotopy_groups}
    Let $ X $ be a pointed connected stack over a field $ k $. 
    Then $ X $ is a perfect affine stack if and only if each $ \pi_i^{\mathrm{U}}(X,*) $ is represented by a perfect unipotent group scheme over $ k $. 
\end{corollary}{}
\begin{proof}
    Suppose $ X $ is perfect and affine. 
    Affineness of $ X $ implies that each $ \pi_n(X,*) $ is represented by a unipotent group scheme over $ k $. 
    That the Frobenius $ F $ induces equivalences $ \pi_n(X,*) \xrightarrow{\sim} \pi_n(X,*) $ for each $ n \geq 1 $ follows from $ X $ being perfect. 
    It follows that $ \pi_n(X,*) $ is a perfect unipotent group scheme for all $ n \geq 1 $. 
    
    Conversely, suppose that $ X $ is a pointed connected stack so that $ \pi_n(X,*) $ is represented by a perfect unipotent affine group scheme for all $ n \geq 1 $. 
    By \cite[Théorème 2.4.1]{MR2244263}, $ X $ is an affine stack. 
    By assumption, the Frobenius induces equivalences on $ \pi_n(X,*) $ for all $ n \geq 1 $. 
    Since $ X $ is hypercomplete, it follows that $ X $ is a perfect affine stack. 
\end{proof}

\begin{definition}[Perfect unipotent spectra]\label{perfectunispec}Let $A$ be any perfect ring.
We define the stable $\infty$-category of perfect unipotent spectra to be the $\infty$-category  $\Sp^{\U, \perf}_{A}:=  \Sp({\on{AffSt}}^{\perf}_{A*})$ of spectrum objects; that is, it is the inverse limit 
\begin{equation*}
 \ldots \to \mathrm{AffSt}^\perf_{A*}\xrightarrow{\Omega} \mathrm{AffSt}^\perf_{A*}\to\ldots .
\end{equation*}
\end{definition}

\begin{remark}
 The natural fully faithful functor $\mathrm{AffSt}^\perf_A \to \mathrm{AffSt}_A$ produces a fully faithful functor $$\Sp^{\U, \perf}_{A} \to \Sp^{\U}_{A}.$$ Unwinding the definitions, one sees that a unipotent spectrum $E$ is a perfect unipotent spectrum if and only if $\Omega^{\infty -n}(E)$ is a perfect affine stack for each  $n \in \mathbb{Z}$. 
\end{remark}

Restricting to bounded below unipotent spectra,  which we denote by $\Sp^{\U, \perf -}_k $  we once again obtain a well-behaved $t$-structure as above.  
\begin{proposition}
Let $ k $ be a perfect field of characteristic $p>0$. 
Then there is a $t$-structure on $\Sp^{\U, \perf -}_k$ with heart given by the category of perfect unipotent commutative affine group schemes over $ k $. 
\end{proposition}

\begin{proof}
Follows from the same arguments as in \cref{tstruc}, using \cref{cor:perfect_conn_affine_stacks_from_perfect_homotopy_groups}. 
\end{proof}

\subsection{Recognition theorem for perfect unipotent spectra} \label{subsection:perfect_uni_spectra_recognition}
We will study some recognition theorems for various $\infty$-categories  of perfect unipotent spectra, similar to \cref{subsection:recognition_thm}. In addition to the techniques in \cref{subsection:recognition_thm}, we will crucially use the following result due to Breen. Before stating it, let us fix some notations for this subsection.

Let $S= \Spec R$, where $R$ is a perfect ring of characteristic $p>0$. Let $D_{fpqc} (S_\perf, \bZ/p^m)$ denote the $ \infty $-category of $D(\bZ/p^m)$-valued fpqc sheaves on $\Aff_S^\perf$. 
Let $W$ denote the $p$-typical Witt group scheme and $W_n$ denote its $n$-truncated variant. Let $W^\perf$ and $W_n^\perf$ denote their perfections. Let $\sigma$ denote the Witt vector Frobenius on $W(R)$ as well as $W_n(R).$ Let $W_n(R)_{\sigma}[F,F^{-1}]$ be the non-commutative Laurent polynomial ring subject to the relation $
Fa = \sigma(a) F.$

\begin{theorem} [{\cite[Theorem~0.1]{Breen}}]\label{breensthm}
Let $ S = \Spec R $ as above. There is a natural equivalence  
\[
R\on{Hom}_{D_{fpqc}(S_{\perf}, \bF_p)}\left(\mathbb{G}_a^\perf, \mathbb{G}_a^\perf\right) \simeq R_{\sigma}[F,F^{-1}].
\]

\end{theorem}

\begin{corollary}\label{apple}
Let $ S = \Spec R $ be as in \cref{breensthm}. 
There is a natural equivalence 
\[
R\on{Hom}_{D_{fpqc}(S_{\perf}, \bZ)}(\mathbb{G}_a^\perf, \mathbb{G}_a^\perf) \simeq R_{\sigma}[F,F^{-1}] \oplus R_{\sigma}[F,F^{-1}][-1] \,.
\] 
\end{corollary}

\begin{proof}
First, note that by adjunction, we have an equivalence 
$$R\Hom_{D_{fpqc}(S_{\perf}, \bZ)}\left(\bG_a^\perf, \bG_a^\perf\right) \simeq R\Hom_{D_{fpqc}(S_{\perf}, \bF_p)}\left(\bG_a ^\perf \otimes \bZ / p \bZ , \bG_a^\perf\right)$$ of $ R_\sigma [F,F^{-1}] $-modules. 
Thus, using the fiber sequence 
\[
0 \to \bZ \xrightarrow{ p} \bZ \to \bZ / p \bZ  \to 0  
\] and the fact that $\bG_a^\perf$ is a ring object of characteristic $p$, we have 
$$ R\Hom_{D_{fpqc}(S_{\perf}, \bF_p)}\left(\bG_a ^\perf \otimes \bZ / p \bZ  , \bG_a^\perf\right) \simeq R\Hom_{D_{fpqc}(S_{\perf}, \bF_p)}\left(\bG_a ^\perf \oplus \bG_a ^\perf[1]   , \bG_a^\perf\right).$$ 
Applying \cref{breensthm} now gives the desired computation. 
By \cite[Proposition 7.1]{bayindir}, there exists a unique $\mathbb{E}_1$-$ R_{\sigma} [F,F^{-1}] $-algebra with these homotopy groups, so this can be promoted to an equivalence of $\mathbb{E}_1$-algebras. 
\end{proof}

\begin{remark}
 Note that $\mathbb{G}_a^{\perf} \otimes_\bZ \bZ/p^m\bZ$ for $m \ge 2$ as a $\bZ/p^m \bZ $-module (induced from the right factor) is not isomorphic to $\mathbb{G}_a^\perf \oplus \mathbb{G}_a^\perf [1].$   
\end{remark}{}

The recognition theorem now takes the following form. Let $k$ be a perfect field. Note that for any $E \in \mathbb{F}_p-\mathrm{Mod}^\mathrm{\perf,U-}_k,$ the mapping spectrum denoted by $R\mathrm{Hom}\left(E, \mathbb{G}^\perf_a\right)$ can naturally be viewed as a right module over $\mathrm{End}\left(\bG_a^\perf\right) \simeq k_\sigma[F,F^{-1}] $ (see \cref{breensthm}). The assignment $E \mapsto R\mathrm{Hom}\left(E, \mathbb{G}^\perf_a\right)$ promotes to a functor 
$$M_{\bF_p}: \mathbb{F}_p-\mathrm{Mod}^\mathrm{\perf,U-}_k \to \on{RMod}_{k_{\sigma}[F,F^{-1}]}^{\op}.$$ Similarly, using \cref{apple}, we have a functor
$$M_{\bZ}: \mathbb{Z}-\mathrm{Mod}^\mathrm{\perf,U-}_k \to \on{RMod}_{k_{\sigma}[F,F^{-1}] \oplus k_{\sigma}[F,F^{-1}][-1]}^\op.$$

\begin{proposition}
The functors $M_{\bF_p}$ and $M_{\bZ}$ are fully faithful.
\end{proposition}

\begin{proof}
The proof follows in the same way as that of \cref{recog}.
\end{proof}

\subsection{Duality for perfect unipotent spectra}\label{subsection:duality_perfect_unipotent_spectra}
In \cite{Milne111}, Milne established a duality on the category of perfect unipotent group schemes over a perfect field. He then applied this to study a duality in the context of flat cohomology of surfaces, which foreshadowed several duality phenomena in the syntomic cohomology of characteristic $p$ schemes. We show that Milne's duality for perfect unipotent group schemes extends to the $\infty$-category of perfect unipotent $ \mathbb{F}_p $-modules which are bounded with respect to the $t$-structure and which satisfy the condition of being \emph{quasi-finite}.  

\begin{definition}\label{perfunispec}
    A perfect unipotent spectrum $ E $ over $ k $ is said to be of \emph{quasi-finite type} if for all $ i \in \mathbb{Z} $,  $ \pi_i E $ is representable by a quasi-finite type perfect unipotent affine group scheme over $ k $ in the sense of \cref{qftgp}.  

    We let $\mathrm{Sp}_k^{\U, \mathrm{perf}, \mathrm{ft}}$ denote the full subcategory of $ \mathrm{Sp}_k^{\U, \mathrm{perf}} $ spanned by quasi-finite type perfect unipotent spectra. 
    It follows from \cref{rmk:quasi_ft_perf_gp_sch_are_abelian_cat} that $ \mathrm{Sp}_k^{\U, \mathrm{perf}, \mathrm{ft}} $ forms a stable subcategory of $ \mathrm{Sp}_k^{\U, \mathrm{perf}} $. 
\end{definition}

\begin{remark}
We can analogously define the $\infty$-category $\bZ- \Mod^{\U, \mathrm{perf}, \mathrm{ft}}_k$  to be the subcategory of $\bZ-\Mod^{\U, \mathrm{perf}}_k$ spanned by the perfect unipotent $\bZ$-modules $E$ for which $\pi_i E$ is representable by a quasi-finite type perfect unipotent affine group scheme. Similarly, one can define $\bF_p- \Mod^{\U, \mathrm{perf}, \mathrm{ft}}_k \subset \bF_p-\Mod^{\U, \mathrm{perf}, \mathrm{ft}}_k$ in this way. 
\end{remark} 
In this setting, we have the following extension of \cref{ts}. 
\begin{remark}\label{rmk:tstruct_on_modules_in_perfect_unipotent_spectra}
    Let $\bZ- \Mod^{\U, \mathrm{perf}, \mathrm{ft},-}_k$, $ \bZ/p^n- \Mod^{\U, \mathrm{perf}, \mathrm{ft},-}_k $ resp. denote the full subcategories of $\bZ- \Mod^{\U, \mathrm{perf}, \mathrm{ft}}_k$, $ \bZ/p^n- \Mod^{\U, \mathrm{perf}, \mathrm{ft}}_k $ resp. consisting of objects which are bounded below. 
    Then these $ \infty $-categories have t-structures so that an object is connective if and only if its underlying unipotent spectrum is connective. 
\end{remark}
\begin{proposition} \label{galdescent2}
Let $E$ be a bounded perfect unipotent $ \bF_p $-module spectrum over a field $ k $. 
Let $\overline{k}$ be an algebraic closure of $k$ and write $ i^*_{\bF_p} $ for the base change along a fixed embedding $ k \to \underline{k} $ of \cref{obs:basechange}. 
Suppose that $i^*_{\bF_p}(E)$ is of quasi-finite type. Then $E$ is itself of quasi-finite type.
\end{proposition}

\begin{proof}
For this we remark that the base-change functor 
\[
i^*_{\bF_p}: {\bF_p}-\Mod(\on{St}_k) \to {\bF_p}-\Mod(\on{St}_{\overline{k}})
\]
is $t$-exact. This follows, for example  from \cite[Remark 1.3.2.8]{SAG}. Note moreover, that this $t$-structure is by construction compatible with the natural $t$-structure induced on bounded  below unipotent $\bF_p$-modules introduced in \cref{ts}. 
Hence it follows that the induced functor on (perfect) unipotent $\bF_p$-modules is $t$-exact as well. In each degree we have the cofiber sequence 
\[
\tau_{\geq (n+1)}(E) \to \tau_{\geq n}(E) \to \pi_{n}E[n]
\] 
and $E$ will be obtained in finitely many stages in this way from its homotopy sheaves. Now, for each $n$ for which $\pi_n(E) \neq 0$, t-exactness implies that 
\[
\pi_n \left(i^*_{\bF_p}E\right)[n] \simeq i^*_{\bF_p}(\pi_n(E)[n])\simeq i^*_{\bF_p}(\pi_n(E))[n] \simeq (\pi_n(E) \times_{\Spec k} \Spec \overline{k})[n].
\]
Since we assumed that $i^*_{\bF_p}(E)$ is a quasi-finite type spectrum object, its  homotopy sheaves will be unipotent group schemes over $ \overline{k} $ of quasi-finite type. In particular, $ \pi_n(E) \times_{\Spec k} \Spec \overline{k}$ is quasi-finite over $ \overline{k} $. Hence, by \cref{galdescent}, we see that $\pi_n(E)$ is itself quasi-finite over $ k $. 
It follows that $E$ is a quasi-finite type perfect unipotent $ \bF_p $-module spectrum over $ k $. 
\end{proof}

It is only after restricting to perfect unipotent $\bF_p$-modules of quasi-finite type, that we obtain a duality. First we set up some preliminaries. 

\begin{remark}
The stable $\infty$-category $\Sp(\on{St}_k)$ acquires a closed symmetric monoidal structure. Indeed, this category can be written as a tensor product 
\[
\Sp(\on{St}_k) \simeq \Sp \otimes^{L} \on{St}_k
\]
in $\on{Pr}^{L}$, the symmetric monoidal $\infty$-category of presentable $\infty$-categories. 
From this description, we see that $\Sp(\on{St}_k)$ is an $\mathbb{E}_\infty$-algebra in this category. Now for any object $A \in  \Sp(\on{St}_k)$, the functor $A \otimes - $ commutes with $\mathbb{V}$-small colimits and thus admits a right adjoint $R\underline{\Hom}(A, -)$.  
\end{remark}

\begin{remark}
In the same way, by passing to $\bF_p$-module objects, we see that \\ $\bF_p-\Mod(\on{St}_k) $ inherits a closed symmetric monoidal structure. 
\end{remark}

We will see that the desired duality functor can be defined by restricting the linear duality on $\bF_p-\Mod(\on{St}_k) $, which we forthwith denote by  $R\underline{\Hom}( -, \bZ / p)$, to the relevant subcategory of unipotent $\bF_p$-modules. First we recall some basic Ext-computations:

\begin{lemma} \label{extcomp}
There are equivalences  
\[
R \underline{ \Hom}(\bG_a, \bZ /p) \simeq \bG_a[-1] , \, \, \, \, \, R  \underline{ \Hom}(\bZ / p, \bZ / p) \simeq \bZ / p
\]
in $\Mod_{\bF_p}^{\U,\perf}$. 
\end{lemma} 

\begin{proof}
We recall an argument given by Breen in \cite{breenperf}. Applying $R \underline{\Hom}( \bG_a , -)$ to the Artin--Schreier sequence 
\[
0 \to \bZ /p \to \bG_a \xrightarrow{F -1} \bG_a,
\]
together with the vanishing of higher Ext groups of $\bG_a$, we get the exact sequence 
\[
0 \to  R[F, F^{-1}] \to  R[F, F^{-1}]  \xrightarrow{\pi} R \to 0 
\]
for any perfect $ k $-algebra $ R $, where 
\[
\pi\left( \sum^{n}_{i = -m} a_i F^{i}\right) = \sum a_i^{-p^i}.
\]
Hence, $\on{Ext}^1(\bG_a, \bZ / p) = \pi_{-1} R \underline{\Hom}(\bG_a, \bZ / p)(R) \cong R$, with all other Ext terms vanishing. This globalizes to  $R \underline{ \Hom}(\bG_a, \bZ /p) \simeq \bG_a[-1] $. 

For the second identification we use the first equivalence, together with $R \underline{\Hom}( - , \bZ / p )$ applied to the Artin-Schreier sequence  to deduce the cofiber sequence 
\[
\bG_a[-1] \xrightarrow{F-1} \bG_a[-1] \to \bZ / p \simeq R \underline{\Hom}( \bZ / p, \bZ / p ),
\]  
in the category of perfect unipotent $\bF_p$-modules. 
\end{proof}

This gives the following dualizability statement in $\bF_p$-modules in stacks.

\begin{proposition} \label{dualizability}
Let $ E $ be a perfect unipotent spectrum of quasi-finite type which is moreover bounded with respect to the $t$-structure of Remark \ref{rmk:tstruct_on_modules_in_perfect_unipotent_spectra}. 
Then $ E $ is dualizable with respect to the symmetric monoidal structure on $\Mod_{\bF_p}\left(\Sp(\on{St}_k))\right) $.  
\end{proposition}

\begin{proof}
Using the $t$-structure, any perfect unipotent spectrum of quasi-finite type satisfying the hypotheses above may be built in finitely many steps via extensions from perfect unipotent group schemes of quasi-finite type. Since dualizable objects are closed under extensions and shifts, it is enough therefore to show that a unipotent group scheme $G$ of quasi-finite type is dualizable.  
Furthermore, since any such $G$ has a finite filtration with associated graded pieces being closed subgroup schemes of $\bG_a^{\perf}$ by \cref{filtrationonuni}, we may without loss of generality assume that $G$ is such a group scheme. 

Let us form $R\underline{\Hom}(G, \bZ / p )$, the internal mapping object. We will show that the natural map
\begin{equation}\label{eq:map_to_double_dual}
G \to R\underline{\Hom}(R\underline{\Hom}(G, \bZ/ p ), \bZ / p)
\end{equation}
is an equivalence. 
Since objects of $ \Mod_{\bF_p}(\on{St}_k) $ are in particular sheaves of $\bF_p$-module spectra satisfying fpqc descent, it suffices to verify that the pullback of \cref{eq:map_to_double_dual} along the cover $\Spec \overline{k} \to \Spec k$ associated to a fixed embedding $ k \to \overline{k} $ is an equivalence. 
Since $ \overline{k} $ is algebraically closed, the only choices for $ G \times_{\Spec k} \Spec \overline{k} $ will be $\bG_a^{\perf}$ or $\bZ /p$, and these are clearly dualizable by the computation in \cref{extcomp}. 
\end{proof}

We will now use this to show that the $\infty$-category of perfect unipotent spectra of quasi-finite type which are moreover bounded with respect to the t-structure inherits a duality, which reduces to the duality of Milne mentioned in the beginning of the section.

\begin{theorem} \label{duality}
  Let $(\bF_p- \Mod^{\U, \mathrm{perf}, \mathrm{ft}}_{k})^{bd} $ denote the category of quasi-finite type perfect unipotent $\bF_p$-modules over $k$ which are bounded with respect to the $t$-structure on unipotent spectra. Then the functor 
\[
R\underline{\Hom}( - , \bZ / p): \left({\on{Mod}^{\U, bd}_{\bF_p}}\right)^{\op} \to \on{Mod}^{\U, bd}_{\bF_p}
\]
defines an autoduality of $\left(\bF_p- \Mod^{\U, \mathrm{perf}, \mathrm{ft}}_{k}\right)^{bd}$ 
\end{theorem}

\begin{proof}
Let $E$ be a perfect unipotent $\bF_p$-module satisfying the hypotheses in the statement.  We have shown in \cref{dualizability} that $E$ is a dualizable object in $\Mod_{\bF_p}( \on{St}_k)$. Let $R \underline{\Hom}(E, \bZ / p)$ denote its dual. We need to show that this is also perfect unipotent of quasi-finite type.  Let us first  assume that we are working over an algebraically closed field $\overline{k}$.

By the hypotheses on $E$, there exist integers  $-N, M$ for which 
\[
 0 \to 0 \to \tau_{\geq M}(E) \to \tau_{\geq (M-1)}(E)  \cdots \to \tau_{\geq - N}(E) \simeq E = E = \cdots
\]
such that in each degree we have cofiber sequences 
\[
\tau_{\geq (n+1)}(E) \to \tau_{\geq n}(E) \to \pi_{n}E[n]. 
\]
Hence, $E \simeq \tau_{\geq -N}(E)$ is built inductively in finitely many steps out of shifts of perfect unipotent group schemes of quasi-finite type. Applying $R\underline{\Hom}( - , \bZ /p)$ to everything in sight, we obtain analogous cofiber sequences
\[
R\underline{\Hom}( \pi_{n}E[n] , \bZ /p) \to R\underline{\Hom}( \tau_{\geq n}(E) , \bZ /p)  \to R\underline{\Hom}( \pi_{n}E[n] , \bZ /p),
\]
so that $R\underline{\Hom}(E , \bZ /p)$ is itself built up in finitely many steps out of objects of the form $R \underline{\Hom}(G, \bZ / p)$, for $G$ a perfect unipotent group scheme of quasi-finite type. 
We claim now that for $G$ of this form, that $R\underline{\Hom}(G, \bZ /p )$ is itself a perfect unipotent $\bF_p$-module of quasi-finite type. For this, recall from \cref{filtrationuni1} that every perfect unipotent group $G$ of quasi-finite type has a  (finite) composition series 
\[
\cdots G_{i +1} \subset G_{i} \subset \cdots G_1 \subset G_0 = G
\] 
where the quotients $G_{i}/ G_{i+1}$ are either $\bG_a^{\perf}$ or $\bZ/p$. So it reduces to showing the  claim for $G$ being either one of these two groups, and this will be a consequence of the computations in \cref{extcomp}.

We now let $k$ be an arbitrary perfect field of characteristic $p$, and let $E$ be as in the statement. Then $E$ is dualizable when viewed as an object of the symmetric monoidal $\infty$-category $\Mod_{\bF_p}(\on{St}_k)$. Hence there is an equivalence 
\[
\iota^*_{\bF_p} R \underline{\Hom}(E, \bZ /p) \simeq  R \underline{\Hom}( \iota^*_{\bF_p}(E), \bZ /p)
\]
in $\Mod_{\bF_p}(\on{St}_{\overline{k}})$.  By the first part of the proof, since $\iota^*_{\bF_p}(E)$ is perfect unipotent of quasi-finite type, so will  $R \underline{\Hom}( \iota^*_{\bF_p}(E), \bZ /p)$. The result now follows
from \cref{galdescent2}. 
\end{proof}

\subsection{Representability and duality for  $p$-torsion syntomic cohomology}\label{subsection:syntomic_duality}
We now apply the above duality of perfect unipotent spectra to mod $ p $ syntomic cohomology. We first recall the following result originally due to Milne in \cite[Theorem~1.9]{Milne111} (cf.~\cite[Corollary~4.5.6]{fg}).

\begin{theorem}[Milne]\label{theorem:Poincare_duality_Fgauge}
Let  $k$ be a finite field and let $X / k$ be a smooth proper $k$-scheme of dimension $d$. For each integer $i$ there is a natural isomorphism
\[
R \Gamma_{\on{Syn}}(X, \bZ /p(i)) \simeq R \Gamma_{\on{Syn}}(X, \bZ /p(d-i))^{\vee}[-2d-1]
\]
in $\on{Perf}(\bF_p)$. 
\end{theorem}

We emphasize that the above statement is specific to the case where $k$ is a finite field. If, for instance, $k= \overline{k}$ is algebraically closed, the above statement does not hold. 
Milne observed that in this case, one can obtain a more uniform duality statement, which needs to be interpreted at the level of sheaves of complexes over the étale site over the base, and not at the level of the derived category of $\bF_p$; see \cite[Theorem~2.4]{Milne111}. 

 In this section, we interpret the latter duality of Milne in its natural context of perfect unipotent spectra. 
We begin with the following proposition, concerning the relevant object to which the duality shall be applied.  

\begin{proposition} \label{representability syntomic}
Let $X$ be a smooth proper $k$-scheme of dimension $d$ and fix $ i \in \mathbb{Z}$ and $ \nu \geq 1 $. Then the functor determined by 
\[
\on{Sch}^{\on{perf}}_{k} \ni S \mapsto R \Gamma_{\on{Syn}}(X \times S, \bZ/p^\mathrm{\nu}(i)) 
\]
is represented by a quasi-finite type perfect unipotent spectrum over $k$, which we denote by $ \bZ/p^{\nu}(i)^{\uni}_{X}$.
\end{proposition}

\begin{proof}
By devissage, we immediately reduce to the case of $\nu=1$. By the mod $ p $ reduction of \cite[Proposition 4.4.2]{fg} and \cref{rmk:uni_quasi_ft_perf_gp_sch_are_abelian_cat}, it suffices to show that the assignments $$ S \mapsto \mathcal{N}^{\geq i} \phi^* R\Gamma_{\Prism}(X \times S) / p\,\,  \mathrm{and} \,\,  S\mapsto R\Gamma_{\Prism}(X \times S) / p \ $$ where $ \mathcal{N}^{\geq i} $ is the Nygaard filtration are represented by quasi-finite type perfect unipotent spectra over $ k $. 
Recall that we have equivalences $$ \mathcal{N}^{\geq 0} \phi^* R\Gamma_{\Prism}(X \times S) / p = \phi^* R\Gamma_{\Prism}(X \times S) / p \overset{\phi^{-1}}{\simeq} R\Gamma_{\Prism}(X \times S) / p \simeq R\Gamma_{\mathrm{HT}}(X \times S) $$ where the latter denotes Hodge--Tate cohomology. 
Furthermore, for each $ j $ there are fiber sequences $$ \mathcal{N}^{\geq j+1 } R\Gamma_{\Prism}(X \times S) / p \to \mathcal{N}^{\geq j} R\Gamma_{\Prism}(X \times S) / p \to \mathrm{Fil}^j_{\mathrm{conj}}R\Gamma_{\mathrm{HT}}(X \times S)/p$$ where $ \mathrm{Fil}^j_{\mathrm{conj}}R\Gamma_{\mathrm{HT}}(X \times S)$ denotes the conjugate filtration on Hodge--Tate cohomology.
Thus it suffices to show that the assignments $$ S \mapsto R\Gamma_{\mathrm{HT}}(X \times S) \,\, \mathrm{ and} \,\,  S \mapsto \mathrm{Fil}^j_{\mathrm{conj}}R\Gamma_{\mathrm{HT}}(X \times S) $$ are represented by a quasi-finite type perfect unipotent spectrum over $ k $; the result for $ \mathcal{N}^{\geq i} \phi^* R\Gamma_{\Prism}(X \times S) / p $ follows by induction on $ j $. 
Since $ S $ is perfect, $ \mathbb{L}_{S/k} \simeq 0 $ \cite[Tag 0G60]{stacks-project}, therefore $ \mathbb{L}_{X\times S/k} \simeq p_S^*\mathbb{L}_{X/k} $ where $ p_S \colon X \times S \to S $ is the canonical projection. 
Now $ R\Gamma(X \times S, \mathbb{L}^{\wedge j}_{X\times S/k})  \simeq R\Gamma(X,\Omega^j_{X/k}) \otimes_k S$. 
Since $ X $ is smooth and proper, $ R\Gamma(X,\Omega^j_{X/k}) $ is a perfect complex of $ k $-vector spaces and the assignment $ S \mapsto R\Gamma(X \times S, \mathbb{L}^{\wedge j}_{X\times S/k}) $ is represented by a finite product of $ \mathbb{G}_a^{\mathrm{perf}}[m]$ for $m \in \mathbb{Z}$. 
Since $ \mathrm{Fil}^j_{\mathrm{conj}}R\Gamma_{\mathrm{HT}}(X \times S) $ has a finite filtration with associated graded $R\Gamma (X \times S, \wedge^r \mathbb{L}_{X \times S/k}[-r]) $ for $ 0 \leq r \leq j $ and $\mathrm{Fil}^j_{\mathrm{conj}}R\Gamma_{\mathrm{HT}}(X \times S) \simeq R\Gamma_{\mathrm{HT}}(X \times S)$ for large enough $j$, the desired result follows. 
\end{proof}

One has the following more general result when $X$ is not necessarily assumed to be smooth but only proper lci. However, the resulting unipotent spectrum is not necessarily quasi-finite in this generality.

\begin{proposition} 
Let $X$ be a proper lci $k$-scheme of dimension $d$ and $ i \in \mathbb{Z}$. Then the assignment 
\[
S \mapsto R \Gamma_{\on{Syn}}(X \times S, \bZ/p^\nu(i)) 
\]
ranging over $S \in \on{Sch}^{\on{perf}}_{k}$ is represented by a perfect unipotent spectrum over $k$, which we denote by $ \bZ/p^\nu(i)^{\uni}_{X}$.    
\end{proposition}{}
\begin{proof}
  The proof is similar to \cref{representability syntomic}, but requires some modifications. Once again, by devissage, we immediately reduce to the case of $\nu=1$. By \cite[Proposition 7.4.6]{BhaLur},   it will suffice to show that the assignment $$ S \mapsto \mathcal{N}^{\geq i} \phi^* \widehat{R\Gamma_{\Prism}(X \times S)} / p ,$$ is represented by a perfect unipotent spectrum, where the latter denotes Nygaard-completed variant of crystalline cohomology. By Nygaard completeness, we have
  $$ \mathcal{N}^{\geq i} \phi^* \widehat{R\Gamma_{\Prism}(X \times S)} \simeq \varprojlim_s \mathcal{N}^{\geq i} \phi^* \widehat{R\Gamma_{\Prism}(X \times S)}/ \mathcal{N}^{\geq i+s} \phi^* \widehat{R\Gamma_{\Prism}(X \times S)} .$$
  
  Since the category of unipotent spectra is stable under limits by \cref{beach1}, by considering the graded pieces of the Nygaard filtration, it suffices to show that the assignment $$ S \mapsto \mathrm{Fil}^j_{\mathrm{conj}}R\Gamma_{\mathrm{HT}}(X \times S) $$ is representable by a perfect unipotent spectrum over $ k $ for each $ j $. For this, by passing to the graded pieces of the conjugate filtration, it suffices to show that the assignment $$ S \mapsto R\Gamma(X \times S, \mathbb{L}^j_{X\times S/k}) \simeq R\Gamma (X, {\mathbb{L}^j_{X/k}} ) \otimes_k S $$ is representable by a perfect unipotent spectrum. However, since $\mathbb{L}_{X/k}$ is a perfect complex with Tor amplitude in homological degrees $[0,1]$ and $X$ is proper, the above functor is isomorphic to a finite product of $\mathbb{G}_a^{\mathrm{perf}}[m]$ for $m \in \mathbb{Z}$. This finishes the proof. 
\end{proof}{}

We extract the following corollary, recovering a result of Illusie-Raynaud, cf. the discussion after \cite[Lemme 3.2.2]{IR} and extending it to the lci case. 

\begin{corollary}
    Let $ k $ be a field and let $ X $ be a proper lci scheme over $ k $. 
    Let $\bZ /p^\nu(i)^{\uni}_{X}$ be as above. Then for each $n, i \geq 0$, the homotopy sheaves $\pi_{n}(\bZ /p^\nu(i)^{\uni}_{X})$, will be representable by unipotent group schemes over $k$; these will be of quasi-finite type if $X$ is smooth.     
\end{corollary}

Milne's duality statement  in the smooth case implies the following statement.

\begin{theorem}  \label{thm duality syntomic mod p}
Let $ X $ be a smooth proper $ k $-scheme of dimension $ d $ and $i \in \mathbb{Z}$. Then there is a natural equivalence 
\[
(\bZ /p(i)^{\uni}_{X})^{\vee} \simeq \left(\bZ /p(d-i)^{\uni}_{X} \right)[2d]
\]
of perfect unipotent $ \bF_p $-module spectra of quasi-finite type over $k$, where we regard $ \bZ /p(i)^{\uni}_{X} $ and $ \bZ /p(d-i)^{\uni}_{X} $ as perfect unipotent spectra by \cref{representability syntomic} and $ (-)^\vee $ denotes the linear duality of \cref{duality}. 
\end{theorem}

\begin{proof} Below we assume $i \ge 0$, since both sides above vanish for $i <0$.
Recall that there exist pairings 
\[ 
\bZ / p (m) \otimes \bZ / p (n) \to \bZ / p (m+n) 
\]
of sheaves on the quasi-syntomic site of $k$. This can  be seen  from the equivalence $\cO_{\on{Syn}} \{ m\} \otimes  \cO_{\on{Syn}} \{ n\} \simeq \cO_{\on{Syn}} \{ m +n \}$ of invertible  objects in $F$-gauges. This gives rise to a natural pairing 
\[
\bZ /p(m)^{\uni}_{X} \otimes  \bZ /p(n)^{\uni}_{X} \to \bZ /p(m+n)^{\uni}_{X}
\]   
of objects in $\Mod_{\bF_p}(\mathrm{St}^{\perf}_k)$.
Now if $\pi: X \to \Spec(k)$ is a proper smooth morphism of relative dimension $d$, there is a trace map\footnote{Milne's paper has a typo where the map is off by a shift.}
\[
\bZ /p(d)^{\uni}_{X} \to \bZ /p [-2d].
 \]
This is a consequence of \cite[Theorem 2.4]{Milne111}; it can be alternatively viewed as a consequence of Poincare duality (cf. \cite{MR4804676, fg}) for the $F$-gauge $\cH_{\on{Syn}}(X)$ and the resulting map 
\[
\cH_{\on{Syn}}(X) \to \cO_{\on{Syn}} \{-d\}[-2d]
\]
in the category of $F$-gauges over $k$. As a consequence of \cite[Theorem 2.4]{Milne111}, this gives rise to a perfect pairing, for each $i$,  
\[
\bZ / p^{\uni} _{X}(i) \otimes  \bZ / p^{\uni} _{X}(d-i) \to \bZ / p^{\uni} _{X}(d) \to \bZ/p[-2d],
\]
where the latter denotes the constant sheaf on $\St^{\perf}_k$ with value $\bZ/p[-2d]$. 
Hence, we obtain the desired equivalence 
\[
\bZ / p^{\uni} _{X}(d-i)  \simeq R \underline{\Hom}( \bZ / p^{\uni} _{X}(i), \bZ /p)[-2d] \,. \qedhere
\] 
\end{proof}

\begin{remark}
We remark that the objects $\bZ / p^{\uni} _{X}(i)$ agree with the objects $\pi_* \nu (i)[-i]$  in the notation Milne uses in \cite{Milne111}, when viewed as objects in the same category. These  correspond to certain étale sheaves $\nu(i)$ on the perfect site over $X$, pushed forward to the perfect site over $\Spec k$. By \cref{representability syntomic}, these are in fact fpqc sheaves on the perfect site, and are moreover unipotent $\bF_p$-modules in our terminology. 
\end{remark}

\subsection{Duality for  $p$-power torsion syntomic cohomology} \label{subsection duality in Zmod}
In this section, we exhibit a duality on $ \bZ $-modules in quasi-finite type perfect unipotent spectra; our construction is analogous to Pontryagin duality. 
Using this duality functor, we show that Poincaré duality for syntomic cohomology with mod $ p^n $ coefficients lifts to an equivalence in perfect unipotent $ \bZ $-modules. 

\begin{proposition} \label{prop Q/Z} 
Let $F$ be a $\bZ/ p^n$-module sheaf on the perfect site over $k$. Then there is an equivalence 
\[
R \Hom_{\mathbb{Z}}(F, \mathbb{Q}_p / \bZ_p) \simeq R \Hom_{\mathbb{Z}/p^n\mathbb{Z}}(F, \bZ /p^n)
\] 
of $\bZ$-modules. 
\end{proposition}

\begin{proof}
Recall that there exists a right adjoint $u^!: \Mod_{\bZ/p^n} \to \Mod_{\bZ}$ to the forgetful functor given by $M \mapsto R \Hom_{\bZ}(\bZ/p^n, M)$.  Letting $M = \bQ_p/ \bZ_p$, and $F$ be as in the statement, and using the adjunction, we get that 
\[
R \Hom_{\bZ}(F,  \bQ_p / \bZ_p) \simeq R \Hom_{\bZ /p}(F,  u^!(\bQ_p / \bZ_p)).
\]
Now we identify $u^!(\bQ_p / \bZ_p)$ with $\bZ /p^n$. For this, note that. 
\[
u^!(\bQ_p / \bZ_p) \simeq  R \Hom_{\bZ}(\bZ/p^n, \bQ_p / \bZ_p) \simeq \colim_m (R \Hom_{\bZ}(\bZ/p^n, \bZ/p^m)) 
\] 
since $\bZ/p^n$ is compact as a $\bZ$-module. The above colimit stabilizes to $\bZ/p^n$, giving the desired identification. 
\end{proof}

Now, for any perfect  unipotent $\bZ$-module  of quasi-finite type $E$, we set  
\[
E^\vee =  R \underline{\Hom}(E, \bQ_p / \bZ_p)
\]
We have the following refinement of Milne's duality \cite{Milne111} in the general $\bZ$-linear setting. 

\begin{theorem} \label{duality Q/Z}
Let $(\bZ-\Mod_{k}^{\U, \on{perf}, \on{ft}})$ denote the $\infty$-category of quasi-finite type perfect unipotent $\bZ$-modules over $k$ which are bounded with respect to the induced $t$-structure. Then the functor 
\[
(-)^{\vee} = R \underline{\Hom}_{\bZ}(- , \bQ_p/ \bZ_p) : \Mod_{\bZ}(\Sp(\St_k)) \to  \Mod_{\bZ}(\Sp(\St_k))^{\op}
\]
restricts to an autoduality on $(\bZ-\Mod_{k}^{\U, \on{perf}, \on{ft}})$. 
\end{theorem}

\begin{proof}
This will essentially follow formally from Theorem \ref{duality}, which holds over $\bF_p$. Let $E$ be an arbitrary perfect unipotent $\bZ$-module of quasi-finite type satisfying the conditions of the statement. We need to show that $E^{\vee}$ lands in this category, and that $(E^{\vee})^{\vee} \simeq E$.  
As in the proof of Theorem \ref{duality}, we take the Postnikov tower of $E$, which allows us to reduce to the case $G = \pi_n(E)$, for $G$ a perfect unipotent group scheme of finite type. Now, we use the fact, cf. Proposition \ref{filtrationonuni}, that $G$ has a finite filtration where the graded pieces are closed unipotent perfect subgroup schemes of $\bG_a^{\perf}$. In particular, these associated graded pieces are naturally $\bZ /p$-modules. Let $\on{gr}^i(G)$ denote one of these quotients. Then, via the previous proposition, we have equivalences 
\[
R \underline{\Hom}_{\bZ}(\on{gr}^i(G), \bQ_p / \bZ_p) \simeq  R \underline{\Hom}_{\bZ/p}(\on{gr}^i(G), \bZ / p) 
\]
of $\bZ$-module objects. In this case  $(\on{gr}^i(G))^\vee $ is itself perfect unipotent of quasi-finite type and 
\[
 ((\on{gr}^i(G))^\vee)^\vee  \simeq \on{gr}^i(G)
 \]
 by Theorem \ref{duality}. By dévissage,  we deduce an equivalence  $((G)^\vee)^\vee \simeq  G$ of unipotent spectra. 
\end{proof}

We now conclude by showing that Theorem \ref{thm duality syntomic mod p} holds more generally with mod $p^n$ coefficients. 

\begin{theorem} \label{duality syntomic p^n}

Let $ X $ be a smooth proper $ k $-scheme of dimension $ d $ and $i \in \mathbb{Z}$. Then there is a natural equivalence 
\[
(\bZ /p^n(i)^{\uni}_{X})^\vee \simeq (\bZ /p^n(d-i)^{\uni}_{X} )[2d]
\]
of perfect unipotent $\bZ$-module spectra of quasi-finite type over $k$, where we regard $ \bZ /p(i)^{\uni}_{X} $ and $ \bZ /p(d-i)^{\uni}_{X} $ as perfect unipotent spectra by \cref{representability syntomic} and $ (-)^\vee $ denotes $R \underline{\Hom}_{\bZ}(- , \bQ_p / \bZ_p)$ of \cref{duality Q/Z}. 

\end{theorem}

\begin{proof}
We will again take $ i \geq 0 $, since both sides vanish for $ i < 0 $. 
Let $\phi: S \to \Spec k$ be an arbitrary perfect affine scheme over $\Spec k$. Since $X$ is smooth and proper over $\Spec k $, there is a perfect pairing  
\[
\cH_{\on{syn}}(X)\{i\} \otimes \cH_{\on{syn}}(X)\{d-i\}[2d] \to   \cO_{k^{\on{syn}}}.
\]
 of $F$-Gauges over $k$ \cite{MR4804676}. Applying the symmetric monoidal functor 
 \[
 (\phi^{\on{syn}})^* : F-\on{Gauge}(k) \to F-\on{Gauge}(S)
 \]
 gives a perfect pairing in $F$-Gauge$(S)$. We set $\cH_{\on{syn}}(X \times S) :=   (\phi^{\on{syn}})^*(\cH_{\on{syn}}(X))$. Reducing modulo $p^n$ for each $n$ gives a perfect pairing 
\[
\cH_{\on{syn}}(X \times S)\{i\} /p^n\otimes \cH_{\on{syn}}/ p^n (X \times S)\{d-i\}[2d] \to   \cO_{S^{\on{syn}}}/ p^n.
\]

Finally, consider the cohomology functor 
\[
R \Gamma( S^{\on{syn}}, -): F-\on{Gauge}(S) \to \Mod_{\bZ/p^n}.
\] 
Since the cohomology functor is lax monoidal, the aforementioned perfect pairing induces a map
\begin{equation}\label{eq:syntomic_pairing_after_basechange}
R\Gamma_{\on{syn}}(X \times S , \bZ/ p^n(i)) \otimes R\Gamma_{\on{syn}}(X \times S , \bZ/ p^n(d-i))[2d] \to R \Gamma_{\on{syn}}(S, \bZ/p^n) \,.
\end{equation}
Note that there is an equivalence $R \Gamma_{\on{syn}}(S, \bZ/p^n)  \simeq R \Gamma_{\on{\acute{e}t}}(S, \bZ/p^n)$ with étale cohomology (which also agrees with fppf cohomology as well).

Observe that for a fixed perfect affine scheme $S$, we have described a functor $X \mapsto R \Gamma_{\on{syn}}(X \times S, \bZ /p^n(i))$ for each $ i \geq 0 $. 
Letting $S$ vary over perfect affine schemes over $ k $, we regard $ R \Gamma_{\on{syn}}(X \times (-), \bZ /p^n(i)) $ for $ i \geq 0 $ as objects of $ \mathrm{PSh}\left(\Aff^\perf_k\right) $.  
These are moreover representable by perfect unipotent $\bZ/p^n$-modules via Proposition \ref{representability syntomic}.  In particular, functoriality of \cref{eq:syntomic_pairing_after_basechange} in $ S $ induces maps  
\[
\bZ /p^{n}(i)_{X}^{\uni} \otimes   \bZ /p^{n}(d-i)_{X}^{\uni}[2d] \to  \bZ/p^n
\]
of unipotent $\bZ/p^n$-modules for every $n \geq 1$ and $ i \geq 0 $, which we compose with $\bZ/ p^n \to \bQ_p / \bZ_p$ to obtain the pairing 
\[ 
\bZ /p^{n}(i)_{X}^{\uni} \otimes   \bZ /p^{n}(d-i)_{X}^{\uni}[2d] \to  \bQ_p / \bZ_p \,.
\]
We show that this pairing is perfect, namely that the adjoint 
\[
\alpha: \bZ /p^{n}(d-i)_{X}^{\uni}[2d] \to R\underline{\Hom}_{\bZ}(\bZ /p^{n}(i)_{X}^{\uni}, \bQ_p / \bZ_p) 
\]
is an equivalence. For this, recall first that 
\[
R\underline{\Hom}_{\bZ}\left( \bZ /p^{n}(i)_{X}^{\uni}, \bQ_p / \bZ_p\right) \simeq  R\underline{\Hom}_{\bZ/p^n}\left( \bZ /p^{n}(i)_{X}^{\uni},  \bZ / p^n\right),
\]
by Proposition \ref{prop Q/Z}.  By a variant of Proposition \ref{dualizability},  $\bZ /p^{n}(i)_{X}^{\uni}$ will be a dualizable object in fpqc sheaves of $\bZ/p^n$-modules on the perfect site. Hence there will be an equivalence 
\[
R\underline{\Hom}_{\bZ/p^n}\left( \bZ /p^{n}(i)_{X}^{\uni},  \bZ / p^n\right) \otimes_{\bZ / p^n} \bF_p \simeq R\underline{\Hom}_{\bF_p}\left( \bZ /p(i)_{X}^{\uni},  \bF_p\right) \,.
\] 
Via these equivalences, we may identify the mod $p$ reduction of $\alpha$ with the map 
\[
\bF_p(d-i)_{X}^{\uni}[2d] \to R \underline{\Hom}\left(\bF_p(i)_{X}^{\uni}, \bF_p\right)    
\] 
of \cref{thm duality syntomic mod p}; in particular, it is an equivalence. Hence the map $\alpha$ is an equivalence as well by the derived Nakayama's lemma, since $\alpha$ is a map between $\bZ/ p^n$-modules, and thus a map of $p$-complete objects.
\end{proof}

\subsection{Duality for $p$-complete syntomic cohomology} \label{finalsection}
In this final section, we describe how to extend the $\bZ/p^n$-linear dualities described above to $p$-complete perfect unipotent spectra. 
We then show that in the context of \cref{representability syntomic}, Poincaré duality for syntomic cohomology promotes to an equivalence of perfect unipotent spectra. 
For this we need a notion of $p$-complete unipotent $\bZ_p$-module, which we define below. 

\begin{definition}
Let $\mathcal{D}(\bZ)^{\wedge}_{p}$ denote the stable $\infty$-category of $p$-complete $\bZ_p$-modules. For a $\mathcal{D}(\bZ)$-module $\mathcal{C}$ in $\on{Pr}^{\on{L}}$, its $p$-completion is defined as  $\mathcal{C}^{\wedge}_{p} := \mathcal{C} \otimes_{\mathcal{D}(\bZ)}  \mathcal{D}(\bZ)^{\wedge}_{p} $
\end{definition}

\begin{proposition}
There is an equivalence 
\[
(\mathbb{Z}\mathrm{-Mod}^{\mathrm{U}}_k)^{\wedge}_p \simeq  \lim (\mathbb{Z}/p \mathrm{-Mod}^{\mathrm{U}}_k \leftarrow \mathbb{Z} / p^2 \mathrm{-Mod}^{\mathrm{U}}_k  \leftarrow \cdots) 
\]
\end{proposition}

\begin{proof}
For this we note that there is an obvious equivalence of presheaf categories given by 
\[
\Fun(\on{Aff}^{\op}_{k}, \mathcal{D}(\bZ)^{\wedge}_{p}) \simeq \lim( \Fun(\on{Aff}^{\op}_{k}, \mathcal{D}(\bZ / p)) \leftarrow \Fun(\on{Aff}^{\op}_{k}, \mathcal{D}(\bZ / p^2)) \leftarrow \cdots) 
\]
Now, note that we may view the natural map  
\[
(\mathbb{Z}\mathrm{-Mod}^{\mathrm{U}}_k)^{\wedge}_p \to  \lim( \mathbb{Z}/p \mathrm{-Mod}^{\mathrm{U}}_k \leftarrow \mathbb{Z} / p^2 \mathrm{-Mod}^{\mathrm{U}}_k  \leftarrow \cdots) 
\]
as a retract of this equivalence, via \cref{unicomofspectra}. 
We now conclude the equivalence in the statement,  using the fact that equivalences are stable under retracts. 
\end{proof}

\begin{construction}\label{cons:duality_and_reduction}
 For each $n$, let $\cC_{n}^{\ft}$ denote the full subcategory of $\bZ/p^n$-modules in perfect unipotent spectra  spanned by the quasi-finite type objects which are bounded with respect to the t-structure of \cref{rmk:tstruct_on_modules_in_perfect_unipotent_spectra}. By Theorem \ref{duality Q/Z}, the functors 
 \[
 \mathbb{D}_n(-) = R \underline{\Hom}_{\bZ /p^n}(-, \bZ/p^n)
 : (\Mod_{\bZ / p^n}(\Sp(\on{St}_k)))^{\op} \to \Mod_{\bZ / p^n}(\Sp(\on{St}_k))
 \]
 restrict to an equivalence $\mathbb{D}_n: \cC_n^{\ft} \simeq (\cC^{\ft})^{\op}$.  Note that if $E \in \cC_{n}^{\ft}$, then $E \otimes \bZ  /p^{n-1} \in \cC_{n-1}^{\ft}$. In other words, we have a commutative diagram 
\begin{equation}\label{diagram:mod_p_n-1}
 \begin{tikzcd}
\cC_{n}^{\ft} \arrow[rr, ""]   \arrow[d, ""]                                       &  &  \cC_{n-1}^{\ft}   \arrow[d, ""']                                             \\
\Mod_{\bZ / p^n}(\Sp(\on{St}_k)) \arrow[rr, "\otimes_{\bZ /p^n} \bZ / p^{n-1}"'] &  & \Mod_{\bZ / p^{n-1}}(\Sp(\on{St}_k)),
\end{tikzcd} 
\end{equation}
where the vertical arrows are the inclusions of full subcategories.    
\end{construction}

\begin{definition}
  We set  
\[
 \cC^{\on{pro-}\ft} := \lim  \cC_n^{\ft}
 \]
  to be the limit along the horizontal maps in the diagram (\ref{diagram:mod_p_n-1}). Alternatively, $\cC^{\ft}$ can be described as the full subcategory of perfect unipotent $p$-complete modules for $E$ for which $E \otimes \bZ /p^n \in \cC_n^{\ft}$ for every $n > 0$.    

  We remark that $\bZ-\Mod^{\U, \ft}$, the $\infty$-category of quasi-finite type perfect unipotent spectra sits as a full subcategory of $\cC^{\on{pro-}\ft} := \lim  \cC_n^{\ft}$ as defined here. 
\end{definition}

Taking the limit  of the following diagram 
\begin{equation}
 \begin{tikzcd}
\cdots \to \cC_{n+1}^{\ft} \arrow[rr, ""]   \arrow[d, " \simeq"]         &  &   \cC_{n}^{\ft} \arrow[rr, ""]   \arrow[d, "\simeq"]                                       &  &  \cC_{n-1}^{\ft}   \arrow[d, "\simeq"']        \to \cdots                                     \\
\cdots \to (\cC_{n+1}^{\ft})^\op \arrow[rr, "\otimes_{\bZ /p^n} \bZ / p^{n-1}"'] &  & (\cC_{n}^{\ft})^{\op}  \arrow[rr, ""]     &  &  (\cC_{n-1}^{\ft})^{\op}  \to \cdots ,
\end{tikzcd} 
\end{equation}
where the vertical arrows are the dualities over $\bZ/p^n$ of \cref{cons:duality_and_reduction}, we obtain an  equivalence which we denote by 
\[
\mathbb{D}: \cC^{\on{pro-}\ft} \to (\cC^{\on{pro-}\ft})^{\op}.
\]
The following proposition summarizes the above discussion.

\begin{proposition}
Let $\cC^{\on{pro-}\ft}$ denote the full subcategory of $p$-complete perfect unipotent $\bZ$-modules spanned by those objects $E$ for which $E \otimes \bZ /p^n$ is a perfect unipotent $\bZ/p^n$-module of quasi-finite type for each $ n \geq 1 $. Then there exists an involutive equivalence  
\[
\mathbb{D}: \cC^{\on{pro-}\ft} \to (\cC^{\on{pro-}\ft})^{\op},
\] 
which is compatible with the dualities of Theorem \ref{duality Q/Z}. 

\end{proposition} 

We conclude with the following description of the behavior of syntomic cohomology as a $p$-complete unipotent $\bZ$-module.  Let $k$ be a perfect field of characteristic $p$, and let 
$X$ be a smooth and proper scheme over $\Spec k$. We let $\bZ_p(i)^{\uni}_{X}$   be the presheaf on the perfect site which, for every perfect scheme $S$, sends  
\[
S \mapsto R \Gamma_{\on{Syn}}(X \times S, \bZ_p(i)) \in D(\bZ)^{\wedge}_p
\]
By \cref{representability syntomic} the above functor is representable by a $p$-complete perfect unipotent $\mathbb{Z}$-module, and is an object of $\cC^{\on{pro-}\ft}$.
\begin{theorem}
Let $\bZ_p(i)^{\uni}_{X}$ be as above. Then there is an equivalence 
\[
\mathbb{D}(\bZ_p(i)^{\uni}_{X}) \simeq  \bZ_p(d-i)_{X}^{\uni}[2d]  
\]
of $p$-complete unipotent $\bZ$-modules.
\end{theorem}

\begin{proof}
The proof will be a consequence of Theorems \ref{thm duality syntomic mod p} and \ref{duality syntomic p^n}. Indeed, for each $n$, we have an equivalence 
\[
\bZ_p(d-i)_{X}^{\uni}[2d]  \otimes \bZ/p^n \simeq \bZ /{p^n} (d-i)_{X}^{\uni}[2d]  \simeq (\bZ/p^n(i)_{X}^{\uni})^\vee.
\]
Here, the last term on the right denotes the $\bZ/p^n$-linear dual of $\bZ/p^n(i)_{X}^{\uni}$. By construction, these equivalences are all compatible with extension along scalars $\bZ/p^n \to \bZ/p^{n-1}$. Taking the limit of these equivalences ranging over all $n$ produces an equivalence 
\[
\bZ_p(d-i)_{X}^{\uni}[2d] \simeq \mathbb{D}(\bZ_p(i)^{\uni}_{X}),
\]
as desired. 
\end{proof}
\newpage

\bibliographystyle{amsalpha}
\bibliography{main}
\end{document}